\documentclass[a4paper,11pt]{article}

\usepackage[utf8]{inputenc} 
\usepackage[T1]{fontenc}
\usepackage[english]{babel}
\usepackage{amsmath, amsfonts, amsthm,amssymb,here,dsfont, hyperref}
\usepackage{mathtools}
\usepackage{graphicx,color,subfigure,multirow, here}
\usepackage{natbib}
\usepackage{enumitem}
\usepackage{comment}
\usepackage{amsfonts}
\usepackage{float}

\newtheorem{theo}{Theorem}[section]

\newtheorem{prop}[theo]{Proposition}

\newtheorem{coro}[theo]{Corollary}
\newtheorem{lemma}[theo]{Lemma}

\newtheorem{assumption}[theo]{Assumption}

\newcommand{\w}{\widehat}

\title{Nonparametric estimation of the diffusion coefficient of an ergodic diffusion process on non-compact supports under Osgood's conditions}
\author{Eddy-Michel Ella-Mintsa
\\
\small{Institut de Recherche Technologique, CENAREST, Gabon}}

\begin{document}

\maketitle

\begin{abstract}
    In this paper, we study the nonparametric estimation of the squared diffusion coefficient of an ergodic diffusion process on non-compact supports, when the two coefficients of the stochastic differential equation belong to a space larger than the space of H\"older continuous functions. The estimators are constructed by projection onto finite-dimensional spaces, thereby avoiding truncation of the dimension. We establish risk bounds of non adaptive estimators and explicit rates of convergence for bounded and unbounded diffusion coefficients. A model selection procedure is performed, followed by the derivation of risk bounds using Talagrand's inequality. The theoretical results are completed with a numerical study over simulated data.
\end{abstract}

\vspace*{0.25cm} 



\section{Introduction}\label{sec:intro}

Consider $l \in \mathbb{R}\cup\{-\infty\}$ and a diffusion process $X = (X_t)_{t \geq 0}$ taking values in the interval $(l, +\infty)$ solution of the following stochastic differential equation:
\begin{equation}\label{eq:model}
    dX_t = b(X_t)dt + \sigma(X_t)dW_t, ~~ t \geq 0, 
\end{equation}
where $W = (W_t)_{t \geq 0}$ is the standard Brownian motion and $X_0$ is independent of $W$. The coefficients $b$ and $\sigma$ are assumed to be unknown measurable functions on the interval $(l, +\infty)$, and the process $X$ is defined on a filtered probability space $\left(\Omega, \mathcal{F}, (\mathcal{F}_t^X)_{t \geq 0},  \mathbb{P}\right)$ where $(\mathcal{F}_t^X)_{t \geq 0}$ is its natural filtration. This paper addresses the construction of nonparametric estimators of $\sigma^2$ on non-compact intervals $I = (l, +\infty)$, where the unknown coefficients $b$ and $\sigma$ belong to a space larger than the space of $\alpha$-H\"older continuous functions with $\alpha \in [1/2, 1]$. The estimators are built from a discrete observation $\{X_{k\Delta_n}, ~ k = 0, \ldots, n\}$ of the process $X$, where $\Delta_n$ is the time step such that $\Delta_n \rightarrow 0$ and $n\Delta_n \rightarrow \infty$ as $n \rightarrow \infty$ (high frequency data). We consider projection estimators of $\sigma^2$ avoiding truncation of the dimension of the approximation spaces. This setting requires the research of stable regimes on the dynamics of $X$ for which consistent estimators can be built with explicit rates of convergence, the diffusion coefficient $\sigma$ being bounded or not. These regimes are the result of a balance between the dissipative effect induced by the drift coefficient and the diffusion regime characterized by the diffusion coefficient. 

\subsection{Motivation and related works}

Statistical estimation of the coefficients of a diffusion process is crucial for solving problems in fields such as quantitative finance and economics for the valuation of financial assets, interest rates or value-at-risk (see, \textit{e.g.}, \cite{cox1985theory}, \cite{vasicek1977equilibrium}, \cite{ait1996testing} \cite{ait2007maximum}), biology or ecology to describe population dynamics or population genetics (see, \textit{e.g.}, \cite{dennis1991estimation}, \cite{mao2003asymptotic}, \cite{crow2017introduction},), epidemiology to study disease spread (see, \textit{e.g.}, \cite{Allen2017APO}, \cite{dale2018estimating}, \cite{cai2019stochastic}), or physics for the study of the dynamics of self-propelled particles (see, \textit{e.g.}, \cite{romanczuk2012active}). 

The diffusion coefficient estimation from high frequency observations of an ergodic process is a common topic in the literature. The estimators are built from either kernel approach (see \cite{fan2003reexamination}, \cite{bandi2003fully}, \cite{kutoyants2004statistical}) or by minimizing a least-squares contrast (see \cite{comte2007penalized}, \cite{schmisser2019non}). Focusing on the estimation of the diffusion coefficient on a non-compact domain, the only estimators proposed are constructed using the kernel method (see \cite{bandi2003fully}). The construction of these estimators is made possible by the use of the local time of the process. However, they do not adapt as easily as least squares contrast estimators and are highly dependent on the choice of the local bandwidth. The goal of the present paper is to investigate projection estimation of $\sigma^2$ on non-compact supports. This approach consists in projecting the unknown function $f$ onto a finite-dimensional space spanned by a basis of functions defined on a non-compact interval of $\mathbb{R}$ (see, \textit{e.g.},  \citeauthor{comte2020nonparametric}(\citeyear{comte2020nonparametric}, \citeyear{comte2021drift})), or a basis of compactly supported functions, the estimation being performed on the growing compact interval $[-A_n, A_n]$ with $A_n > 0$ and $A_n \rightarrow \infty$ as $n \rightarrow \infty$, with the supports of the basis functions dilated to $[-A_n, A_n]$, and the estimation risk on $\mathbb{R}\setminus [-A_n, A_n]$ proven to be negligible with respect to the expected rate of convergence (see \cite{denis2024nonparametric}, \citeauthor{ella2024nonparametric}(\citeyear{ella2024nonparametric}, \citeyear{ella2025minimax})). However, a drawback arises when the estimation interval is unbounded; one must control the inverse of the minimum eigenvalue $\|\Psi_m^{-1}\|_{\mathrm{op}}$ of the Gram matrix $\Psi_m$ of the chosen basis of dimension $m$. This control is crucial for deriving the risk bounds of both non-adaptive and adaptive estimators. In the current literature, the required condition on $\|\Psi_m^{-1}\|_{\mathrm{op}}$ is often assumed, leading to estimators constructed via dimension truncation of the approximation spaces (see \citeauthor{comte2020nonparametric}(\citeyear{comte2020nonparametric}, \citeyear{comte2021drift})). More recently, it has been proven under restrictive assumptions on the diffusion model (see \cite{denis2024nonparametric}, \cite{ella2025minimax}), where the approximation of the underlying space is based on compactly supported bases as described above. In the present paper, we establish an upper bound for $\|\Psi_m^{-1}\|_{\mathrm{op}}$ in the case of the Hermite and Laguerre bases by exploiting their respective analytic properties (see \cite{askey1965mean} or \cite{thangavelu1993lectures}). 

\subsection{Main results}

To our knowledge, only \cite{comte2007penalized} studies least squares contrast estimators of the diffusion coefficient of a continuous and ergodic process on a compact interval. The authors derived an optimal rate of order $n^{-s/(2s+1)}$ over a Besov space of smoothness parameter $s > 0$. We extend the study to the estimation of $\sigma^2$ on non-compact supports under weak Osgood-type regularity assumptions on the coefficients, we establish the existence, uniqueness and ergodicity of the underlying diffusion process. The main results of the paper are as follows.
\begin{enumerate}
    \item Establishment of upper bounds of $\mathcal{L}(m)\left\|\Psi_m^{-1}\right\|_{\mathrm{op}}$ where $\Psi_m$ is the Gram matrix of the Hermite basis or the Laguerre basis of dimension $m \geq 1$ and $\mathcal{L}(m) := \underset{x \in (l, +\infty)}{\sup}{\sum_{i=0}^{m-1}\phi_i^2(x)}$, with $(\phi_0, \ldots, \phi_{m-1}$) a generic basis that represents the Hermite basis and the Laguerre basis. This result is crucial as it allows us to avoid dimension truncation in the approximation spaces and to derive risk bounds for projection estimators. The derivation of these bounds requires the diffusion dynamics to lie in a stable regime characterized by a balance between the dissipativity of the drift coefficient $b$ and the diffusion intensity. This leads to the study of both exponential ergodicity and polynomial ergodicity of the diffusion process. 
    \item Derivation of risk bounds and a convergence rate of order $n^{-s/(2s+d_0)}$, with $s>d_0>1$, when $b$ and $\sigma$ are coefficients of a polynomially ergodic process, $\sigma^2$ is globally elliptic and belongs to $\mathcal{W}_{\pi}^s((l,+\infty), R_0) \cap \mathcal{C}^{0,\alpha}((l, +\infty), \mathbb{R})$, where $\mathcal{W}_{\pi}^s((l,+\infty), R_0)$ is a Sobolev space of smoothness parameter $s>0$ and $\mathcal{C}^{0,\alpha}((0, +\infty), \mathbb{R})$ is a space of $\alpha$-H\"older continuous fonctions with $\alpha \in [1/2,1]$. For exponentially ergodic processes with an unbounded diffusion coefficient $\sigma$ that exhibit polynomial growth near infinity, we establish an explicit convergence rate of order $n^{-s/(2s+k_0)}$, where $s>k_0$ and $k_0  >  1$ is a constant. The association of a globally elliptic $\sigma$ with a polynomially ergodic process, and of an unbounded $\sigma$ with an exponentially ergodic process, is motivated by stability considerations in the dynamics of the process, which are required to control $\mathcal{L}(m)\|\Psi_m^{-1}\|_{\mathrm{op}}$.
    \item Model selection and derivation of risk bounds for each ergodicity regime using Talagrand's inequality. 
\end{enumerate}

\subsection{Outline of the paper}

Section~\ref{sec:framework} is devoted to the definition of the general framework for the estimation of $\sigma^2$, and Section~\ref{sec:projection} focuses on the construction of projection estimators of $\sigma^2$ and the derivation of some key results. In Sections~\ref{sec:consistency-rates} and \ref{sec:Adaptive}, the risk bounds of both non-adaptive and adaptive estimators of $\sigma^2$ are established, and Section~\ref{sec:Numerical-study} is devoted to a numerical study. Finally, the proofs of the main results are provided in Section~\ref{sec:proofs} and the appendix.

\section{General framework and assumptions}
\label{sec:framework}

The present section addresses the definition of a statistical setting tailored to our diffusion model. First, we introduce the notation and define the metrics that will be used throughout the paper. Second, we present some key assumptions on the coefficients of our diffusion model that allow us to construct consistent non-adaptive and adaptive estimators of the diffusion coefficient. Finally, we briefly introduce the function spaces that contain the unknown function $\sigma^2$ and the associated approximation spaces.

\subsection{Notation and definitions}
\label{subsec:Def-Not}

We fix the notation and definitions of key objects that are used throughout the paper. For this purpose, we suppose to have $n+1$ observations $X_{k\Delta_n}, ~ k = 0, \ldots, n$ of the unique strong solution $X$ of Equation~\eqref{eq:model}, where $\Delta_n$ tends to zero and $n\Delta_n$ tends to infinity as $n$ tends to infinity. Assume that the process $X$ admits an invariant distribution $\pi_X$. Then for any $h \in \mathbb{L}^2((l,+\infty), \pi_X(dx))$, we have the following metrics:
$$\left\|h\right\|_n^2 := \dfrac{1}{n}\sum_{k=0}^{n-1}h^2(X_{k\Delta_n}), ~~~ \left\|h\right\|_{\pi}^2 := \int_{\mathbb{R}}h^2(x)\pi_X(x)dx.$$
We denote by $\mathbb{P}$ the distribution of the process $X$ and $\mathbb{E}$ its corresponding expectation. We particularly obtain for all $h \in \mathbb{L}^2((l,+\infty), \pi_X(dx)), ~ \mathbb{E}_X\left[\left\|h\right\|_n^2\right] = \left\|h\right\|_{\pi}^2$. For $k \in \mathbb{N}$, we denote by $\mathcal{C}^k(I,\mathbb{R})$ the set of $k$-continuously differentiable functions from the interval $I$ to $\mathbb{R}$. $\mathcal{B}_b(I, \mathbb{R})$ denotes the set of Borel functions from $I$ to $\mathbb{R}$. For any integers $p,q \in \mathbb{N}\setminus\{0\}$ and for any matrix $P \in \mathbb{R}^{p \times q}$, $P^{\prime}$ denotes the transpose matrix of $P$. In addition, when $P \in \mathbb{R}^{p \times p}$ is a symmetrical and positive definite matrix, the operator norm of $P^{-1}$ is defined by
\begin{equation*}
    \left\|P^{-1}\right\|_{\mathrm{op}} := \underset{v \in \mathbb{R}^p\setminus\{0\}}{\sup}\dfrac{\left\|P^{-1}v\right\|_{2,p}}{\|v\|_{2,p}} = \dfrac{1}{\min\{\lambda_1, \ldots, \lambda_r\}},
\end{equation*}
where $\{\lambda_1, \ldots, \lambda_r\} \subset (0, +\infty)$ is the spectrum of $P$ with $r \leq p$. Denote by $\left<.,.\right>$ the standard inner product and $\|.\|_{2,p}$ the euclidean norm in $\mathbb{R}^p$. For any $A,B \in \mathbb{R}$, we denote $A \land B = \min\{A, B\}$ and $A \vee B = \max\{A, B\}$, $\lfloor A \rfloor$, the highest integer strictly smaller than $A$, $\lceil A \rceil$, the smallest integer strictly greater than $A$. For all $p,q \in \mathbb{Z}$ such that $p<q$, we denote by $[\![p, q]\!]$ the set $\{p, p+1, \ldots, q\}$. Finally, consider the functions $f$ and $g$ defined on $I \subset \mathbb{R}$. We say $f(x) = \mathcal{O}(g(x))$ if there exists a constant $C>0$ such that for all $x \in I, ~ |f(x)| \leq C|g(x)|$.

\subsection{Assumptions}
\label{subsec:Assumptions}

In this section, we define an appropriate statistical framework for the construction of non-adaptive and adaptive estimators of $\sigma^2$, and the establishment of theoretical results that validate the performance of the estimators. To this end, consider the following sets:
$$  \mathcal{N}(b) := \left\{x \in (l,+\infty) : ~ b(x) = 0\right\}, ~~~~ \mathcal{N}(\sigma) := \left\{x \in (l,+\infty) : ~ \sigma(x) = 0\right\}$$
and 
$$\mathcal{I}(\sigma) := \left\{x \in (l,+\infty): ~ \int_{-\varepsilon}^{+\varepsilon}\dfrac{dy}{\sigma^2(x + y)} = \infty, ~~ \forall \varepsilon > 0\right\}.$$
We make the following assumptions on the coefficients $b$ and $\sigma$ of the stochastic differential equation \eqref{eq:model}. 

\begin{assumption}\label{ass:NoBlocade}
    The following condition and convention hold.
    \begin{itemize}
        \item $\mathcal{I}(\sigma) \subset \mathcal{N}(b) \cap \mathcal{N}(\sigma)$.
        \item When $\mathcal{N}(b) \cap \mathcal{N}(\sigma) \neq \emptyset$, $b(x)/\sigma^2(x) = 0$ for all $x \in \mathcal{N}(b) \cap \mathcal{N}(\sigma)$.
    \end{itemize}
\end{assumption}

\begin{assumption}\label{ass:drift}
There exists a positive increasing concave function $\kappa : (0, \infty) \rightarrow (0,\infty)$ such that  
\begin{itemize}
    \item $\left|b(x) - b(y)\right| \leq \kappa(|x-y|)$ for all $x, y \in (l, +\infty)$,
    \item for all $\varepsilon > 0, ~~ \int_{0}^{\varepsilon}\kappa^{-1}(u)du = \infty$.
\end{itemize}
\end{assumption}

\begin{assumption}\label{ass:diffusion}
    There exists a positive increasing and continuous function $\rho: (0, \infty) \rightarrow (0, \infty)$ such that
\begin{itemize}
    \item $\left|\sigma(x) - \sigma(y)\right| \leq \rho(|x-y|)$ for all $x, y \in (l,+\infty)$,
    \item for all $\varepsilon > 0, ~~ \int_{0}^{\varepsilon}\rho^{-2}(u)du = \infty$,
    \item $\rho$ has polynomial growth when $x \rightarrow +\infty$. More precisely, there exists $D \geq 1, ~ C>0$ and $A>0$ large enough such that $\rho(x) \leq Cx^D$ for all $x \geq A$.
    \item There exists $\alpha \in [1/2, 1]$ such that $\rho(x) = \mathcal{O}(x^{\alpha})$ when $x \rightarrow 0^+$.
\end{itemize}
\end{assumption}
The above assumptions, which are weak enough to include a sufficiently large range of diffusion models characterized by the diffusion coefficients $b$ and $\sigma$, allow us to define an appropriate statistical setup for the study of projection estimators of $\sigma^2$. Let us start with Assumptions~\ref{ass:drift} and \ref{ass:diffusion} which are the minimum assumptions on the coefficients $b$ and $\sigma$ that ensure the pathwise uniqueness of a solution of Equation~\eqref{eq:model} if any (see \cite{yamada1971uniqueness}, \textit{Theorem 1}). These assumptions require modulus of continuity on coefficients $b$ and $\sigma$, and the Osgood's conditions imposed on positive increasing functions $\kappa$ and $\rho$ imply that 
$\underset{x \rightarrow 0^+}{\lim} \kappa(x) = \underset{x \rightarrow 0^+}{\lim} \rho(x) = 0$, yielding continuity of functions $b$ and $\sigma$. Combining the continuity of functions $b$ and $\sigma$ through Assumptions~\ref{ass:drift} and \ref{ass:diffusion} with Assumption~\ref{ass:NoBlocade}, we also find that $b\sigma^{-2}$ is locally integrable on $\mathcal{I}(\sigma)^{c}$, resulting in the existence of a weak solution of Equation~\eqref{eq:model} (see \cite{engelbert1989strong}, \textit{Theorem 4.53}). Furthermore, coupling the existence of a weak solution with pathwise uniqueness yields the existence of a unique strong solution $X = (X_t)_{t \geq 0}$ (see, \textit{e.g.} \cite{revuzyor1999}, \textit{Chap. IX, Theorem 1.7, p.368}, or \cite{karatzas2014brownian}, \textit{Chap 5, Corollary 5.10, p.338}). One can also obtain the existence of a unique strong solution under Assumptions~\ref{ass:drift} and \ref{ass:diffusion} provided that $\kappa(0) = \rho(0) = 0$ (see \cite{prokhorov1998probability}, \textit{Chap. 2, Theorem 1.4, p. 40-41}). 

We now focus on the study of some key properties of the diffusion process $X$ that are indispensable for the study carried out in this paper. Consider the sequence of random variables $(\zeta_N)_{N \in \mathbb{N}}$ given by
$$\zeta_N := \inf\left\{t \geq 0: ~ X_t \notin (l_N, r_N)\right\}, ~~ N \in \mathbb{N},$$
with $\underset{N \rightarrow \infty}{\lim} l_N = l, ~~ \underset{N \rightarrow \infty}{\lim} r_N = +\infty$, and set $\zeta = \underset{N \rightarrow \infty}{\lim}{\zeta_N}$. $\zeta$ is known as the explosion time of the process $X$ taking values in the interval $(l,+\infty)$. For the purpose of this paper, we want the process $X$ to be non-explosive, that is, $\mathbb{P}(\zeta = +\infty) = 1$, which allows one to establish the ergodicity of the process followed with the existence of a unique invariant probability measure on the support $(l, +\infty)$ of the process $X$. To do this, we make the following additional assumptions.
\begin{assumption}\label{ass:NonDegeneracy}
    The following conditions hold for $b$ and $\sigma$.
    \begin{itemize}
        \item For all $x \in (l, +\infty), ~~ \sigma(x) \neq 0$.
        \item There exists $d > 0$ such that $\dfrac{xb(x)}{\sigma^2(x)} \longrightarrow -d ~~ \mathrm{as} ~~ |x| \rightarrow \infty.$
        \item When $l \in \mathbb{R}$, there exists $\mathfrak{c} \geq 1$ such that for all $\mathfrak{d} \in (l, +\infty)$, $\underset{x \rightarrow l^+}{\lim}{(x-l)^{\mathfrak{c}}\exp\left(-2\int_{\mathfrak{d}}^{x}\dfrac{b(z)}{\sigma^2(z)}dz\right)} = +\infty$.
    \end{itemize}
\end{assumption}
\begin{assumption}\label{ass:StrongAssDrift}
    There exist $r > 0, ~ q \geq 1$ and $R > 0$ such that $xb(x) \leq - r |x|^{q+1}, ~~ \forall ~ x \in (l, +\infty) : ~ |x| > R$.
\end{assumption}
Assumption~\ref{ass:NonDegeneracy} automatically implies Assumption~\ref{ass:NoBlocade} since $\mathcal{I}(\sigma) = \mathcal{N}(b) \cap \mathcal{N}(\sigma) = \emptyset$ and the function $x \mapsto b(x)/\sigma^2(x)$ is well defined on $(l, +\infty)$. Consider the scale function $x \mapsto S(x)$ and the speed measure $\mu(dx)$ related to Equation~\eqref{eq:model} and given as follows (see \cite{revuzyor1999}, \textit{Chap. VII, Exercise 3.20, p.311}):
\begin{align*}
    \forall x \in (l, +\infty), ~ & S(x) = \int_{\mathfrak{d}}^{x}s(y)dy, ~~ \mathrm{where} ~~ s(y) = \exp\left(-2\int_{\mathfrak{d}}^{y}\dfrac{b(z)}{\sigma^2(z)}dz\right),\\
    & \mu(dx) = \dfrac{2dx}{s(x)\sigma^2(x)} = \dfrac{2dx}{\sigma^2(x)}\exp\left(2\int_{\mathfrak{d}}^{x}\dfrac{b(z)}{\sigma^2(z)}dz\right),
\end{align*}
where $\mathfrak{d} \in (l, +\infty)$. The functions $S$ and $m$ are well defined under Assumptions~\ref{ass:drift}, \ref{ass:diffusion} and \ref{ass:NonDegeneracy}. Under Assumption~\ref{ass:StrongAssDrift}, the function $y \in (R, +\infty) \mapsto -2\int_{R}^{y}\frac{b(z)}{\sigma^2(z)}dz$ takes values in $(0, \infty)$, the same remark is applied to the function $y \in (l, -R) \mapsto 2\int_{y}^{-R}\frac{b(z)}{\sigma^2(z)}dz \in (0, +\infty)$ when $l=-\infty$. Then, from Assumptions~\ref{ass:drift}, \ref{ass:diffusion}, \ref{ass:NonDegeneracy} and \ref{ass:StrongAssDrift} we obtain the following for all $l \in \mathbb{R} \cup \{-\infty\}$, 
\begin{equation}\label{eq:S-m}
    S(l) = -\infty, ~~ S(+\infty) = +\infty ~~ \mathrm{and} ~~ \int_{l}^{+\infty}m(x)dx < \infty.
\end{equation} 
The above results have multiple implications for the diffusion process $X$. First, the process $X$ is non-explosive (see, \textit{e.g.} \cite{karatzas2014brownian}, \textit{Chap.5, Proposition 5.22, p.345}). Second, $X$ is positive Harris recurrent and admits an invariant probability measure $\pi_X(dx)$ given by
\begin{equation}\label{eq:InvariantDensity}
    \pi_X(dx) := \dfrac{2dx}{M\sigma^2(x)}\exp\left(2\int_{\mathfrak{d}}^{x}\dfrac{b(z)}{\sigma^2(z)}dz\right), ~~ \mathrm{with} ~~ M = \int_{l}^{+\infty}m(x)dx < \infty
\end{equation}
(see \cite{karatzas2014brownian}, Chap. 5, Exercise 5.40, p.352-353). 
Third, the condition on the function $x \mapsto xb(x)/\sigma^2(x)$ (Assumption~\ref{ass:NonDegeneracy}) governs the dynamical balance between dissipative confinement induced by the drift coefficient and geometric dispersion driven by the diffusion coefficient. The parameter $d$ characterizes the relative strength of the mean-reverting pull of the process toward the origin compared to its volatility. Thus, as $d$ increases, the restoring force becomes dominant relative to the diffusivity, whereas when $d$ is close to zero, this restoring effect becomes negligible. Consequently, stable regimes that lead to a robust estimation of $\sigma^2$ can arise either from the coupling of a weak dissipative effect with a weakly diffusive regime ($\sigma(x)$ essentially varies in $(0,1]$), or from the coupling of a strong but balanced dissipative effect with a sufficiently strong diffusive regime (where $\sigma$ is unbounded or bounded but takes sufficiently large values). Assumption~\ref{ass:NonDegeneracy} also allows us to avoid truncation of the dimension of the approximation spaces and contributes to the derivation of explicit convergence rates when the diffusion coefficient is unbounded, provided that $d$ is large enough. It implies that $\pi_X(x) \rightarrow 0$ as $|x| \rightarrow \infty$ with a polynomial rate, which has several theoretical implications (see Propositions~\ref{prop:Exp-Holder} and \ref{prop:operator-norm-identity}). 

Fourth, since the scale function $S \in \mathcal{C}^2((l, +\infty), \mathbb{R})$ is increasing and the invariant measure $\pi_X(dx)$ satisfies $\pi_X(dx) > 0$, from \cite{itodiffusion}, \textit{Chap. 4, p.106-107, 149 - 158}, the diffusion process $X$ admits a transition density $(t,x,y) \mapsto p_X(t,x,y)$ with respect to the invariant probability measure $\pi_X(dx)$ that is continuous, symmetric, and strictly positive. Consequently, the transition semi-group $P^X: (t, x, f) \mapsto P_{t}^Xf(x)$ given for all $(t, x, f) \in (0, +\infty) \times (l, +\infty) \times \mathcal{B}_b((l, +\infty), \mathbb{R})$ by
$$P_t^Xf(x) := \int_{(l, +\infty)}p_X(t,x,y)f(y)\pi_X(dy)$$
has the strong Feller property. Then $X$ is called a $T$-process in the sense of Meyn \& Tweedie (see \cite{meyn1993stabilityII}, Section 3.2, Theorem 3.3). Fifth, for any $A \subset (l, +\infty)$ such that $\pi_X(A) > 0$, we have 
$$\mathbb{P}\left(X_t \in A | X_0 = x\right) = \int_{A}p_X(t,x,y)\pi_X(dy) > 0 ~~ \forall (t,x) \in (0, +\infty) \times (l, +\infty).$$
So, the process $X$ is $\pi_X$-irreducible (see \cite{meyn1992stabilityI}, Section 1.3), and the invariant probability measure $\pi_X$ is then unique. It follows from Theorem 4.1 in \cite{meyn1993stabilityII} that every compact subset of $(l, +\infty)$ is $\pi_X$-petite. Moreover, recall that the infinitesimal generator $\mathcal{L}$ of the diffusion process $X$ is given for all $f \in \mathcal{C}^2((l, +\infty), \mathbb{R})$ and for all $x \in (l, +\infty)$ by
$$\mathcal{L}f(x) = b(x)f^{\prime}(x) + \dfrac{1}{2}\sigma^2(x)f^{\prime\prime}(x).$$
Consider the function $V_0: x \mapsto (1 + x^{2p})^{1/2p}$ where $p \geq 1$. For all $x \in (l, +\infty)$, we have
$$\mathcal{L}V_0(x) = b(x)V_0^{\prime}(x) + \dfrac{1}{2}\sigma^2(x)V_0^{\prime\prime}(x) = \dfrac{x^{2p-1}b(x)}{(1 + x^{2p})^{(2p-1)/2p}} + \dfrac{(2p-1)x^{2p-2}\sigma^2(x)}{(1 + x^{2p})^{(4p-1)/2p}}.$$
Under Assumptions~\ref{ass:drift}, \ref{ass:diffusion}, \ref{ass:NonDegeneracy} and \ref{ass:StrongAssDrift} and provided that the functions $b$ and $\sigma$ are continuous, there exist constants $c, C > 0$ such that
\begin{equation}\label{eq:Lyapunov-Foster}
    \mathcal{L}V_0(x) \leq -cV_0(x) + C\mathds{1}_{|x| \leq A} \leq -cV_0(x) + C, ~~ x \in (l, +\infty),
\end{equation}
where $A > R$ is large enough with respect to $R$. We deduce from Theorems 5.1 and 6.1 in \cite{meyn1993stabilityIII} that the process $X$ is ergodic and there exist $\rho \in (0,1)$ and $C_0 > 0$ such that
$$\left\|\mathbb{P}\left(X_t \in dy | X_0 = x\right) - \pi_X(dy)\right\|_{TV} \leq C_0V(x)\exp\left(t\log(\rho)\right), ~~ \forall x \in (l, +\infty),$$
where $V = 1 + V_0$, and $\|.\|_{TV}$ is the total variation norm. We can now assume that the process is in the stationary regime for all $t \geq 0$, via the following assumption.
\begin{assumption}\label{ass:stationary}
    $X_0 \sim \pi_X(dx)$.
\end{assumption}
The above assumption on $X_0$ with $\pi_X((l, +\infty)) < \infty$ implies the following result:
\begin{equation}\label{eq:expo-beta-mixing}
    \beta_X(t) \leq C\exp(-\gamma t),
\end{equation}
where $C, \gamma > 0$ are constants and $\beta_X(t)$ is the $\beta$-mixing coefficient of $X_t, ~ t \geq 0$ given by
$$\beta_X(t) = \int_{(l, +\infty)}\pi_X(dx)\left\|\mathbb{P}\left(X_t \in dy | X_0 = x\right) - \pi_X(dy)\right\|_{TV}.$$
Note that under Assumption~\ref{ass:StrongAssDrift}, the drift coefficient $b$ satisfies $b(x) \rightarrow \pm \infty$ as $x \rightarrow \pm \infty$. In these conditions, we show in the next sections that to obtain a consistent projection estimator of $\sigma^2$ without any truncation of dimension, $\sigma^2$ should be unbounded and $\sigma^2(x) \geq c|x|^{q+1}$ when $|x| \rightarrow \infty$, with $c>0$ a constant and $q \geq 1$ given in Assumption~\ref{ass:StrongAssDrift}. The following assumption allows us to consider diffusion models with elliptic diffusion coefficients on $(l, +\infty)$.
\begin{assumption}\label{ass:limit-case-elliptic}
    For all $x \in (l, +\infty), ~ \sigma_{-}^2 \leq \sigma^2(x) \leq \sigma_{+}^2$ where
    $$\sigma_{-}^2 := \underset{x \in (l, +\infty)}{\inf}{~\sigma^2(x)} > 0 ~~ \mathrm{and} ~~ \sigma_{+}^2 := \underset{x \in (l, +\infty)}{\sup}{~\sigma^2(x)} < \infty.$$
    There exist $r \geq 8\sigma_{+}^2 + (\sigma_{+}^2 - \sigma_{-}^2)/2$ and $R>0$ such that
    $$\forall ~ x \in (l, +\infty): ~ |x| > R, ~~ \mathrm{sgn}(x)b(x) \leq -r/|x|,$$
    where $\mathrm{sgn}(x) = x/|x|$ for all $x \neq 0$ and $\mathrm{sgn}(0)=-1$.
\end{assumption}
Under Assumption~\ref{ass:limit-case-elliptic} combined with Assumptions~\ref{ass:drift}, \ref{ass:diffusion}, \ref{ass:NonDegeneracy} and \ref{ass:stationary}, we have a non-explosive process that is positive Harris recurrent, ergodic with a unique invariant probability measure. However, exponential ergodicity is no longer guaranteed as it is not possible to find a norm-like function $V_0$ so that Equation~\eqref{eq:Lyapunov-Foster} holds. We can only rely on the polynomial bound on the $\beta$-mixing coefficient $\beta_X(t)$ established in \cite{veretennikov1997polynomial}, \textit{Theorem 5}, and we obtain
$$\beta_X(t) \leq C(1+t)^{-(k+1)},$$
where $k \in (0, \gamma), ~ \gamma = r_0 - 3/2$ and $r_0 = \left[r - (\sigma_{+}^2 - \sigma_{-}^2)/2\right]/\sigma_{+}^2$. In the sequel, we assume that $r_0$ is large enough (for example $r_0 \geq 16$). The immediate consequence of a polynomial $\beta$-mixing is a slower decay of the dependence between past and future observations of the diffusion process $X$ contrary to the exponential $\beta$-mixing. 

A crucial point in the study of risk bounds for projection estimators of $\sigma^2$ is the control of quantities of the form $\mathbb{E}\left[\left|f(X_{t+h}) - f(X_t)\right|^{p}\right]$ when $h \rightarrow 0$ with $p > 0$ and $f = b, \sigma^2$. We obtain the following result.
\begin{prop}\label{prop:Exp-Holder}
    Under Assumptions~\ref{ass:drift}, \ref{ass:diffusion}, \ref{ass:NonDegeneracy}, \ref{ass:StrongAssDrift} and \ref{ass:stationary}, for all $p \in (0, d/D)$ with $d > 8D$, we have $\mathbb{E}\left[\sigma^{2p}(X_0)\right] < \infty$ and there exist constants $C_{p}, C_{p}^{\prime} > 0$ such that
    \begin{itemize}
        \item[(i)] $\forall ~ h>0, ~~ \mathbb{E}\left[\left|b(X_{t+h}) - b(X_t)\right|^{2p}\right] \leq C_{p}h^{p},$
        \item[(ii)] $\mathbb{E}\left[\left|\sigma(X_{t+h}) - \sigma(X_t)\right|^{2p}\right] \leq C_{p}^{\prime}h^{\alpha p} ~~ \mathrm{when} ~~ h \rightarrow 0^+$.
    \end{itemize}
 Results (i) and (ii) still hold under Assumptions~\ref{ass:drift}, \ref{ass:diffusion}, \ref{ass:NonDegeneracy}, \ref{ass:stationary} and \ref{ass:limit-case-elliptic} for all $p \in (0, d-1)$ with $d>4$. Under Assumption~\ref{ass:limit-case-elliptic}, we have $\mathbb{E}\left[\sigma^{2p}(X_0)\right] < \infty$ for all $p>0$.
\end{prop}
Proposition~\ref{prop:Exp-Holder} provides a key result allowing establishment of risk bounds for projection estimators of $\sigma^2$ on $(l, +\infty)$. Under Assumption~\ref{ass:limit-case-elliptic}, we can always choose the drift coefficient $b$ so that $\left|xb(x)\right|/\sigma^2(x) \rightarrow d > 4$ as $|x| \rightarrow \infty$. Note that if $\rho(x) = \mathcal{O}(x)$ on the interval $(0, \infty)$, then the above result coincides with that of Proposition A in \cite{gloter2000discrete} with $f \in \{b, \sigma\}$,  assuming that $b$ and $\sigma$ are differentiable and for all $x \in (l, +\infty)$, $\left|b^{\prime}(x)\right| + \left|\sigma^{\prime}(x)\right| = \mathcal{O}\left(1 + |x|^{\gamma}\right)$ with $\gamma \geq 0$. However, the proofs differ, as in the present paper, we make use of Assumptions~\ref{ass:drift} and \ref{ass:diffusion}, while the proof method in \cite{gloter2000discrete} relies on the finite increment inequality. Proposition~\ref{prop:Exp-Holder} extends the result to functions that are not necessarily differentiable. 

In the next section, we define bases of non compactly supported functions suitable for estimating $\sigma^2$ on $(l, +\infty)$.

\subsection{Approximation spaces}
\label{subsec:spaces}

The initial space $\mathbb{L}^2((l, +\infty), \pi_X(dx))$ should be approximated by finite dimensional spaces generated by bases whose Gram matrices are invertible and the operator norm of their inverses can be controlled, the goal being to avoid truncation of the dimension. In this context, the Hermite and the Laguerre bases are suitable candidates that satisfy the conditions required and are presented below.

\paragraph{\textbf{The Hermite basis.}} The Hermite basis of dimension $m \geq 1$ is the orthonormal basis $(h_0, \ldots, h_{m-1})$, where for each $j \in \{0, \ldots, m-1\}$, $x \mapsto h_j(x) = c_jH_j(x)\mathrm{e}^{-x^2/2}$ belongs to the space $\mathbb{L}^2(\mathbb{R})$ with
$$ H_j(x) := (-1)^j\mathrm{e}^{x^2}\dfrac{d^j}{dx^j}\left(\mathrm{e}^{-x^2}\right) ~~ \mathrm{and} ~~ c_j := \dfrac{1}{\sqrt{2^j j! \sqrt{\pi}}}.$$
For all $j \geq 0, ~ \|h_j\|_{\infty} < \infty$ and we deduce the following: 
$$\mathcal{L}(m) := \underset{x \in \mathbb{R}}{\sup}{\sum_{j=0}^{m-1}h_j^2(x)} \leq \underset{j \in \{0,\ldots, m-1\}}{\max}\|h_j\|_{\infty}^2m.$$ 
Moreover, for all $j \geq 0$, there exist constants $c,c_0>0$ such that $|h_j(x)| \leq c|x|\exp\left(-c_0x^2\right)$ for all $x \in \mathbb{R}$ satisfying $x^2 \geq 2(2j+1)$ (see, \textit{e.g.}, \cite{askey1965mean}). This property is essential for the control of $\mathcal{L}(m)\|\Psi_m^{-1}\|_{\mathrm{op}}$ (see the proof of Proposition~\ref{prop:operator-norm-identity}).

\paragraph{\textbf{The Laguerre basis.}} The Laguerre basis of dimension $m \geq 1$ is the orthonormal basis $(\ell_0, \ldots, \ell_{m-1})$ where the functions $x \mapsto \ell_j(x) = \sqrt{2}L_j(2x)\mathrm{e}^{-x}\mathds{1}_{x \geq 0}, ~ j \in \{0, \ldots, m-1\}$ belong to the space $\mathbb{L}^2([0, +\infty))$ with
$$ L_j(x) := \sum_{k=0}^{j}(-1)^k\binom{j}{k}\dfrac{x^k}{k!}.$$
Similarly, we have $\|\ell_j\|_{\infty} < \infty$ and we obtain
$$\mathcal{L}(m) := \underset{x \in \mathbb{R}}{\sup}{\sum_{j=0}^{m-1}\ell_j^2(x)} \leq \underset{j \in \{0,\ldots, m-1\}}{\max}\|\ell_j\|_{\infty}^2m.$$ 
We also obtain for each $j \geq 0$ that there exist constants $c, \gamma > 0$ such that $|\ell_j(x)| \leq c\exp\left(-\gamma x\right)$ for all $x \geq 3(2j+1)$ (see, \textit{e.g.}, \cite{askey1965mean}). 

\paragraph{\textbf{Spaces of approximation.}} Let $m \geq 1$ be an integer. We consider the $m$-dimensional space $\mathcal{S}_m := \mathrm{Span}(\phi_0, \ldots, \phi_{m-1})$ spanned by a generic orthonormal basis $(\phi_0, \ldots, \phi_{m-1})$ on the interval $(l, +\infty)$ that represents the Hermite basis ($l = - \infty$) or the Laguerre basis ($l = 0$). The collection $(\mathcal{S}_m)_{m \geq 1}$ of approximation sub-spaces is nested, that is, for all $m,m^{\prime} \geq 1$ such that $m \leq m^{\prime}$, we have $\mathcal{S}_m \subset \mathcal{S}_{m^{\prime}}$. We impose a control on the coordinate vectors of elements of $\mathcal{S}_m$ and obtain the following space:
$$\mathcal{S}_{m,L} := \left\{f = \sum_{i = 0}^{m-1}a_i\phi_i \in \mathcal{S}_m: ~ \sum_{i=0}^{m-1}a_i^2 \leq mL\right\},$$
where $L>0$ is a real number. The collection of models $(\mathcal{S}_{m,L})_{m \geq 1}$ is also nested ($\forall ~ m,m^{\prime} \geq 1, ~~ m \leq m^{\prime} \Rightarrow \mathcal{S}_{m,L} \subset \mathcal{S}_{m^{\prime}, L}$). The $\ell^2$-constraint imposed on the coordinate vectors of elements of the sub-spaces allows the construction of nonparametric ridge estimators of $\sigma^2$ and the derivation of their risk bounds. The real number $L>0$ may depend on $n$ with logarithmic growth.\\
The next section is devoted to the construction of projection estimators of $\sigma^2$ on the Hermite and Laguerre bases.

\section{Projection estimators of the diffusion coefficient}
\label{sec:projection}
\label{subsec:ProjectionEstimators}

We define projection estimators of the squared diffusion coefficient $\sigma^2$ on non-compact supports, built from the observation points $X_{k\Delta_n}, ~ k = 0, \ldots, n$.

\subsection{Projection estimators of the diffusion coefficient}
\label{subsec:estimator-sigma}

The definition of an efficient regression model for the construction of estimators of $\sigma^2$ requires the following additional assumption on $\sigma$.
\begin{assumption}\label{ass:Regular-Sigma}
    $\sigma^2 \in \mathcal{C}^2((l, +\infty), \mathbb{R})$, and there exist $C>0$ and $\theta \in [0,D]$ such that $|\sigma^{\prime}(x)| + |\sigma^{\prime\prime}(x)| \leq C(1+|x|^{\theta})$ for all $x \in (l,+\infty)$, where $D\geq 1$ is defined in Assumption~\ref{ass:diffusion}.
\end{assumption}
Then, under Assumptions~\ref{ass:NoBlocade}, \ref{ass:drift}, \ref{ass:diffusion} and \ref{ass:Regular-Sigma}, the projection estimators of $\sigma^2$ are built from the following regression model:
\begin{equation*}
     U_{k\Delta_n} := \dfrac{\left(X_{(k+1)\Delta_n} - X_{k\Delta_n}\right)^2}{\Delta_n} = \sigma^2(X_{k\Delta_n}) + \xi_{k\Delta_n} + R_{k\Delta_n}, ~~ k \in \{0, \ldots, n-1\},
\end{equation*}
where $U_{k\Delta_n}$ is the response variable, $\xi_{k\Delta_n} = \xi_{k\Delta_n}^{(1)} + \xi_{k\Delta_n}^{(2)} + \xi_{k\Delta_n}^{(3)}$ is the error term with
\begin{equation}\label{eq:Error-term}
    \begin{aligned}
        \xi_{k\Delta_n}^{(1)} := &~ \dfrac{1}{\Delta_n}\left[\left(\int_{k\Delta_n}^{(k+1)\Delta_n}\sigma(X_s)dW_s\right)^2 - \int_{k\Delta_n}^{(k+1)\Delta_n}\sigma^2(X_s)ds\right],\\
        \xi_{k\Delta_n}^{(2)} := &~ \dfrac{2}{\Delta_n}\int_{k\Delta_n}^{(k+1)\Delta_n}((k+1)\Delta_n - s)\sigma^{\prime}(X_s)\sigma^2(X_s)dW_s ~~ \mathrm{and} ~~ \xi_{k\Delta_n}^{(3)} := 2b(X_{k\Delta_n})\int_{k\Delta_n}^{(k+1)\Delta_n}\sigma(X_s)dW_s,
    \end{aligned}
\end{equation}
and $R_{k\Delta_n} = R_{k\Delta_n}^{(1)} + R_{k\Delta_n}^{(2)} + R_{k\Delta_n}^{(3)}$ is a negligible residual with
\begin{equation}\label{eq:Residual}
    \begin{aligned}
        R_{k\Delta_n}^{(1)} := &~ \dfrac{1}{\Delta_n}\left(\int_{k\Delta_n}^{(k+1)\Delta_n}b(X_s)ds\right)^2, ~~~~ R_{k\Delta_n}^{(2)} := \dfrac{1}{\Delta_n}\int_{k\Delta_n}^{(k+1)\Delta_n}((k+1)\Delta_n - s)\Gamma(X_{k\Delta_n}))ds,\\
        R_{k\Delta_n}^{(3)} := &~ \dfrac{2}{\Delta_n}\int_{k\Delta_n}^{(k+1)\Delta_n}\left(b(X_s) - b(X_{k\Delta_n})\right)ds\int_{k\Delta_n}^{(k+1)\Delta_n}\sigma(X_s)dW_s.
    \end{aligned}
\end{equation}
where $\Gamma = 2b\sigma\sigma^{\prime} + \left[\sigma\sigma^{\prime\prime} + (\sigma^{\prime})^2\right]\sigma^2$. Consider the contrast function $\gamma_n : f \in \mathcal{S}_n \mapsto \gamma_n(f)$ given by
    \begin{equation}\label{eq:gamma-function}
        \gamma_n(f) := \dfrac{1}{n}\sum_{k=0}^{n-1}\left(U_{k\Delta_n} - f(X_{k\Delta_n})\right)^2.
    \end{equation}
    The projection estimators of $\sigma^2$ are defined by
    \begin{equation}\label{eq:NA-estimator}
        \widehat{\sigma}_{m}^2 := \underset{f \in \mathcal{S}_{m, L}}{\arg\min}{~\gamma_n(f)}, ~~ m \geq 1.
    \end{equation}
    Let $\Psi_m$ be the Gram matrix of the generic basis $(\phi_0, \ldots, \phi_{m-1})$, and $\widehat{\Psi}_m$ its empirical counterpart defined by
\begin{equation*}
    \begin{aligned}
        \Psi_m := &~ \left(\int_{(l, +\infty)}\phi_{i}(x)\phi_{j}(x)\pi_X(dx)\right)_{0 \leq i,j \leq m-1}, ~~~ \widehat{\Psi}_m := \dfrac{1}{n}\widehat{\Phi}_m^{\prime}\widehat{\Phi}_m,
    \end{aligned}
\end{equation*}
where $\widehat{\Phi}_m = \left(\phi_{j}(X_{k\Delta_n})\right)_{0 \leq k \leq n-1, ~ 0 \leq j \leq m-1} \in \mathbb{R}^{n \times m}$.  Consequently, for each $m \geq 1$, the projection estimator $\widehat{\sigma}_m^2$ is given for all $x \in (l, +\infty)$ by
$$\widehat{\sigma}_m^2(x) = \left<\widehat{\mathbf{a}}, (\phi_0(x), \ldots, \phi_{m-1}(x))^{\prime}\right> = \sum_{i = 0}^{m-1}\widehat{a}_i\phi_{i}(x),$$ 
where the vector $\widehat{\mathbf{a}} = (\widehat{a}_0, \ldots, \widehat{a}_{m-1})$ is given by
$$\widehat{\mathbf{a}} := \underset{\|\mathbf{a}\|_2 \leq mL}{\arg\min}{~\left\|\mathbf{U} - \widehat{\Phi}_m\mathbf{a}\right\|_2^2},$$
with $\mathbf{U} = \left(U_{k\Delta_n}, ~ k = 0, \ldots, n-1\right) \in \mathbb{R}^n$. If we assume that $\widehat{\Psi}_m$ is a.s. invertible, then 
\begin{equation}\label{eq:Coordinate-a}
    \widehat{\mathbf{a}} := \dfrac{1}{n}\widehat{\Psi}_m^{-1}\widehat{\Phi}_m^{\prime}\mathbf{U},
\end{equation}
otherwise,
\begin{equation}\label{eq:Coordinate-b}
    \widehat{\mathbf{a}} := \dfrac{1}{n}\left(\widehat{\Psi}_m + \widehat{\lambda} I_m\right)^{-1}\widehat{\Phi}_m^{\prime}\mathbf{U},
\end{equation} 
where $\widehat{\lambda}$ is the unique solution of $\left\|\widehat{\mathbf{a}}_{\lambda}\right\|_2^2 = mL$ with $\widehat{\mathbf{a}}_{\lambda} = (1/n)\left(\widehat{\Psi}_m + \lambda I_m\right)^{-1}\widehat{\Phi}_m^{\prime}\mathbf{U}$.
In the next section, we establish a key result on the Gram matrix $\Psi_m$ that is crucial for the study of risk bounds of the non-adaptive estimators of $\sigma^2$.

\subsection{Key properties on the Gram matrix}
\label{subsec:GramMatrix}

Establishing risk bounds for projection estimators of $\sigma^2$ requires a control of the quantity $\mathcal{L}(m)\left\|\Psi_m^{-1}\right\|_{\mathrm{op}}$ where the Gram matrix $\Psi_m$ is proven to be invertible. In \cite{comte2021drift}, it is assumed that
$$\mathcal{L}(m)\left\|\Psi_m^{-1}\right\|_{\mathrm{op}} \leq \dfrac{\mathfrak{c}n\Delta_n}{\log^2(n\Delta_n)},$$
where $\mathfrak{c} > 0$ is a numerical constant. This assumption leads to the construction of a truncated estimator conditional on event $\left\{\mathcal{L}(m)\|\widehat{\Psi}_m^{-1}\|_{\mathrm{op}} \leq \mathfrak{c}n\Delta_n/\log^2(n\Delta_n)\right\}$. To avoid this truncation through $\widehat{\Psi}_m$, we derive the following results.
\begin{prop}\label{prop:operator-norm-identity}
    Under Assumptions~\ref{ass:drift}, \ref{ass:diffusion}, \ref{ass:NonDegeneracy}, \ref{ass:StrongAssDrift} and \ref{ass:stationary} and for $m$ large enough, the following holds:
    $$\mathcal{L}(m)\left\|\Psi_m^{-1}\right\|_{\mathrm{op}} \leq C\left(\underset{j \in [\![0,m-1]\!]}{\max}{\left\|\phi_j\right\|_{\infty}^2}\right)m^{D+d+3/2},$$
where $d>8D$ and $C>0$ is a constant depending on $b, \sigma, D$ and $d$, and $(\phi_0, \ldots, \phi_{m-1})$ is the Hermite basis or the Laguerre basis. Moreover, under Assumptions~\ref{ass:drift}, \ref{ass:diffusion}, \ref{ass:NonDegeneracy}, \ref{ass:stationary} and \ref{ass:limit-case-elliptic}, there exists a constant $C>0$ such that
$$\mathcal{L}(m)\left\|\Psi_m^{-1}\right\|_{\mathrm{op}} \leq C\left(\underset{j \in [\![0,m-1]\!]}{\max}{\left\|\phi_j\right\|_{\infty}^2}\right)m^{d+3/2},$$
where $d>4$.
\end{prop}
Proposition~\ref{prop:operator-norm-identity} is a key ingredient in the derivation of sharp risk bounds for projection estimators of $\sigma^2$. The bound on $\mathcal{L}(m)\left\|\Psi_m^{-1}\right\|_{\mathrm{op}}$ was established for exponentially ergodic processes where $\sigma$ is unbounded and $\sigma(x) \rightarrow \infty$ as $|x| \rightarrow \infty$. This condition on $\sigma$ is indispensable to ensure that $|xb(x)|/\sigma^2(x) \rightarrow d > 8D$ as $|x| \rightarrow \infty$ since, under Assumption~\ref{ass:StrongAssDrift} and for $|x| \geq R > 0, ~ |b(x)| \geq r|x|^q$. For polynomially ergodic processes, we obtain a sharper bound on $\mathcal{L}(m)\left\|\Psi_m^{-1}\right\|_{\mathrm{op}}$ due to the ellipticity condition on $\sigma^2$ under Assumption~\ref{ass:limit-case-elliptic}. In other words, the parameter $D$ characterizes the boundlessness of $\sigma^2$ when Assumption~\ref{ass:StrongAssDrift} is granted. Then, since the functions $\phi_j, ~ j \in [\![0,m-1]\!]$ are bounded, setting $s_0 = D + d + 2$ and $m = \left\lfloor (n\Delta_n)^{2/(2s+k_0)} \right\rfloor, ~ s \geq s_0$ (or $s_0 = 2d + 3$ and $m = \left\lfloor (n\Delta_n)^{2/(2s+d_0)} \right\rfloor, ~ s \geq s_0$ under Assumption~\ref{ass:limit-case-elliptic}), there exists a constant $C>0$ such that
\begin{equation}\label{eq:operator-norm-identity}
    \mathcal{L}(m)\left\|\Psi_m^{-1}\right\|_{\mathrm{op}} \leq C(n\Delta_n)^{\beta(s)},
\end{equation}
with $\beta(s) := (2D+2d+3)/(2s+1) < 1$ under Assumption~\ref{ass:StrongAssDrift} and $\beta(s) := (2d+3)/(2s+1)<1/2$ under Assumption~\ref{ass:limit-case-elliptic}. Consider the random set $\Omega_{n,m}$ in which the metrics $\|.\|_n^2$ and $\|.\|_{\pi}^2$ are equivalent and given as follows:
$$ \Omega_{m} := \bigcap_{f \in \mathcal{S}_{m}\setminus\{0\}}\left\{\left|\dfrac{\|f\|_n^2}{\|f\|_{\pi}^2} - 1\right|\leq \dfrac{1}{2}\right\} = \left\{\underset{f \in \mathcal{S}_m, ~ \|f\|_{\pi}^2 = 1}{\sup}{\left|\|f\|_n^2 - 1\right|} \leq \dfrac{1}{2}\right\}.$$
Then, in the event $\Omega_{m}$ and for $f \in \mathcal{S}_{m}$, we have $(1/2)\|f\|_{\pi}^2 \leq \|f\|_n^2 \leq (3/2)\|f\|_{\pi}^2$. The derivation of optimal upper bounds on the risk of estimation requires an upper bound on $\mathbb{P}(\Omega_{m}^{c})$ that is sufficiently sharp with respect to the expected rates. We obtain the following result.
\begin{lemma}\label{lm:omega-comp}
    Suppose that $n \rightarrow \infty$. We have the following results. 
    \begin{itemize}
        \item Suppose that $m = \left\lfloor (n\Delta_n)^{2/(2s+k_0)} \right\rfloor$ with $k_0 = D+d+3/2$, $n\Delta_n^2 = \varepsilon_0\log^4(n)$, where $\varepsilon_0 \in (0, 10^{-1})$ is a fixed numerical constant. Under Assumptions~\ref{ass:drift}, \ref{ass:diffusion}, \ref{ass:NonDegeneracy}, \ref{ass:StrongAssDrift} and \ref{ass:stationary}, there exists a constant $C>0$ such that
        $$\mathbb{P}(\Omega_m^c) \leq C\exp\left(-\dfrac{\gamma}{4}\sqrt{\varepsilon_0}\log^2(n)\right),$$
        where the constant $\gamma>0$ is found in Equation~\eqref{eq:expo-beta-mixing}.
        \item Suppose that $m = \left\lfloor (n\Delta_n)^{2/(2s+d_0)} \right\rfloor$ with $d_0 = d+3/2$, and $n\Delta_n^2 = 1$. Under Assumptions~\ref{ass:drift}, \ref{ass:diffusion}, \ref{ass:NonDegeneracy}, \ref{ass:stationary} and \ref{ass:limit-case-elliptic}, there exists a constant $C>0$ such that
        $$\mathbb{P}(\Omega_m^c) \leq Cn^{-2}.$$
    \end{itemize}
\end{lemma}
The choice of dimension $m$ is highly sensible as it directly impacts the order of the Bayes term in the risk bound of projection estimators of $\sigma^2$. Then, the choice on $m$ is convenient, given that $n\Delta_n^2$ is either constant or of logarithmic growth. In addition, for the case of polynomially ergodic processes under Assumption~\ref{ass:limit-case-elliptic}, the result of Lemma~\ref{lm:omega-comp} remains valid if we assume that $n\Delta_n^2$ is a constant independent of $n$, $n\Delta_n^2 = 1$ being a particular case that we consider in the sequel for simplicity and without loss of generality. For the case of exponentially ergodic processes, the condition $n\Delta_n^2 = \varepsilon_0\log^4(n)$ is crucial for the derivation of a sharper bound in $\mathbb{P}\left(\Omega_m^c\right)$ and for the study of the risk bounds of adaptive estimators of $\sigma^2$ (see Section~\ref{sec:Adaptive}, proof of Theorem~\ref{thm:adaptation}). The numerical constant $\varepsilon_0 \in (0,10^{-1})$ is essential in a practical situation, used as a leverage to counter the explosion of $\log^4(n)$ when the numerical values of $n$ are not large enough, so that the time step $\Delta_n$ can remain close enough to zero (see Section~\ref{sec:Numerical-study} for more details). In the next section, we study the risk bounds of non-adaptive estimators of $\sigma^2$ considering both exponential ergodicity and polynomial ergodicity of the diffusion process.

\section{Rates of convergence of non-adaptive estimators}
\label{sec:consistency-rates}

This section is devoted to the study of convergence rates of non-adaptive estimators of $\sigma^2$ under each type of ergodic property on the diffusion process. In its current form, the results derived in this section follow from the technical results provided by Propositions~\ref{prop:Exp-Holder} and \ref{prop:operator-norm-identity} and Lemma~\ref{lm:omega-comp}. In this context, we define the set of possible values with respect to $n$ of the dimension $m$ of approximation spaces by
$$\mathcal{M}_n := \left\{1, \ldots, \left\lceil (n\Delta_n)^{2/(2s_0+1)}\right\rceil\right\}, ~~ n \rightarrow \infty,$$
where $s_0 = D+d+2$ for exponentially ergodic diffusion processes and $s_0 = 2d+3$ for polynomially ergodic processes. Recall that $\sigma^2$ is estimated on the interval $I = (l, +\infty)$ and we set $\sigma_I^2 = \sigma^2\mathds{1}_{I}$. The following theorem provides risk bounds for non-adaptive estimators of $\sigma_I^2$ for each type of ergodic property of the diffusion process. 

\begin{theo}\label{thm:upper-limit-NAE}
    Grant Assumptions~\ref{ass:drift}, \ref{ass:diffusion}, \ref{ass:NonDegeneracy} with $d > 8D$, \ref{ass:stationary} and \ref{ass:Regular-Sigma}. Moreover, assume that $L = \log(n)$ and $n \rightarrow \infty$. We have the following results.
    \begin{itemize}
        \item [(i)] Suppose that $n\Delta_n^2= \varepsilon_0\log^{4}(n) $. Under Assumption~\ref{ass:StrongAssDrift},  there exists a constant $C>0$ such that for all $m \in \mathcal{M}_n$,
        $$\mathbb{E}\left[\left\|\w{\sigma}_m^2 - \sigma_I^2\right\|_n^2\right] \leq 3\underset{f \in \mathcal{S}_{m,L}}{\inf}{\left\|f - \sigma_I^2\right\|_{\pi}^2} + C\left(\dfrac{m^{k_0}}{n} + m\log(n)\exp\left(-\frac{\gamma}{8}\sqrt{\varepsilon_0}\log^2(n)\right) + \Delta_n^{2}\right),$$
where $k_0 = D+d+3/2$.
        \item [(ii)] Assume that $n\Delta_n^2 = 1$. Under Assumption~\ref{ass:limit-case-elliptic}, there exists a constant $C>0$ such that for all $m \in \mathcal{M}_n$,
        $$\mathbb{E}\left[\left\|\w{\sigma}_m^2 - \sigma_I^2\right\|_n^2\right] \leq 3\underset{f \in \mathcal{S}_{m,L}}{\inf}{\left\|f - \sigma_I^2\right\|_{\pi}^2} + C\left(\dfrac{m^{d_0}}{n} + \dfrac{m\log(n)}{n^{2}} + \Delta_n^{2}\right),$$
where $d_0 = d+3/2$.
    \end{itemize}
\end{theo}
The above theorem provides risk bounds for non-adaptive estimators $\widehat{\sigma}_m^2, ~ m \in \mathcal{M}_n$ for both exponentially and polynomially ergodic processes. In each case, the first term in the risk bound is the approximation error, the second is the estimation error, which is of order $m^{d_0}/n$ for polynomially ergodic processes and $m^{k_0}/n$ for exponentially ergodic processes. These estimation errors are established in the event $\Omega_m$. The risk in its complementary $\Omega_m^c$ is controlled by the third term in the risk bounds. Finally, the last term is the cost of time  discretization. Note that the obtained estimation errors with $k_0,d_0 > 1$ are not standard. These results are a direct consequence of the control of $\mathcal{L}(m)\|\Psi_m^{-1}\|_{\mathrm{op}}$ to avoid dimension truncation of approximation spaces. The following corollary extends the results of Theorem~\ref{thm:upper-limit-NAE} to the risks of estimation defined with metric $\|.\|_{\pi}$.

\begin{coro}\label{cor:upper-limit-NAE}
    Grant Assumptions~\ref{ass:drift}, \ref{ass:diffusion}, \ref{ass:NonDegeneracy} with $d > 8D$, \ref{ass:stationary}, \ref{ass:Regular-Sigma} and assume that $L = \log(n)$ and $n \rightarrow \infty$.  The results are as follows.
    \begin{itemize}
        \item [(i)] Suppose that $n\Delta_n^2 = \varepsilon_0\log^{4}(n)$. Under Assumption~\ref{ass:StrongAssDrift},  there exists a constant $C>0$ such that for all $m \in \mathcal{M}_n$,
        $$\mathbb{E}\left[\left\|\w{\sigma}_m^2 - \sigma_I^2\right\|_{\pi}^2\right] \leq 9\underset{f \in \mathcal{S}_{m,L}}{\inf}{\left\|f - \sigma_I^2\right\|_{\pi}^2} + C\left(\dfrac{m^{k_0}}{n} + m\log(n)\exp\left(-\frac{\gamma}{8}\sqrt{\varepsilon_0}\log^2(n)\right) + C\Delta_n^{2}\right),$$
where $k_0 = D+d+3/2$.
        \item [(ii)] Assume that $n\Delta_n^2 = 1$. Under Assumption~\ref{ass:limit-case-elliptic}, there exists a constant $C>0$ such that for all $m \in \mathcal{M}_n$,
        $$\mathbb{E}\left[\left\|\w{\sigma}_m^2 - \sigma_I^2\right\|_{\pi}^2\right] \leq 9\underset{f \in \mathcal{S}_{m,L}}{\inf}{\left\|f - \sigma_I^2\right\|_{\pi}^2} + C\left(\dfrac{m^{d_0}}{n} + \dfrac{m\log(n)}{n^{2}} + C\Delta_n^{2}\right),$$
where $d_0 = d+3/2$.
    \end{itemize}
\end{coro}
Let $R_0>0$, $s \geq s_0$ and $I \in (l, +\infty)$ with $l \in \mathbb{R} \cup \{-\infty\}$ and consider the space $$\mathcal{W}_{\pi}^s(I, R_0) = \left\{f \in \mathbb{L}^2(I, \pi_X(dx)), ~~ \forall m \geq 1, ~~ \left\|f - P_m(f)\right\|_{\pi}^2 \leq R_0m^{-2s}\right\},$$
where $P_m(f)$ is the orthogonal projection of $f$ onto $\mathcal{S}_m$. We also consider the space $\mathcal{C}^{0,\alpha}(I, \mathbb{R})$ of $\alpha$-H\"older continuous functions on $I$. Suppose that $\sigma^2 \in \mathcal{W}_{\pi}^s(I,R_0) \cap \mathcal{C}^{0,\alpha}(I, \mathbb{R})$ with $\alpha \in [1/2, 1]$. From Corollary~\ref{cor:upper-limit-NAE} and the assumptions therein, and using the bias-variance tradeoff, we obtain, under Assumption~\ref{ass:StrongAssDrift} and for $n\Delta_n^2 = \varepsilon_0\log^4(n)$, $m = \lfloor (n\Delta_n)^{2/(2s+k_0)}\rfloor \in \mathcal{M}_n, ~ s \geq s_0 = k_0 + 1/2$ and
$$\mathbb{E}\left[\left\|\w{\sigma}_m^2 - \sigma_I^2\right\|_{\pi}^2\right] \leq 9R_0m^{-2s} +  C\left(\dfrac{m^{k_0}}{n} + \Delta_n^{2}\right) \leq C\log^4(n)n^{-2s/(2s+k_0)},$$
where $C>0$ is a new constant. Under Assumption~\ref{ass:limit-case-elliptic} with $m = \lfloor (n\Delta_n)^{2/(2s+d_0)}\rfloor, ~ s \geq s_0 = 2d+3$ and $n\Delta_n^2 = 1$, we obtain the following:
$$\mathbb{E}\left[\left\|\w{\sigma}_m^2 - \sigma_I^2\right\|_{\pi}^2\right] \leq 9R_0m^{-2s} +  C\left(\dfrac{m^{d_0}}{n} + \Delta_n^{2}\right) \leq Cn^{-2s/(2s+d_0)},$$
where $C>0$ is another constant. Then, we obtain an explicit rate of order $\log^2(n)n^{-s/(2s+k_0)}$, with $s \geq k_0+1/2$, for exponentially ergodic diffusion processes with an unbounded diffusion coefficient. The derivation of an explicit rate was possible due to the polynomial growth of modulus $\rho$, then of the diffusion coefficient $\sigma$ in the neighborhood of infinity. For polynomially ergodic diffusion processes, we obtain a rate of order $n^{-s/(2s+d_0)}$ with $s \geq d_0 + 3/2$. The difference between the convergence rates obtained in the two ergodicity regimes is mainly due to the assumptions on the diffusion coefficient $\sigma$. In particular, the possibility that $\sigma$ is unbounded naturally leads to slower convergence. Intuitively, this can be understood through the diffusive effect on the dynamics of the process, as a growing $\sigma$ without bound leads to stronger stochastic fluctuations far from the origin (the strength of this effect being quantified by $D$), making excursions toward infinity more likely and thus slowing the return mechanism responsible for ergodic behavior.\\

Note that the above results are established under a stronger regularity assumption on the diffusion coefficient. If we are in a more general setting where $b$ and $\sigma$ are only continuous on $(l, +\infty)$, then we consider the following regression model:    
$$ U_{k\Delta_n} = \sigma^2(X_{k\Delta_n}) + \zeta_{k\Delta_n} + \widetilde{R}_{k\Delta_n}, ~~ k \in \{0, \ldots, n-1\},$$
where and $\zeta_{k\Delta_n} + \widetilde{R}_{k\Delta_n}$ is the error term with
\begin{equation}\label{eq:Error-term2}
    \zeta_{k\Delta_n} := \dfrac{1}{\Delta_n}\sigma^2(X_{k\Delta_n})\left[\left(W_{(k+1)\Delta_n} - W_{k\Delta_n}\right)^2 - \Delta_n\right],
\end{equation}
and $\widetilde{R}_{k\Delta_n} = \widetilde{R}_{k\Delta_n}^1 + \widetilde{R}_{k\Delta_n}^2 + \widetilde{R}_{k\Delta_n}^3$ with
\begin{equation}\label{eq:Residual2}
    \begin{aligned}
        \widetilde{R}_{k\Delta_n}^{(1)} := &~ \dfrac{1}{\Delta_n}\left(\int_{k\Delta_n}^{(k+1)\Delta_n}b(X_s)ds\right)^2 + \dfrac{1}{\Delta_n}\left(\int_{k\Delta_n}^{(k+1)\Delta_n}(\sigma(X_s) - \sigma(X_{k\Delta_n}))dW_s\right)^2,\\
        \widetilde{R}_{k\Delta_n}^{(2)} := &~ \dfrac{2}{\Delta_n}\int_{k\Delta_n}^{(k+1)\Delta_n}\left(b(X_s) - b(X_{k\Delta_n})\right)ds\int_{k\Delta_n}^{(k+1)\Delta_n}\sigma(X_s)dW_s,\\
        \widetilde{R}_{k\Delta_n}^{(3)} := &~ \dfrac{2}{\Delta_n}\sigma(X_{k\Delta_n})\left(W_{(k+1)\Delta} - W_{k\Delta_n}\right)\int_{k\Delta_n}^{(k+1)\Delta_n}(\sigma(X_s) - \sigma(X_{k\Delta_n}))dW_s.
    \end{aligned}
\end{equation}
We then obtain the following result.
\begin{theo}\label{thm:upper-limit2-NAE}
    Grant Assumptions~\ref{ass:drift}, \ref{ass:diffusion}, \ref{ass:NonDegeneracy} and \ref{ass:stationary} and assume that $L = \log(n)$ and $n \rightarrow \infty$. We have the following results.
    \begin{itemize}
        \item [(i)] Suppose that $n\Delta_n^2= \varepsilon_0\log^{4}(n)$ and $d > 8D$. Under Assumption~\ref{ass:StrongAssDrift},  there exists a constant $C>0$ such that for all $m \in \mathcal{M}_n$,
        \begin{equation*}
    \begin{aligned}
        \mathbb{E}\left[\left\|\w{\sigma}_m^2 - \sigma_I^2\right\|_{\pi}^2\right] \leq &~ 9\underset{f \in \mathcal{S}_{m,L}}{\inf}{\left\|f - \sigma_I^2\right\|_{\pi}^2} + C\left(\dfrac{m^{k_0}}{n} + m\log(n)\exp\left(-\frac{\gamma}{8}\sqrt{\varepsilon_0}\log^2(n)\right) + \Delta_n^{\alpha}\right),
    \end{aligned}
\end{equation*}
where $k_0 = D+d+3/2$ and $\alpha \in [1/2,1]$.
        \item [(ii)] Assume that $n\Delta_n^2 = 1$ and $d > 4$. Under Assumption~\ref{ass:limit-case-elliptic}, there exists a constant $C>0$ such that for all $m \in \mathcal{M}_n$,
        $$\mathbb{E}\left[\left\|\w{\sigma}_m^2 - \sigma_I^2\right\|_{\pi}^2\right] \leq 9\underset{f \in \mathcal{S}_{m,L}}{\inf}{\left\|f - \sigma_I^2\right\|_{\pi}^2} + C\left(\dfrac{m^{d_0}}{n} + \dfrac{m\log(n)}{n^{2}} + \Delta_n^{\alpha}\right), ~~ \alpha \in [1/2, 1].$$
    \end{itemize}
\end{theo}
Suppose that $\sigma^2 \in \mathcal{W}_{\pi}^s(I,R_0) \cap \mathcal{C}^{0,\alpha}(I, \mathbb{R})$ with $\alpha \in [1/2,1]$. For exponentially ergodic diffusion processes with $n\Delta_n^2 = \varepsilon_0\log^4(n)$ and $m = \lfloor (n\Delta_n)^{2/(2s+k_0)}\rfloor \in \mathcal{M}_n, ~ s \geq s_0 = d+D+2$, there exists a constant $C>0$ such that
$$\mathbb{E}\left[\left\|\w{\sigma}_m^2 - \sigma_I^2\right\|_{\pi}^2\right] \leq C\log^{2\alpha}(n)n^{-\alpha/2}, ~~ \alpha \in [1/2,1].$$ 
The obtained bound reveals that $\widetilde{R}_{k\Delta_n}$ is the real error term, the other terms being negligible residuals (see Proof of Theorem~\ref{thm:upper-limit2-NAE}). Similarly, for polynomially ergodic diffusion processes with $n\Delta_n^2 = 1$ and $m = \lfloor (n\Delta_n)^{2/(2s+d_0)} \rfloor \in \mathcal{M}_n, ~ s \geq s_0 = 2d+3$, there exists a constant $C>0$ such that 
$$\mathbb{E}\left[\left\|\w{\sigma}_m^2 - \sigma_I^2\right\|_{\pi}^2\right] \leq Cn^{-\alpha/2}, ~~ \alpha \in [1/2,1].$$
In both cases, we obtain convergence rates that depend on the smoothness parameter $\alpha \in [1/2,1]$ of $\mathcal{C}^{0,\alpha}((l,+\infty), \mathbb{R})$. These are slower rates compared to that of Theorem~\ref{thm:upper-limit-NAE} obtained under stronger assumptions on $b$ and $\sigma$. 

In the next section, we proceed to model selection for practical situations and the derivation of risk bounds for adaptive estimators. 

\section{Adaptive estimation - Model selection}
\label{sec:Adaptive}

We proceed to the selection of the optimal dimension of the approximation space, crucial for practical situations. Formally, we consider the adaptive estimator $\widehat{\sigma}_{\widehat{m}}^2$ of $\sigma_I^2$ built from the diffusion path $(X_{k\Delta_n})_{k = 0,\ldots,n}$, the optimal dimension $\widehat{m}$ being selected as follows:
\begin{equation}\label{eq:select-dimension}
    \widehat{m} := \underset{m \in \mathcal{M}_n}{\inf}\left\{\gamma_n(\widehat{\sigma}_m^2) + \mathrm{pen}(m)\right\},
\end{equation}
where for each $m \in \mathcal{M}_n = \left\{1, \ldots, \left\lceil (n\Delta_n)^{2/(2s_0+1)} \right\rceil\right\}$, $\widehat{\sigma}_m^2$ is a non-adaptive estimator of $\sigma_I^2$ built from $(X_{k\Delta_n})_{k=0, \ldots, n}$ and $m \mapsto \mathrm{pen}(m)$ is the penalty function derived using Talagrand's inequality. The following theorem provides risk bounds for adaptive estimators of $\sigma_I^2$ for exponentially and polynomially ergodic diffusion processes.  
\begin{theo}\label{thm:adaptation}
    Grant Assumptions~\ref{ass:drift}, \ref{ass:diffusion}, \ref{ass:NonDegeneracy}, \ref{ass:stationary}, \ref{ass:Regular-Sigma} and suppose that $d> 8D, ~ L = \log(n)$ and $n \rightarrow \infty$. We obtain the following results.
    \begin{enumerate}
        \item Suppose that $n\Delta_n^2 = \varepsilon_0\log^4(n)$ with $\varepsilon_0 \in (0, 10^{-1})$. Under Assumption~\ref{ass:StrongAssDrift}, there exists a constant $C>0$ such that
    \begin{equation*}
        \begin{aligned}
           \mathbb{E}\left[\left\|\w{\sigma}_{\w{m}}^2 - \sigma_I^2\right\|_{\pi}^2\right] \leq &~ 9\underset{m \in \mathcal{M}_n}{\inf}\left\{\underset{f \in \mathcal{S}_{m,L}}{\inf}{\left\|f - \sigma^2\right\|_{\pi}^2} + \mathrm{pen}_1(m)\right\} + C\dfrac{\log^4(n)}{n},
        \end{aligned}
    \end{equation*}
    where $\mathrm{pen}_1: m \mapsto \tau_1 m^{k_0}/n$ is the penalty function with $\tau_1 > 0$ a constant depending on $\sigma$.
    \item Suppose that $n\Delta_n^2 = 1$. Under Assumption~\ref{ass:limit-case-elliptic}, there exists a constant $C>0$ such that
    \begin{equation*}
        \begin{aligned}
           \mathbb{E}\left[\left\|\w{\sigma}_{\w{m}}^2 - \sigma_I^2\right\|_{\pi}^2\right] \leq &~ 9\underset{m \in \mathcal{M}_n}{\inf}\left\{\underset{f \in \mathcal{S}_{m,L}}{\inf}{\left\|f - \sigma^2\right\|_{\pi}^2} + \mathrm{pen}_2(m)\right\} + \dfrac{C}{n},
        \end{aligned}
    \end{equation*}
    where $\mathrm{pen}_2: m \mapsto \tau_2 m^{d_0}/n$ is the penalty function with $\tau_2 > 0$ a constant depending on $\sigma$.
    \end{enumerate}
\end{theo}
The above results still hold for the estimation risks defined by the empirical pseudo-norm $\|.\|_n$. The robustness of Talagrand’s inequality leads to penalty terms having the same order as the corresponding estimation errors. Other methods, such as the chaining technique developed in \cite{baraud2001model}, often lead to penalty terms that match the corresponding variance terms up to an additional logarithmic factor (see, \textit{e.g.}, \cite{denis2021ridge}, \cite{ella2024nonparametric}). The constants of the penalty functions are numerically calibrated in the appendix. Then, from Theorem~\ref{thm:adaptation} and the assumptions therein, for exponentially ergodic diffusion processes, we obtain:
$$\mathbb{E}\left[\left\|\w{\sigma}_{\w{m}}^2 - \sigma_I^2\right\|_{\pi}^2\right] = \mathcal{O}\left(\log^4(n)n^{-2s/(2s+k_0)}\right).$$
Similarly, for polynomially ergodic diffusion processes, the adaptive estimator of $\sigma_I^2$ satisfies
$$\mathbb{E}\left[\left\|\w{\sigma}_{\w{m}}^2 - \sigma_I^2\right\|_{\pi}^2\right] = \mathcal{O}\left(n^{-2s/(2s+d_0)}\right).$$
The result of Theorem~\ref{thm:adaptation} is obtained under a stronger regularity assumption on the diffusion coefficient $\sigma$ (see Assumption~\ref{ass:Regular-Sigma}). Relaxing this regularity assumption leads to the following result.
\begin{theo}\label{thm:adaptation2}
    Grant Assumptions~\ref{ass:drift}, \ref{ass:diffusion}, \ref{ass:NonDegeneracy} and \ref{ass:stationary} and suppose that $L = \log(n)$ and $n \rightarrow \infty$. We obtain the following results.
    \begin{enumerate}
        \item Suppose that $n\Delta_n^2 = \varepsilon_0\log^4(n)$ with $\varepsilon_0 \in (0, 10^{-1})$. Under Assumption~\ref{ass:StrongAssDrift}, there exists a constant $C>0$ such that
    \begin{equation*}
        \begin{aligned}
           \mathbb{E}\left[\left\|\w{\sigma}_{\w{m}}^2 - \sigma_I^2\right\|_{\pi}^2\right] \leq &~ 9\underset{m \in \mathcal{M}_n}{\inf}\left\{\underset{f \in \mathcal{S}_{m,L}}{\inf}{\left\|f - \sigma^2\right\|_{\pi}^2} + \mathrm{pen}_1(m)\right\} + C\dfrac{\log^{2\alpha}(n)}{n^{\alpha/2}},
        \end{aligned}
    \end{equation*}
    where $\mathrm{pen}_1: m \mapsto \tau_1 m^{k_0}/n$ is the penalty function with $\tau_1 > 0$ a constant depending on $\sigma$.
    \item Suppose that $n\Delta_n^2 = 1$. Under Assumption~\ref{ass:limit-case-elliptic}, there exists a constant $C>0$ such that
    \begin{equation*}
        \begin{aligned}
           \mathbb{E}\left[\left\|\w{\sigma}_{\w{m}}^2 - \sigma_I^2\right\|_{\pi}^2\right] \leq &~ 9\underset{m \in \mathcal{M}_n}{\inf}\left\{\underset{f \in \mathcal{S}_{m,L}}{\inf}{\left\|f - \sigma^2\right\|_{\pi}^2} + \mathrm{pen}_2(m)\right\} + \dfrac{C}{n^{\alpha/2}},
        \end{aligned}
    \end{equation*}
    where $\mathrm{pen}_2: m \mapsto \tau_2 m^{d_0}/n$ is the penalty function with $\tau_2 > 0$ a constant depending on $\sigma$.
    \end{enumerate}
\end{theo}
The proofs of these results follow from the arguments used in the proofs of Theorems~\ref{thm:upper-limit2-NAE}. The penalty terms remain the same as that provided by Theorem~\ref{thm:adaptation}, the only contrast coming from the control of the error term $\widetilde{R}_{k\Delta_n}$ given in Equation~\eqref{eq:Residual2}. The next section provides a numerical study based on simulated data.

\section{Numerical study}
\label{sec:Numerical-study}

We investigate the numerical performance of estimators of $\sigma^2$ on the interval $I = (l,+\infty)$, where $l \in \{-\infty,0\}$. The numerical study is performed in \texttt{R}, using the \texttt{orthopolynom} package for the implementation of the Hermite and Laguerre bases. We also use the \texttt{R}-function \texttt{uniroot} or \texttt{optimize} for the selection of $\widehat{\lambda}$ given in Equation~\eqref{eq:Coordinate-b}. 

\subsection{Models and simulation} 

We consider the following models:
\begin{itemize}
    \item Model 1: $b(x) = -d\log(2+x^2)x, ~~ \sigma(x) = \sqrt{(1+x^2)\log(2+x^2)}$, ~~~ $d=10$,
    \item Model 2: $b(x) = -d\left(3+\sqrt{1+x^2}\right)^2x/9\left(1+x^2\right)^2, ~~ \sigma(x) = 1/3+1/\sqrt{1+x^2}$, ~~~ $d=9/2$,
    \item Model 3: $b(x) = 1/x - dx/4, ~~ \sigma(x) = \sqrt{x} + x/2$, ~~~ $d \in \{10, 100\}$,
    \item Model 4: $b(x) = -d(2+\cos x)x/(1+x^2), ~~ \sigma(x) = \sqrt{2+\cos x}$, ~~~ $d=9/2$,
    \item Model 5: $b(x) = 1/x - dx^3, ~~ \sigma(x) = \sqrt{x} + x^2$, ~~~ $d \in \{10, 100\}$.
\end{itemize}
Model 1 satisfies Assumptions~\ref{ass:drift}, \ref{ass:diffusion} with $D=2$, Assumptions~\ref{ass:NonDegeneracy} and \ref{ass:StrongAssDrift}. The process takes values in $\mathbb{R}$ (i.e., $l = -\infty$), is exponentially ergodic and admits a unique invariant distribution. Models 2 and 4 satisfy Assumptions~\ref{ass:drift}, \ref{ass:diffusion}, Assumption~\ref{ass:NonDegeneracy} and \ref{ass:limit-case-elliptic}, and their respective processes are polynomially ergodic and take values in $\mathbb{R}$. Finally, for Models 3 and 5, Assumptions~\ref{ass:diffusion}, \ref{ass:NonDegeneracy} and \ref{ass:StrongAssDrift} are satisfied. Assumption~\ref{ass:drift} is satisfied on the intervals $[a, +\infty)$ with $a>0$ for Model 3, and on any compact interval of $(0, +\infty)$ for Model 5, and the scale function satisfies $S(0) = -\infty$ in both cases. The unique strong solution $X$ with $X_0>0$, is positive and never hits $0$, being non-explosive. We choose $X_0 = 1$ for the numerical experiment. 

For each of the five diffusion models, the data points $X_{k\Delta_n}, ~ k=1, \ldots, n$ are simulated using the Euler scheme, that is,
\begin{equation*}
    X_{k\Delta_n} = X_{(k-1)\Delta_n} + \Delta_nb(X_{(k-1)\Delta_n}) + \sigma(X_{(k-1)\Delta_n})(W_{k\Delta_n} - W_{(k-1)\Delta_n}), ~~ k=1, \ldots, n.
\end{equation*}
The sample paths $(X_{k\Delta_n})_{k = 0, \ldots, n}$ are generated for $n \in \mathcal{N} = \{100000, 1000000\}$. For exponentially ergodic processes, the time step $\Delta_n$ is chosen so that $\Delta_n = \sqrt{\varepsilon_0}\log^2(n)/\sqrt{n}$ with $\varepsilon_0 = 10^{-4}$. We choose $\varepsilon_0$ sufficiently small to compensate for the growth of $\log^2(n)$ so that $\Delta_n$ is close enough to $0$. For polynomially ergodic processes, $\Delta_n = 1/\sqrt{n}$. The Hermite basis is used for Models 1, 2, and 4, and the Laguerre basis for Models 3 and 5. Finally, we set $\mathcal{M} = \{1, \ldots, 10\}$ the set of possible values of the dimension of approximation spaces. 

\subsection{Numerical results} 

The penalty functions $m \in \mathcal{M}_n \mapsto \mathrm{pen}_1(m) = \tau_1 m^{k_0}/n$ and $m \in \mathcal{M}_n \mapsto \mathrm{pen}_2(m) = \tau_2 m/n$ respectively for exponentially ergodic processes and polynomially ergodic processes contain constants $\tau_1>0$ and $\tau_2>0$ that depend on the diffusion coefficients. These constants are numerically calibrated in the appendix, and we obtain $\widehat{\tau}_1 = 10^{-4}$ and $\widehat{\tau}_2 = 100$. Then, the numerical assessment of the risks of projection estimators of $\sigma_I^2$ is performed by repeating $100$ times the following steps:
\begin{enumerate}
    \item [(i)] Simulate $X^{\mathrm{train}} = (X_0, \ldots, X_{n\Delta_n})$ and $X^{\mathrm{test}} = (X_0, \ldots, X_{q\Delta_{q}})$ with $n \in \mathcal{N}$ and $q= 10000$.
    \item [(ii)] For each $m \in \mathcal{M}$, compute $\widehat{\sigma}_m^2$ from the diffusion path $X^{\mathrm{train}}$ using Equations~\eqref{eq:Coordinate-a} and \eqref{eq:Coordinate-b}.
    \item [(iii)] Select the optimal dimension $\widehat{m}$ in the set $\mathcal{M}$ using Equation~\eqref{eq:select-dimension}.
    \item [(iv)] From the path $X^{\mathrm{test}}$, compute the empirical estimation risk $\|\widehat{\sigma}_{\widehat{m}}^2 - \sigma_I^2\|_n^2$.
\end{enumerate}
The risks of estimation are deduced from the average value of $\|\widehat{\sigma}_{\widehat{m}}^2 - \sigma_I^2\|_n^2$ over the $100$ repetitions. 

\begin{table}
\centering
\begin{tabular}{lccc}
\hline
 & $n = 100 000$ & $n = 1000 000$ \\
\hline
Model 1 & $0.0014 ~~ (0.00057)$ & $0.00098 ~~ (0.00096)$ \\
Model 2 & $0.0016 ~~ (0.00051)$ & $0.00008 ~~ (0.00004)$ \\
Model 3, ~ $d=100$ & $0.017 ~~ (0.00086)$ & $0.016 ~~ (0.0024)$ \\
Model 4 & $0.0053 ~~ (0.0037)$ & $0.00043 ~~ (0.00022)$ \\
Model 5, ~ $d=100$ & $0.033 ~~ (0.0014)$ & $0.031 ~~ (0.0019)$ \\
\hline
\end{tabular}
\caption{Numerical assessment of the risks of estimation of adaptive estimators of $\sigma_I^2$ on $I = (l, +\infty)$ for each of the five models.}
\label{tab:estimation_risk}
\end{table}

The numerical results are reported in Table~\ref{tab:estimation_risk}. We have average risks of estimation for each model, with the standard derivation in parentheses. From Table~\ref{tab:estimation_risk}, we note that Models 3 and 5 with the respective unbounded diffusion coefficients $\sigma(x) = \sqrt{x} + x/2$ and $\sigma(x) = \sqrt{x} + x^2$ are more complex compared to other models. The corresponding theoretical rates are of order $\log(n)n^{-1/8}$ provided by Theorem~\ref{thm:upper-limit2-NAE} since Assumption~\ref{ass:Regular-Sigma} is not satisfied in these two cases and $\alpha = 1/2$. The diffusive regime seems to be dominant with respect to the dissipative effect induced by the drift coefficient. For Model 1, with the non-Lipschitz diffusion coefficient $\sigma(x) = \sqrt{(1+x^2)\log(2+x^2)}$ satisfying Assumption~\ref{ass:Regular-Sigma}, which implies a faster theoretical rate compared to Models 3 and 5, resulting in an improved numerical performance. The same analysis is performed for Models 2 and 4, where Model 4 has a continuous sinusoidal diffusion coefficient satisfying Assumption~\ref{ass:Regular-Sigma} with $\theta = 0$, and Model 2 has a Lipschitz diffusion coefficient that also satisfies Assumption~\ref{ass:Regular-Sigma} with $\theta=0$. 

\begin{table}
\centering
\begin{tabular}{lcccc}
\hline
 & $\varepsilon_0 = 0.01$ & $\varepsilon_0 = 0.001$ & $\varepsilon_0 = 0.0001$ & $\varepsilon_0 = 0.0001$ \\
\hline
Model 1 & $0.10 ~~ (0.0080)$ & $0.0081 ~~ (0.0012)$ & $0.0013 ~~ (0.0013)$ & $0.0011 ~~ (0.0011)$ \\
Model 3 & xxx ~~ (xxx) & xxx ~~ (xxx) & $0.017 ~~ (0.00089)$ & $0.017 ~~ (0.0029)$ \\
Model 5 & xxx ~~ (xxx) & xxx ~~ (xxx) & $0.033 ~~ (0.0012)$ & $0.031 ~~ (0.0026)$ \\
\hline
\end{tabular}
\caption{Impact of $\varepsilon_0$ on the quality of estimation of $\sigma_I^2$ on $I = (l, +\infty)$ with $n=100000$. }
\label{tab:varepsilon_impact}
\end{table}

For Models 1, 3, and 5, we chose $\varepsilon_0 = 10^{-4}$, and we can see from Table~\ref{tab:varepsilon_impact} that the value of this numerical parameter is highly sensible as it can dramatically affect the performance of the estimator. The results for Models 3 and 5 confirm the idea that $\varepsilon_0$ should be close enough to $0$ since choosing $\varepsilon \in \{0.01, 0.001\}$ leads to infinite values in the empirical matrix $\widehat{\Phi}_m$.  

\begin{table}
\centering
\begin{tabular}{lccccc}
\hline
 & $d=1$ & $d=5$ & $d=10$ & $d=50$ & d=100 \\
\hline
Model 1 & $r_1 ~~ (s_1)$ & $0.0030 ~~ (0.0054)$ & $0.0013 ~~ (0.00059)$ & $0.0043 ~~ (0.00039)$ & $0.015 ~~ (0.00086)$ \\
Model 2 & $0.06 ~~ (0.23)$ & $0.0014 ~~ (0.0013)$ & $0.0018 ~~ (0.00055)$ & $0.19 ~~ (0.015)$ & $1.20 ~~ (0.046)$ \\
Model 3 & $r_3 ~~ (s_3)$ & $11.19 ~~ (19.80)$ & $0.85 ~~ (1.34)$ & $0.060 ~~ (0.0059)$ & $0.016 ~~ (0.00083)$ \\
Model 4 & $r_4 ~~ (s_4)$ & $0.0053 ~~ (0.0021)$ & $0.033 ~~ (0.0064)$ & $2.095 ~~ (0.096)$ & $17.96 ~~ (0.55)$ \\
Model 5 & xxx ~~ (xxx) & xxx ~~ (xxx) & $0.13 ~~ (0.11)$ & $0.082 ~~ (0.0039)$ & $0.033 ~~ (0.0012)$ \\
\hline
\end{tabular}
\caption{Impact of values of $d$ on the numerical performance of adaptive estimators of $\sigma_I^2$ on $I = (l, +\infty)$ with $n=100000$, where $r_1 = 781.69$, $s_1 = 7295.84$, $r_3 \simeq 1.4 \times 10^4$, $s_3 \simeq 4.4 \times 10^4$, $r_4 \simeq 2.3 \times 10^6$ and $s_4 \simeq 2.3 \times 10^7$.}
\label{tab:d_impact}
\end{table}

Note, however, that the performance of estimators of $\sigma_I^2$ is influenced not only by the nature of the diffusion coefficient, but also by the parameter $d$, which quantifies the balance or imbalance between the dissipative effect and the diffusive behavior of the process. From the numerical results reported in Table~\ref{tab:d_impact}, when $d=1$, which leads to a weak dissipative effect, the estimation risk explodes for Models 1, 3 and 4 whose respective diffusion coefficients are either unbounded (Models 1 and 3), or present a particularly disturbing behavior (Model 4 with sinusoidal $\sigma$ taking values in $[1,\sqrt{3}]$). The case of Model 2 is an example of a stable regime characterized by a weak dissipative effect coupled with a weakly diffusive regime ($\sigma(x) \in [1/3, 4/3]$). A theoretical proof of our main results under this condition goes beyond the scope of the present paper as the parameter $d$ is required to be large enough for the definition of moments of an order sufficiently high (See proofs of Theorems~\ref{thm:upper-limit-NAE}, \ref{thm:upper-limit2-NAE} and \ref{thm:adaptation}). For $d \in \{5,10\}$ the regime is generally stable except for Models 3 and 5, which require a stronger dissipative effect with larger $d$ ($d \in \{50, 100\}$), unless the deterioration of the estimation risks is due to the presence of outliers. Finally, Figure~\ref{fig:courbes} displays the projection estimators of $\sigma_I^2$ for each of the five models. In each case, we observe a good quality of estimation in compact intervals. This phenomenon is the result of the dissipative effect induced by the drift function under Assumption~\ref{ass:StrongAssDrift} or Assumption~\ref{ass:limit-case-elliptic}. In fact, we observe that for Model 1, more than $99\%$ of data points lie within the compact interval $[-1, 1]$, which explains the poor quality of estimation outside $[-1, 1]$. A similar observation is made for Models 2 with an average of $95\%$ of data points within $[-1,1]$, Models 3 (more than $98\%$ within $(0,2]$), Model 4 (more than $96\%$ of data points within $[-1, 1]$) and Models 5 (more than $98\%$ within $(0,2]$).For Models 3 and 5, the estimators are built with d=10 for simplicity, since choosing d=100 leads to the appearance of outliers that must be removed before representing the estimators. However, these outliers were not removed for the computation of the estimation risks reported in Tables~\ref{tab:estimation_risk}, \ref{tab:varepsilon_impact} and \ref{tab:d_impact}. 
 \begin{figure}
\centering
\includegraphics[width=0.7\textwidth]{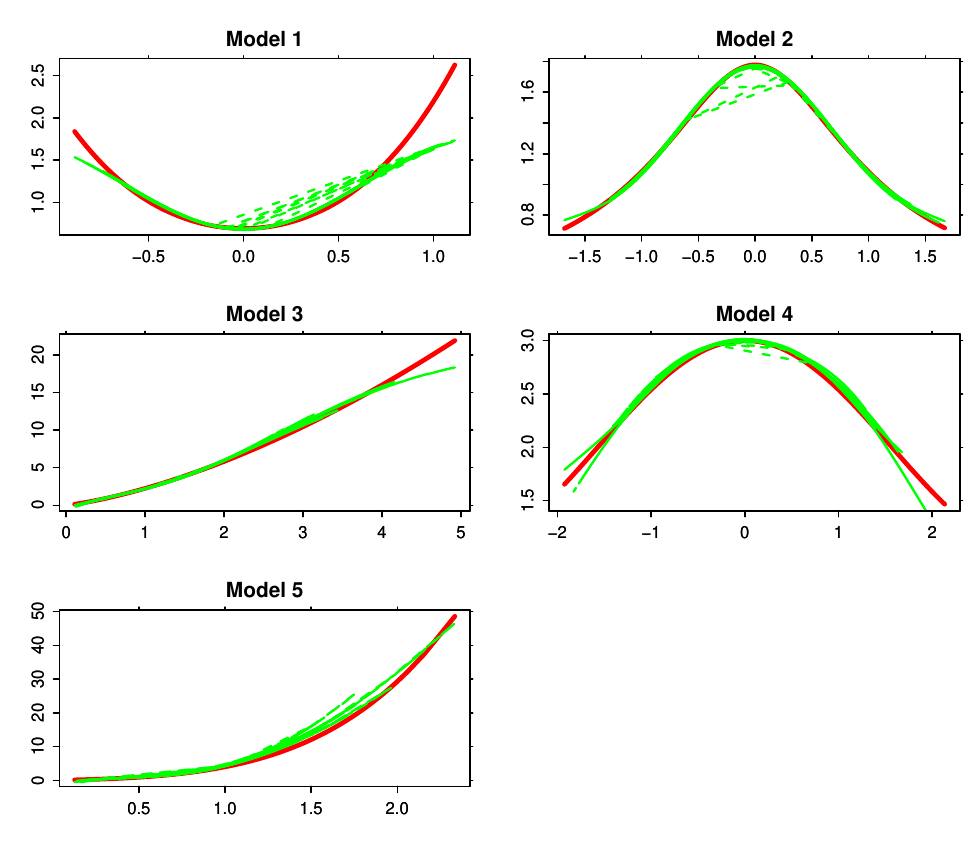}
\caption{Curves of the diffusion coefficients (in red) and a set of $10$ estimators (in green) built from a sample path of size $n+1$ with $n=1000000$.}
\label{fig:courbes}
\end{figure}

\section{Proofs}
\label{sec:proofs}

In this section, Assumptions~\ref{ass:drift}, \ref{ass:diffusion}, \ref{ass:NonDegeneracy}, \ref{ass:stationary} are always satisfied. As we consider both exponential ergodicity and polynomial ergodicity, we will just mention, in most cases,  Assumption~\ref{ass:StrongAssDrift} for the exponential ergodicity and Assumption~\ref{ass:limit-case-elliptic} for the polynomial one. In addition, constants are generally denoted by $C,c>0$ and their values may change from line to line.

\subsection{Proof of Proposition~\ref{prop:Exp-Holder}}

\begin{proof}
Throughout the proof, we set \( l = -\infty \), so that \( (l,+\infty) = \mathbb{R} \).  The results obtained extend readily to any interval \( (l,+\infty) \) with \( l \in \mathbb{R} \).
Dealing with stochastic differential equations with unbounded diffusion coefficients requires one to prove that Lemma 2 in \cite{denis2024nonparametric} and Lemma 2.2 in \cite{ella2024nonparametric} still hold. More precisely, for all $p \in (0, d/D]$ and for all $h \in (0,1/2)$,
\begin{equation}\label{eq:prop1-0}
    \begin{aligned}
        \forall~T>0, ~~ \mathbb{E}\left[\underset{t \in [T,T + h]}{\sup}{|X_t|^{2p}}\right] < \infty ~~ \mathrm{and} ~~ \exists~C_p>0, ~ \forall~t\geq 0: ~ \mathbb{E}\left[\left|X_{t+h} - X_t\right|^{2p}\right] \leq C_ph^{p}.
    \end{aligned} 
\end{equation}
Under Assumption~\ref{ass:drift}, since $\kappa: (0, \infty) \rightarrow (0, \infty)$ is a positive increasing concave function, for all $x>0, ~ \kappa(x) \leq \kappa(1)x$. Consequently, there exists a constant $C>0$ such that for all $x \in \mathbb{R}, ~ |b(x)| \leq C(1 + \kappa(|x|)) \leq C(1 + \kappa(1)|x|)$. Under Assumptions~\ref{ass:drift} and \ref{ass:diffusion} and from \cite{ella2024nonparametric}, combining the H\"older inequality with the Doob inequality and the Burkholder-Davis-Goundy inequality, for all $p \in (0, d/D)$, with $d > 8D$, and for all $(T, h) \in (0, \infty) \times (0,1/2)$, there exists $C_{p} > 0$ depending on $p$ such that
\begin{equation}\label{eq:prop1-1}
    \begin{aligned}
        \mathbb{E}\left[\underset{t \in [T,T+h]}{\sup}{|X_t|^{2p}}\right] \leq &~ C_{p}\left(\mathbb{E}\left[|X_0|^{2p}\right] + h^{2p-1}\int_{T}^{T+h}(1 + \mathbb{E}\left[|X_s|^{2p}\right])ds + \mathbb{E}\left[\left(\int_{T}^{T+h}\sigma^{2}(X_s)ds\right)^p\right]\right)\\
        \leq &~ C_{p}\left(\mathbb{E}\left[|X_0|^{2p}\right] + h^{2p-1}\int_{T}^{T+h}(1 + \mathbb{E}\left[|X_s|^{2p}\right])ds + h^{p-1}\int_{T}^{T+h}\mathbb{E}\left[\sigma^{2p}(X_s)\right]ds\right)\\
        \leq &~ C_{p}\left(\mathbb{E}\left[|X_0|^{2p}\right] + \int_{T}^{T+h}(1 + \mathbb{E}\left[|X_s|^{2p}\right])ds + \int_{T}^{T+h}(1 + \mathbb{E}\left[\rho^{2p}(|X_s|)\right])ds\right),
    \end{aligned}
\end{equation}
since $h < 1$. Finally, under Assumption~\ref{ass:NonDegeneracy}, $\dfrac{xb(x)}{\sigma^2(x)}\longrightarrow -d ~~ \mathrm{as} ~~ |x| \rightarrow \infty,$
with $d > 8D$, there exists $A \geq R > 0$ such that for all $|x| \geq A, ~ \left||xb(x)/\sigma^2(x)| - d\right| \leq 1/2$, and for all $x \notin (-A,A), ~ \sigma^2(x) \geq c|x|^{q+1}$ where $c>0$ is a constant, with $q \geq 1$ and $R>0$ given in Assumption~\ref{ass:StrongAssDrift}. Then, under Assumption~\ref{ass:StrongAssDrift}, for all $s \geq 0$ and for all $p \in (0, d/D)$,
\begin{equation}\label{eq:prop1-2}
    \begin{aligned}
        \mathbb{E}\left[|X_s|^{2p}\right] = &~ \int_{(-\infty, -A)}|x|^{2p}\pi_X(dx) + \int_{[-A, A]}|x|^{2p}\pi_X(dx) + \int_{(A, +\infty)}|x|^{2p}\pi_X(dx).
    \end{aligned}
\end{equation}
Then, by Equation~\eqref{eq:InvariantDensity} and for all $p \in (0, d/D)$
\begin{equation}\label{eq:prop1-3}
    \begin{aligned}
        \int_{(A, +\infty)}|x|^{2p}\pi_X(dx) = &~ \int_{(A, +\infty)}\dfrac{2x^{2p}dx}{M\sigma^2(x)}\exp\left(2\int_{0}^{x}\dfrac{b(z)}{\sigma^2(z)}dz\right)\\
        \leq &~ \int_{(A, +\infty)}\dfrac{2C_Ax^{2p}dx}{M\sigma^2(x)}\exp\left(-2\int_{A}^{x}\dfrac{1}{z}.\left|\dfrac{zb(z)}{\sigma^2(z)}\right|dz\right)\\
        \leq &~ \int_{(A, +\infty)}\dfrac{2C_Ax^{2p}dx}{M\sigma^2(x)}\exp\left((-2d+1)\int_{A}^{x}\dfrac{dz}{z}\right)\\
        \leq &~ \int_{(A, +\infty)}\dfrac{2C_AA^{2d-1}dx}{Mx^{2d-2p+q}} < \infty,
    \end{aligned}
\end{equation}
where $C_A = \exp\left(2\int_{[0,A]}|b(u)|/\sigma^2(u)du\right) < \infty$ and $2d - 2p + q > 1$. On the other hand, for all $p \in (0, d/D)$,
\begin{equation}\label{eq:prop1-4}
    \begin{aligned}
        \int_{(-\infty, -A)}|x|^{2p}\pi_X(dx) = &~ \int_{(A, +\infty)}\dfrac{2x^{2p}dx}{M\sigma^2(-x)}\exp\left(-2\int_{0}^{x}\dfrac{b(-z)}{\sigma^2(-z)}dz\right)\\
        \leq &~ \int_{(A, +\infty)}\dfrac{2\widetilde{C}_Ax^{2p}dx}{M\sigma^2(-x)}\exp\left(-2\int_{A}^{x}\dfrac{1}{z}.\left|\dfrac{zb(-z)}{\sigma^2(-z)}\right|dz\right)\\
        \leq &~ \int_{(A, +\infty)}\dfrac{2\widetilde{C}_Ax^{2p}dx}{M\sigma^2(-x)}\exp\left((-2d+1)\int_{A}^{x}\dfrac{dz}{z}\right)\\
        \leq &~ \int_{(A, +\infty)}\dfrac{2\widetilde{C}_AA^{2d-1}dx}{Mx^{2d-2p+q}} < \infty,
    \end{aligned}
\end{equation}
where $\widetilde{C}_A = \exp\left(2\int_{[0,A]}|b(-u)|/\sigma^2(-u)du\right) < \infty$. We deduce from Equations~\eqref{eq:prop1-4}, \eqref{eq:prop1-3} and \eqref{eq:prop1-2},
\begin{equation}\label{eq:prop1-5}
    \forall~s\geq 0, ~ \forall~p \in (0, d/D), ~~ \mathbb{E}\left[|X_s|^{2p}\right] < \infty.
\end{equation}
Moreover, under Assumption~\ref{ass:diffusion}, there exists $B > 0$ such that for all $p \in (0, d/D)$ and for all $s \geq 0$,
\begin{equation*}
    \begin{aligned}
        \mathbb{E}\left[\rho^{2p}(|X_s|)\right] \leq &~ \mathbb{E}\left[\rho^{2p}(|X_s|)\mathds{1}_{|X_s|\leq B}\right] + C\mathbb{E}\left[|X_s|^{2pD}\right]\\
        \leq &~ \rho^{2p}(B) + \int_{[-A, A]}|x|^{2pD}\pi_X(dx) + \int_{(A, +\infty)}\dfrac{2C_AA^{2d-1}dx}{Mx^{2d-2pD+q}} + \int_{(A, +\infty)}\dfrac{2\widetilde{C}_AA^{2d-1}dx}{Mx^{2d-2pD+q}}.
    \end{aligned}
\end{equation*}
Since $q \geq 1$, for all $p \in (0, d/D)$, we have $2d-2pD+q > 1$ and obtain the following: 
\begin{equation}\label{eq:prop1-6}
    \forall~p \in (0, d/D), ~ \forall~s \geq 0, ~~ \mathbb{E}\left[\rho^{2p}(|X_s|)\right] < \infty.
\end{equation}
Thus, from Equations~\eqref{eq:prop1-6}, \eqref{eq:prop1-5} and \eqref{eq:prop1-1}, we deduce that for all $p \in (0, d/D)$, 
\begin{equation}\label{eq:prop1-7}
    \forall~T>0, ~ \forall~h \in (0,1/2), ~~\mathbb{E}\left[\underset{t \in [T,T+h]}{\sup}{|X_t|^{2p}}\right] < \infty.
\end{equation}
Focusing on the second inequality in Equation~\eqref{eq:prop1-0}, from \cite{denis2024nonparametric}, \textit{Proof of Lemma 2} and Equation~\eqref{eq:prop1-7}, for all $p \in (0, d/D)$ and for all $t\geq 0$ and $h \in (0,1/2)$, there exists $C_p > 0$ such that
\begin{equation}\label{eq:prop1-elliptic}
    \begin{aligned}
       \mathbb{E}\left[\left|X_{t+h} - X_t\right|^{2p}\right] \leq &~ C_p\left(h^{2p} + \mathbb{E}\left[\left(\int_{t}^{t+h}\sigma^2(X_s)ds\right)^p\right]\right).   
    \end{aligned}
\end{equation}
Using the H\"older inequality and under Assumption~\ref{ass:diffusion} and Equation~\eqref{eq:prop1-6}, we obtain
\begin{equation}\label{eq:prop1-8}
    \begin{aligned}
        \mathbb{E}\left[\left|X_{t+h} - X_t\right|^{2p}\right] \leq &~ C_p\left(h^{2p} + h^{p-1}\int_{t}^{t+h}\mathbb{E}\left[\rho^{2p}(|X_s|)\right]ds\right) \leq C_ph^p.
    \end{aligned}
\end{equation}
Focusing on polynomially ergodic processes, it suffices to note that under Assumption~\ref{ass:limit-case-elliptic}, from the first line of Equation~\ref{eq:prop1-1} and from Equations~\eqref{eq:prop1-3} and \eqref{eq:prop1-4}, we obtain for all $p \in (0, d-1)$ and $s \geq 0$,
$$ \mathbb{E}\left[|X_s|^{2p}\right] \leq \int_{[-A, A]}|x|^{2p}\pi_X(dx) + \int_{(A, +\infty)}\dfrac{2(C_A+\widetilde{C}_A)A^{2d-1}dx}{M\sigma_{+}^2x^{2d-2p-1}} < \infty.$$
Then, Equation~\eqref{eq:prop1-7} holds for all $p \in (0, d-1)$, with $d > 4$. In addition, from Equation~\eqref{eq:prop1-elliptic} and Assumption~\ref{ass:limit-case-elliptic}, we have for all $p \in (0, d-1)$,
$$ \forall~(t,h) \in (0, \infty) \times (0, 1/2), ~~ \mathbb{E}\left[\left|X_{t+h} - X_t\right|^{2p}\right] = \mathrm{O}(h^p),$$
We can now focus on the main results and we start with the drift coefficient. Under Assumptions~\ref{ass:drift} and \ref{ass:diffusion}, we deduce that for all $t \geq 0$ and for all $h \in (0,1/2)$,
 \begin{equation*}
     \begin{aligned}
         \mathbb{E}\left[\left|b(X_{t+h}) - b(X_t)\right|^{2p}\right] \leq &~ \mathbb{E}\left[\kappa^{2p}(|X_{t+h} - X_t|)\right] \leq \kappa^{2p}(1)\mathbb{E}\left[\left|X_{t+h} - X_t\right|^{2p}\right] \leq C_p\kappa^{2p}(1)h^p,
     \end{aligned}
 \end{equation*}
 where $p \in (0, d/D)$ with $d > 8D$ under Assumption~\ref{ass:StrongAssDrift}, and $p \in (0, d-1)$ with $d>4$ under Assumption~\ref{ass:limit-case-elliptic}.
 
 Focusing on the diffusion coefficient, from Equation~\eqref{eq:prop1-8} with $p=2$ and the Kolmogorov-\v{C}entsov theorem (see \cite{karatzas2014brownian}, \textit{Chapter 2, Theorem 2.8, p.53}), the diffusion process $(X_t)_{t \geq 0}$ admits a modification that is almost surely continuous. We deduce that for all $\eta \in (0,1/2)$ there exists $h_0 \in (0,1/2)$ such that for all $t \geq 0$ and for all $h \leq h_0, ~ |X_{t+h} - X_t| \leq \eta$ a.s. Then, under Assumption~\ref{ass:diffusion}, there exist $C>0$ and $h_0 \in (0, 1/2)$ such that for all $h \in (0, h_0)$,
$$ \rho(|X_{t+h} - X_t|) \leq C|X_{t+h} - X_t|^{\alpha} ~~ a.s., ~~ \alpha \in [1/2, 1].$$ 
 It follows that under Assumptions~\ref{ass:drift}, \ref{ass:diffusion} and from Equation~\eqref{eq:prop1-8}, for all $p \in (0, d/D)$ and for all $h \in (0, h_0)$,
 \begin{equation*}
     \begin{aligned}
         \mathbb{E}\left[\left|\sigma(X_{t+h}) - \sigma(X_t)\right|^{2p}\right] \leq &~ \mathbb{E}\left[\rho^{2p}(|X_{t+h} - X_t|)\right] \leq C^{2p}\mathbb{E}\left[|X_{t+h} - X_t|^{2\alpha p}\right] \leq C_ph^{\alpha p},
     \end{aligned}
 \end{equation*}
 where $C_p>0$ is a new constant. A similar reasoning leads to the same result for polynomially ergodic processes with $p \in (0, d-1)$ and $d > 4$, which concludes the proof. 
\end{proof}

\subsection{Proof of Proposition~\ref{prop:operator-norm-identity}}

\begin{proof}
For $I \in \{\mathbb{R}, [0, +\infty)\}$ and for all $u \in \mathbb{R}^m$ such that $\|u\|_{2,m} = 1$, we have the following:
    \begin{equation*}
        \begin{aligned}
            u^{\prime}\Psi_m u = &~ \sum_{i = 0}^{m-1}\sum_{j=0}^{m-1}u_iu_j\int_{I}\phi_i(x)\phi_j(x)\pi_X(dx) =  \int_{I}\left(\sum_{j=0}^{m-1}u_j\phi_j(x)\right)^2\pi_X(dx).
        \end{aligned}
    \end{equation*}
    
\subsubsection*{Case of the Hermite basis}

Focusing on the Hermite basis, the estimation interval is $I = \mathbb{R}$ and for all $j \in [\![0,m-1]\!], ~ \phi_j = h_j$. Moreover, from \cite{thangavelu1993lectures}, \textit{Chapter 1, Lemma 1.5.1, page 26} (see also \cite{askey1965mean}), there exist constants $c_0, \gamma > 0$ such that for all $x \notin [-b_m, b_m]$ with $b_m=\sqrt{2(2m + 1)}$, $h_j(x) \leq c_0\exp(-\gamma x^2)$. Then, for all $u \in \mathbb{R}^{m}$ such that $\|u\|_{2,m} = 1$,
\begin{equation*}
    \begin{aligned}
        u^{\prime}\Psi_m u \geq &~ \int_{|x| \leq b_m}\left(\sum_{j=0}^{m-1}u_j\phi_j(x)\right)^2\pi_X(dx) \geq \underset{|y| \leq b_m}{\inf}{\pi_X(y)}\int_{|x| \leq b_m}\left(\sum_{j=0}^{m-1}u_jh_j(x)\right)^2dx\\
        = &~ \underset{|y| \leq b_m}{\inf}{\pi_X(y)}\left[\int_{\mathbb{R}}\left(\sum_{j=0}^{m-1}u_jh_j(x)\right)^2dx - \int_{|x| > b_m}\left(\sum_{j=0}^{m - 1}u_jh_j(x)\right)^2dx\right]\\
        \geq &~ \underset{|y| \leq b_m}{\inf}{\pi_X(y)}\left[1 - \sum_{j=0}^{m-1}\int_{|x| > b_m}h_j^2(x)dx\right] \geq \underset{|y| \leq b_m}{\inf}{\pi_X(y)}\left[1 - c_0^2m\int_{b_m}^{+\infty}\exp(-2\gamma x^2)dx\right].
    \end{aligned}
\end{equation*}
Since $\int_{b_m}^{+\infty}\exp(-2\gamma x^2)dx \leq \dfrac{1}{4\gamma b_m}\exp\left(-2\gamma b_m^2\right)$, for $b_m = \sqrt{2(2m+1)}$ for $m$ large enough, we obtain
\begin{equation}\label{eq:prop2-1}
    \begin{aligned}
         u^{\prime}\Psi_m u \geq &~ \dfrac{1}{2}\underset{|y| \leq \sqrt{2(2m + 1)}}{\inf}{\pi_X(y)} > 0,
    \end{aligned}
\end{equation}
and the Gram matrix $\Psi_m$ is invertible. In addition, we have, on the one hand,
\begin{equation}\label{eq:prop2-2}
    \mathcal{L}(m) = \underset{x \in \mathbb{R}}{\sup}{\sum_{j=0}^{m-1}}h_j^2(x) \leq \underset{j \in [\![0,m-1]\!]}{\max}{\left\|h_j\right\|_{\infty}^2} m.
\end{equation}
On the other hand, from Assumption~\ref{ass:NonDegeneracy} and Equation~\eqref{eq:InvariantDensity}, there exists $A > 1$ such that $|x| \geq A \Rightarrow \left||xb(x)|/\sigma^2(x) - d\right| \leq 1/2$ and we obtain for all $x \in (A, +\infty)$,
\begin{equation*}
    \begin{aligned}
       \pi_X(x) = &~ \dfrac{2C_A}{M\sigma^2(x)}\exp\left(-2\int_{A}^{x}\dfrac{1}{z}\dfrac{|zb(z)|}{\sigma^2(z)}dz\right) \geq \dfrac{2C_A}{M\sigma^2(x)}\exp\left((-2d-1)\int_{A}^{x}\dfrac{dz}{z}\right) = \dfrac{2C_A}{M\sigma^2(x)}\left(\dfrac{A}{x}\right)^{2d+1},
    \end{aligned}
\end{equation*}
and for all $x \in (-\infty, -A)$,
\begin{equation*}
    \begin{aligned}
        \pi_X(x) = &~ \dfrac{2\widetilde{C}_A}{M\sigma^2(x)}\exp\left(-2\int_{A}^{-x}\dfrac{1}{z}.\left|\dfrac{zb(-z)}{\sigma^2(-z)}\right|dz\right)\\
        \geq &~ \dfrac{2\widetilde{C}_A}{M\sigma^2(x)}\exp\left(-(2d+1)\int_{A}^{-x}\dfrac{dz}{z}\right) = \dfrac{2\widetilde{C}_A}{M\sigma^2(x)}\left(\dfrac{A}{-x}\right)^{2d+1},
    \end{aligned}
\end{equation*}
where 
$$C_A = \exp\left(2\int_{[0,A]}|b(u)|/\sigma^2(u)du\right) \in (0, \infty) ~~ \mathrm{and} ~~ \widetilde{C}_A = \exp\left(2\int_{[0,A]}|b(-u)|/\sigma^2(-u)du\right) \in (0, \infty).$$ 
We deduce that for all $x \in \mathbb{R}$ such that $|x| \geq A$,
\begin{equation}\label{eq:inf-densite}
    \pi_X(x) \geq \dfrac{2\min(C_A, \widetilde{C}_A)}{M\sigma^2(x)}\left(\dfrac{A}{|x|}\right)^{2d+1}.
\end{equation}
For exponentially ergodic diffusion processes, under Assumption~\ref{ass:diffusion}, from Equation~\eqref{eq:inf-densite} and for $m$ large enough, there exists a constant $C>0$ such that
\begin{equation}\label{eq:prop2-3}
    \begin{aligned}
        \underset{|y| \leq \sqrt{2(2m + 1)}}{\inf}{\pi_X(y)} \geq &~ \dfrac{2\min(C_A, \widetilde{C}_A)}{M\left[\sigma^2(\sqrt{4m+2}) + \sigma^2(-\sqrt{4m+2})\right]}\left(\dfrac{A}{\sqrt{4m+2}}\right)^{2d+1}\\
        \geq &~ \dfrac{2\min(C_A, \widetilde{C}_A)}{MCm^{D}}\left(\dfrac{A}{\sqrt{4m+2}}\right)^{2d+1}.
    \end{aligned}
\end{equation}
Finally, from Equations~\eqref{eq:prop2-3}, \eqref{eq:prop2-2} and \eqref{eq:prop2-1}, there exist constants $C,C^{\prime} > 0$ depending on $b$, $\sigma^2$, $D$, and $d$ such that for $m$ large enough,
\begin{equation}\label{eq:Gram1}
    \left\|\Psi_m^{-1}\right\|_{\mathrm{op}} \leq C^{\prime}m^{D+d+1/2} ~~ \mathrm{and} ~~ \mathcal{L}(m)\left\|\Psi_m^{-1}\right\|_{\mathrm{op}} \leq C\left(\underset{j \in [\![0,m-1]\!]}{\max}{\left\|h_j\right\|_{\infty}^2}\right)m^{D+d+3/2}.
\end{equation}
For polynomially ergodic processes, under Assumption~\ref{ass:limit-case-elliptic}, and from Equation~\eqref{eq:inf-densite}, we obtain
\begin{equation*}
    \begin{aligned}
        \underset{|y| \leq \sqrt{2(2m + 1)}}{\inf}{\pi_X(y)} 
        \geq &~ \dfrac{2\min(C_A, \widetilde{C}_A)}{M\sigma_{+}^2}\left(\dfrac{A}{\sqrt{4m+2}}\right)^{2d+1},
    \end{aligned}
\end{equation*}
which implies, for $m$ large enough, there exist constants $C,C^{\prime} > 0$ such that
\begin{equation}\label{eq:Gram1.2}
    \left\|\Psi_m^{-1}\right\|_{\mathrm{op}} \leq C^{\prime}m^{d+1/2} ~~ \mathrm{and} ~~ \mathcal{L}(m)\left\|\Psi_m^{-1}\right\|_{\mathrm{op}} \leq C\left(\underset{j \in [\![0,m-1]\!]}{\max}{\left\|h_j\right\|_{\infty}^2}\right)m^{d+3/2}.
\end{equation}

\subsubsection*{Case of the Laguerre basis}

In this case, the estimation interval is $I = [0, +\infty)$ and for all $j \in [\![0,m-1]\!]$, $\phi_j = \ell_j$. From \cite{askey1965mean} or \cite{thangavelu1993lectures}, \textit{Chapter 1, Lemma 1.5.3, page 27}, there exist constants $c_0, \gamma > 0$ such that for all $x \geq 3(2m + 1)$, $\ell_j(x) \leq c_0\exp(-\gamma x)$. Setting $a_m = \sqrt{3(2m + 1)}$ and using a similar reasoning, we obtain
\begin{equation*}
    \begin{aligned}
        u^{\prime}\Psi_m u \geq &~ \underset{|y| \leq a_m}{\inf}{\pi_X(y)}\left[1 - c_0^2m\int_{a_m}^{+\infty}\exp(-2\gamma x)dx\right] = \underset{|y| \leq a_m}{\inf}{\pi_X(y)}\left[1 - \dfrac{c_0^2m}{2\gamma}\exp\left(-\gamma a_m\right)\right].
    \end{aligned}
\end{equation*}
Then, for $a_m = \sqrt{3(2m+1)}$ and $m$ large enough, we have $1 - \dfrac{c_0^2m}{2\gamma}\exp\left(-\gamma\sqrt{3(2m+1)}\right) \geq \dfrac{1}{2}$, 
and under Assumption~\ref{ass:diffusion}, from Equation~\eqref{eq:inf-densite} and for $m$ large enough,
\begin{equation*}
    \begin{aligned}
        \underset{|y| \leq \sqrt{3(2m + 1)}}{\inf}{\pi_X(y)} \geq &~ \dfrac{2\min(C_A, \widetilde{C}_A)}{M\left[\sigma^2(\sqrt{6m+3}) + \sigma^2(-\sqrt{6m+3})\right]}\left(\dfrac{A}{\sqrt{6m+3}}\right)^{2d+1}\\
        \geq &~ \dfrac{2\min(C_A, \widetilde{C}_A)}{MCm^{D}}\left(\dfrac{A^2}{6m+3}\right)^{d+1/2}.
    \end{aligned}
\end{equation*}
We find that there exist constants $C, C^{\prime}>0$ depending on $b, \sigma^2, D$ and $d$ such that
\begin{equation}\label{eq:Gram2}
     \left\|\Psi_m^{-1}\right\|_{\mathrm{op}} \leq C^{\prime}m^{D+d+1/2} ~~ \mathrm{and} ~~ \mathcal{L}(m)\left\|\Psi_m^{-1}\right\|_{\mathrm{op}} \leq C\left(\underset{j \in [\![0,m-1]\!]}{\max}{\left\|\ell_j\right\|_{\infty}^2}\right)m^{D+d+3/2}.
\end{equation}
Under Assumption~\ref{ass:limit-case-elliptic}, since $\sigma^2(x) \leq \sigma_{+}^2$ for all $x \in \mathbb{R}$, there exist constants $C,C^{\prime}>0$ such that
\begin{equation}\label{eq:Gram2-2} 
    \left\|\Psi_m^{-1}\right\|_{\mathrm{op}} \leq C^{\prime}m^{d+1/2} ~~ \mathrm{and} ~~ \mathcal{L}(m)\left\|\Psi_m^{-1}\right\|_{\mathrm{op}} \leq C\left(\underset{j \in [\![0,m-1]\!]}{\max}{\left\|\ell_j\right\|_{\infty}^2}\right)m^{d+3/2}.
\end{equation}
\end{proof}

\subsection{Proof of Theorem~\ref{thm:upper-limit-NAE}}

\begin{proof}
Denote by $P_m(\sigma^2)$ the orthogonal projection of $\sigma^2$ onto $\mathcal{S}_m, ~ m \in \mathcal{M}_n$ with respect to the norm $\|.\|_{\pi}$, and we have $\left\|P_m(\sigma^2) - \sigma^2\right\|_{\pi}^2 = \underset{f \in \mathcal{S}_m}{\inf}{\left\|f - \sigma^2\right\|_{\pi}^2} \leq \underset{f \in \mathcal{S}_{m,L}}{\inf}{\left\|f - \sigma^2\right\|_{\pi}^2}$.
By Equation~\eqref{eq:NA-estimator} and for all $f \in \mathcal{S}_m, ~~ m \in \mathcal{M}_n$, we have the following:
\begin{equation}\label{eq:thm31-1}
    \gamma_n(\w{\sigma}_m^2) - \gamma_n(\sigma^2) \leq \gamma_n(f) - \gamma_n(\sigma^2),
\end{equation}
and by Equation~\eqref{eq:gamma-function} and for all $f \in \mathcal{S}_m, ~ m \in \mathcal{M}_n$,
\begin{equation}\label{eq:thm31-2}
    \gamma_n(f) - \gamma_n(\sigma^2) = \left\|f - \sigma^2\right\|_n^2 + 2\nu_n(\sigma^2 - f) + 2\mu_n(\sigma^2 - f),
\end{equation}
where 
$$\nu_n(\sigma^2 - f) := \dfrac{1}{n}\sum_{k=0}^{n-1}\xi_{k\Delta_n}\left[\sigma^2(X_{k\Delta_n}) - f(X_{k\Delta_n})\right] ~~ \mathrm{and} ~~ \mu_n(\sigma^2 - f) := \dfrac{1}{n}\sum_{k=0}^{n-1}R_{k\Delta_n}\left[\sigma^2(X_{k\Delta_n}) - f(X_{k\Delta_n})\right].$$
We deduce from Equations~\eqref{eq:thm31-1} and \eqref{eq:thm31-2} that for all $m \in \mathcal{M}_n$,
\begin{equation}\label{eq:thm31-3}
    \begin{aligned}
        \left\|\w{\sigma}_m^2 - \sigma^2\right\|_n^2 \leq \left\|P_m(\sigma^2) - \sigma^2\right\|_n^2 + 2\mu_n(\w{\sigma}_m^2 - P_m(\sigma^2)) + 2\nu_n(\w{\sigma}_m^2 - P_m(\sigma^2)).
    \end{aligned}
\end{equation}
Using inequality $2xy \leq rx^2 + r^{-1}y^2$ for all $x,y \in \mathbb{R}$ and for all $r>0$, we obtain the following:
\begin{equation*}
    \begin{aligned}
        \left\|\w{\sigma}_m^2 - \sigma^2\right\|_n^2 \leq &~ \left\|P_m(\sigma^2) - \sigma^2\right\|_n^2 + \dfrac{1}{r}\left\|\w{\sigma}_m^2 - P_m(\sigma^2)\right\|_{\pi}^2 + r\underset{f \in \mathcal{S}_m, ~ \|f\|_{\pi}^2 = 1}{\sup}{\nu_n^2(f)} + \dfrac{1}{r}\left\|\w{\sigma}_m^2 - P_m(\sigma^2)\right\|_n^2\\
        & + \dfrac{r}{n}\sum_{k=0}^{n-1}R_{k\Delta_n}^2.
    \end{aligned}
\end{equation*}
In the event $\Omega_m$, we have $(1/2)\|\w{\sigma}_m^2 - f\|_{\pi}^2 \leq \|\w{\sigma}_m^2 - f\|_n^2 \leq (3/2)\|\w{\sigma}_m^2 - f\|_{\pi}^2$. Then, for all $m \in \mathcal{M}_n$,
$$\left\|\w{\sigma}_m^2 - \sigma^2\right\|_n^2\mathds{1}_{\Omega_m} \leq \left\|P_m(\sigma^2) - \sigma^2\right\|_n^2 + \dfrac{3}{r}\left\|\w{\sigma}_m^2 - P_m(\sigma^2)\right\|_{n}^2 + r\underset{f \in \mathcal{S}_m, ~ \|f\|_{\pi}^2 = 1}{\sup}{\nu_n^2(f)}\mathds{1}_{\Omega_m} + \dfrac{r}{n}\sum_{k=0}^{n-1}R_{k\Delta_n}^2,$$
and then
$$\left(1 - \dfrac{6}{r}\right)\left\|\w{\sigma}_m^2 - \sigma^2\right\|_n^2\mathds{1}_{\Omega_m} \leq \left(1 + \dfrac{6}{r}\right)\left\|P_m(\sigma^2) - \sigma^2\right\|_n^2 + r\underset{f \in \mathcal{S}_m, ~ \|f\|_{\pi}^2 = 1}{\sup}{\nu_n^2(f)} + \dfrac{r}{n}\sum_{k=0}^{n-1}R_{k\Delta_n}^2.$$
For $r = 12$, we deduce that for all $m \in \mathcal{M}_n$,
\begin{equation}\label{eq:risk}
    \begin{aligned}
        \mathbb{E}\left[\left\|\w{\sigma}_m^2 - \sigma^2\right\|_n^2\right] = &~ \mathbb{E}\left[\left\|\w{\sigma}_m^2 - \sigma^2\right\|_n^2\mathds{1}_{\Omega_n}\right] + \mathbb{E}\left[\left\|\w{\sigma}_m^2 - \sigma^2\right\|_n^2\mathds{1}_{\Omega_n^c}\right]\\
        \leq &~ 3\underset{f \in \mathcal{S}_{m,L}}{\inf}\left\|f - \sigma^2\right\|_{\pi}^2 + 24\mathbb{E}\left[\underset{f \in \mathcal{S}_m, ~ \|f\|_{\pi}^2 = 1}{\sup}{\nu_n^2(f)}\right] + \dfrac{24}{n}\sum_{k=0}^{n-1}\mathbb{E}\left[R_{k\Delta_n}^2\right]\\
        & + 2mL\mathbb{P}\left(\Omega_m^c\right) + 2\sqrt{\mathbb{E}\left[\sigma^4(X_0)\right]\mathbb{P}\left(\Omega_m^{c}\right)}.
    \end{aligned}
\end{equation}

\subsection*{Upper bound of $\mathbb{E}\left[R_{k\Delta_n}^2\right]$ for each $k \in \{0, \ldots, n-1\}$}

For all $k \in [\![0, n-1]\!]$, we have
\begin{equation}\label{eq:R}
    \begin{aligned}
        \mathbb{E}\left[R_{k\Delta_n}^2\right] \leq &~ 2\mathbb{E}\left[\left(R_{k\Delta_n}^{(1)}\right)^2\right] + 2\mathbb{E}\left[\left(R_{k\Delta_n}^{(2)}\right)^2\right] + 2\mathbb{E}\left[\left(R_{k\Delta_n}^{(3)}\right)^2\right].
    \end{aligned}
\end{equation}
From Equation~\eqref{eq:Residual} and using the Cauchy Schwarz inequality, the Burkholder-Davis-Gundy inequality, and Equation~\eqref{eq:prop1-7}, there exists a constant $C>0$ such that for each $k \in \{0, \ldots, n-1\}$,
\begin{equation}\label{eq:R1}
    \begin{aligned}
        \mathbb{E}\left[\left(R_{k\Delta_n}^{(1)}\right)^2\right] \leq &~  \dfrac{2}{\Delta_n^2}\mathbb{E}\left[\left(\int_{k\Delta_n}^{(k+1)\Delta_n}b(X_s)ds\right)^4\right] \leq 2\Delta_n\int_{k\Delta_n}^{(k+1)\Delta_n}\mathbb{E}\left[b^4(X_s)\right]ds\\
        \leq &~ 2C\Delta_n^2\left(1 + \mathbb{E}\left[\underset{s \in [k\Delta_n, (k+1)\Delta_n]}{\sup}|X_s|^4\right]\right) \leq C\Delta_n^2.
    \end{aligned}
\end{equation}
For the second term on the right-hand side of Equation~\eqref{eq:R}, from Equation~\eqref{eq:Residual}, note that since $\Gamma = b\sigma\sigma^{\prime} + [\sigma\sigma^{\prime\prime} + (\sigma^{\prime})^2]\sigma^2$,  from Assumption~\ref{ass:Regular-Sigma} and by Proposition~\ref{prop:Exp-Holder} with $p=8, ~ d > 8D$, we obtain $\mathbb{E}[\Gamma^2(X_0)] < \infty$. Then, using Cauchy-Schwarz's inequality, there exists a constant $C>0$ such that
\begin{equation}\label{eq:R2}
    \begin{aligned}
        \mathbb{E}\left[\left(R_{k\Delta_n}^{(2)}\right)^2\right] \leq &~ \dfrac{1}{\Delta_n}\int_{k\Delta_n}^{(k+1)\Delta_n}((k+1)\Delta_n - s)^2\mathbb{E}\left[\Gamma^2(X_{k\Delta_n})\right]ds \leq \Delta_n^2\mathbb{E}\left[\Gamma^2(X_0)\right] \leq C\Delta_n^2.
    \end{aligned}
\end{equation}
For the third term on the right-hand side of Equation~\eqref{eq:R} and from Equation~\eqref{eq:Residual}, using Cauchy-Schwarz's inequality, the H\"older inequality and the Burkholder-Davis-Gundy inequality, there exists a constant $C>0$ such that for each $k \in \{0, \ldots, n-1\}$,
\begin{equation*}
    \begin{aligned}
        \mathbb{E}\left[\left(R_{k\Delta_n}^{(3)}\right)^2\right] \leq &~ \dfrac{4}{\Delta_n^2}\left(\mathbb{E}\left[\left(\int_{k\Delta_n}^{(k+1)\Delta_n}(b(X_s) - b(X_{k\Delta_n}))ds\right)^4\right]\right)^{1/2}\left(\mathbb{E}\left[\left(\int_{k\Delta_n}^{(k+1)\Delta_n}\sigma(X_s)dW_s\right)^4\right]\right)^{1/2}\\
        \leq &~ C\left(\int_{k\Delta_n}^{(k+1)\Delta_n}\mathbb{E}\left[(b(X_s) - b(X_{k\Delta_n}))^4\right]ds\right)^{1/2}\left(\int_{k\Delta_n}^{(k+1)\Delta_n}\mathbb{E}\left[\sigma^4(X_s)\right]ds\right)^{1/2}.
    \end{aligned}
\end{equation*}
Under Assumption~\ref{ass:StrongAssDrift} and by Proposition~\ref{prop:Exp-Holder} with $p = 2 < d/D$, where $d > 8D$ and $D \geq 1$, or under Assumption~\ref{ass:limit-case-elliptic} and by Proposition~\ref{prop:Exp-Holder} with $p = 2 < d-1$ where $d > 4$, we obtain
\begin{equation}\label{eq:R3}
    \mathbb{E}\left[\left(R_{k\Delta_n}^{(3)}\right)^2\right] \leq C\Delta_n^2.
\end{equation}
Finally, from Equations~\eqref{eq:R3}, \eqref{eq:R2}, \eqref{eq:R1} and \eqref{eq:R}, there exists a constant $C>0$ such that for all $k \in \{0, \ldots, n-1\}$,
\begin{equation}\label{eq:R-bis}
    \mathbb{E}\left[\left(R_{k\Delta_n}\right)^2\right] \leq C\Delta_n^{2},
\end{equation}

\subsection*{Upper bound of $\mathbb{E}\left[\underset{f \in \mathcal{S}_m, ~ \|f\|_{\pi}^2 = 1}{\sup}{\nu_n^2(f)}\right]$}

Fix $m \in \mathcal{M}_n$.  For all $f \in \mathcal{S}_m$, there exists $\mathbf{a} = (a_0, \ldots, a_{m-1}) \in \mathbb{R}^m$ such that $f = \sum_{i \in [\![0, m-1]\!]}a_i\phi_i$. Moreover, we have $\left\|f\right\|_{\pi}^2 = \mathbf{a}^{\prime}\Psi_m\mathbf{a} = 1  = \left\|\Psi_m^{1/2}\mathbf{a}\right\|_2^2 \iff \mathbf{a} = \Psi_{m}^{-1/2}\mathbf{u}$, where $\mathbf{u} = (u_0, \ldots, u_{m-1}) \in \mathbb{R}^m$ such that $\|\mathbf{u}\|_{2,m} = 1$. Then, for all $f = \sum_{i=0}^{m-1}a_i\phi_i \in \mathcal{S}_m$, we have the following:
\begin{equation}\label{eq:term00}
    \nu_n^2(f) = \left(\dfrac{1}{n}\sum_{k=0}^{n-1}\xi_{k\Delta_n}\sum_{i=0}^{m-1}a_i\phi_i(X_{k\Delta_n})\right)^2 = \left(\sum_{i=0}^{m-1}a_i\nu(\phi_i)\right)^2 \leq \|\mathbf{a}\|_{2,m}^2\sum_{i=0}^{m-1}\nu_n^2(\phi_i) \leq \left\|\Psi_m^{-1}\right\|_{\mathrm{op}}\sum_{i=0}^{m-1}\nu_n^2(\phi_i).
\end{equation}
We now focus on the upper bound of each term $\nu_n^2(\phi_i), ~ i \in \{0, \ldots, m-1\}$. We have the following:
\begin{equation}\label{eq:term0}
    \begin{aligned}
       \mathbb{E}\left[\nu_n^2(\phi_i)\right] = &~ \dfrac{1}{n^2}\sum_{k,k^{\prime} = 0}^{n-1}\mathbb{E}\left[\xi_{k\Delta_n}\xi_{k^{\prime}\Delta_n}\phi_i(X_{k\Delta_n})\phi_i(X_{k^{\prime}\Delta_n})\right]\\
        = &~ \dfrac{1}{n^2}\sum_{k=0}^{n-1}\mathbb{E}\left[\xi_{k\Delta_n}^2\phi_i^{2}(X_{k\Delta_n})\right] + \dfrac{1}{n^2}\sum_{k \neq k^{\prime}}\mathbb{E}\left[\xi_{k\Delta_n}\xi_{k^{\prime}\Delta_n}\phi_i(X_{k\Delta_n})\phi_i(X_{k^{\prime}\Delta_n})\right].
    \end{aligned}
\end{equation}
On the one hand, by Equation~\eqref{eq:Error-term} and using Cauchy-Schwaz's inequality, for all $k \in \{0, \ldots, n-1\}$,
\begin{equation*}
  \begin{aligned}
      \mathbb{E}\left[\zeta_{k\Delta_n}^2\phi_i^2(X_{k\Delta_n})\right] = &~ 3\|\phi_i\|_{\infty}^2\left(\mathbb{E}\left[\left(\xi_{k\Delta_n}^{(1)}\right)^2\right] + \mathbb{E}\left[\left(\xi_{k\Delta_n}^{(2)}\right)^2\right] + \mathbb{E}\left[\left(\xi_{k\Delta_n}^{(3)}\right)^2\right]\right).
  \end{aligned}  
\end{equation*}
By Proposition~\ref{prop:Exp-Holder} with $p=8 < d/D$ and using the Burkholder-Davis-Gundy inequality, we obtain the following: 
\begin{equation*}
    \begin{aligned}
        \exists C>0, ~ \mathbb{E}\left[(\xi_{k\Delta_n}^{(1)})^2\right] \leq &~ C\mathbb{E}\left[\sigma^4(X_0)\right] < \infty,\\
        \exists C>0, ~ \mathbb{E}\left[(\xi_{k\Delta_n}^{(2)})^2\right] \leq &~ \dfrac{4}{\Delta_n^2}\int_{k\Delta_n}^{(k+1)\Delta_n}((k+1)\Delta_n - s)^2\mathbb{E}\left[((\sigma^{\prime})^2\sigma^4)(X_s)\right]ds \leq \Delta_n\mathbb{E}\left[((\sigma^{\prime})^2\sigma^4)(X_0)\right]\\
        \leq &~ C\Delta_n,\\
        \exists C>0, ~ \mathbb{E}\left[(\xi_{k\Delta_n}^{(3)})^2\right] \leq &~ C\Delta_n\left(\mathbb{E}\left[b^4(X_0)\right]\mathbb{E}\left[\sigma^4(X_0)\right]\right)^{1/2} \leq C\Delta_n.
    \end{aligned}
\end{equation*}
We deduce from the above results that there exists $C>0$ such that for all $k \in \{0, \ldots, n-1\}$,
\begin{equation}\label{eq:term1}
    \mathbb{E}\left[\xi_{k\Delta_n}^2\phi_i^{2}(X_{k\Delta_n})\right] \leq C\|\phi_i\|_{\infty}^2.
\end{equation}
On the other hand, note that from Equation~\ref{eq:Error-term} and for all $k \in \{0, \ldots, n-1\}$, 
\begin{equation*}
    \begin{aligned}
        \mathbb{E}\left[\xi_{k\Delta_n} \biggm\vert \mathcal{F}_{k\Delta_n}^X\right] = \mathbb{E}\left[\xi_{k\Delta_n}^{(1)} \biggm\vert \mathcal{F}_{k\Delta_n}^X\right] + \mathbb{E}\left[\xi_{k\Delta_n}^{(2)} \biggm\vert \mathcal{F}_{k\Delta_n}^X\right] + \mathbb{E}\left[\xi_{k\Delta_n}^{(3)} \biggm\vert \mathcal{F}_{k\Delta_n}^X\right] = 0.
    \end{aligned}
\end{equation*}
Then, for all $k,k^{\prime} \in \{0 \ldots, n-1\}$ such that $k \neq k^{\prime}$, we have the following:
\begin{equation}\label{eq:term2}
    \begin{aligned}
        & \mathbb{E}\left[\xi_{k\Delta_n}\xi_{k^{\prime}\Delta_n}\phi_i(X_{k\Delta_n})\phi_i(X_{k^{\prime}\Delta_n})\right]\\
        & = \mathbb{E}\left[\xi_{(k \land k^{\prime})\Delta_n}\phi_i(X_{(k\land k^{\prime})\Delta_n})\phi_i(X_{(k \vee k^{\prime})\Delta_n})\mathbb{E}\left(\xi_{(k \vee k^{\prime})\Delta_n} \biggm\vert \mathcal{F}_{(k \vee k^{\prime})\Delta_n}^X\right)\right] = 0.
    \end{aligned}
\end{equation}
We deduce from Equations~\eqref{eq:term2}, \eqref{eq:term1}, \eqref{eq:term0} and \eqref{eq:term00} that there exists a constant $C>0$ such that
\begin{equation}\label{eq:kappa}
    \mathbb{E}\left[\underset{f \in \mathcal{S}_m, ~ \|f\|_{\pi}^2 = 1}{\sup}{\nu_n^2(f)}\right] \leq C\dfrac{\left\|\Psi_m^{-1}\right\|_{\mathrm{op}}}{n}\sum_{i=0}^{m-1}\|\phi_i\|_{\infty}^2 \leq C\underset{i\in \{0, \ldots, m-1\}}{\max}{\|\phi_i\|_{\infty}^2} \dfrac{m\left\|\Psi_m^{-1}\right\|_{\mathrm{op}}}{n}.
\end{equation}
Under Assumption~\ref{ass:StrongAssDrift} and from Propositions~\ref{prop:Exp-Holder}, Equations~\eqref{eq:Gram1} and \eqref{eq:Gram1.2} in the proof of Proposition~\ref{prop:operator-norm-identity}, from Lemma~\ref{lm:omega-comp} and  Equations~\eqref{eq:kappa}, \eqref{eq:R-bis} and \eqref{eq:risk}, we deduce that there exists a constant $C>0$ such that for all $m \in \mathcal{M}_n$,
\begin{equation*}
    \begin{aligned}
        \mathbb{E}\left[\left\|\w{\sigma}_m^2 - \sigma^2\right\|_n^2\right] \leq &~ 3\underset{f \in \mathcal{S}_{m,L}}{\inf}{\left\|f - \sigma^2\right\|_{\pi}^2} + C\left(\dfrac{m\left\|\Psi_m^{-1}\right\|_{\mathrm{op}}}{n} + mL\mathbb{P}\left(\Omega_m^c\right) + \sqrt{\mathbb{P}\left(\Omega_m^{c}\right)} + \Delta_n^{2}\right)\\
        \leq &~ 3\underset{f \in \mathcal{S}_{m,L}}{\inf}{\left\|f - \sigma^2\right\|_{\pi}^2} + C\left(\dfrac{m^{k_0}}{n} + m\log(n)\exp\left(-\frac{\gamma}{8}\sqrt{\varepsilon_0}\log^2(n)\right) + C\Delta_n^{2}\right),
    \end{aligned}
\end{equation*}
where $k_0 = D+d+3/2$ and $L = \log(n)$. Under Assumption~\ref{ass:limit-case-elliptic}, from Equations~\eqref{eq:Gram1.2} and \eqref{eq:Gram2-2} in the proof of Proposition~\ref{prop:operator-norm-identity}, from Lemma~\ref{lm:omega-comp} and  Equations~\eqref{eq:kappa}, \eqref{eq:R-bis} and \eqref{eq:risk}, we deduce that for all $m \in \mathcal{M}_n$,
\begin{equation*}
    \begin{aligned}
        \mathbb{E}\left[\left\|\w{\sigma}_m^2 - \sigma^2\right\|_n^2\right] \leq &~ 3\underset{f \in \mathcal{S}_{m,L}}{\inf}{\left\|f - \sigma^2\right\|_{\pi}^2} + C\left(\dfrac{m^{d_0}}{n} + \dfrac{m\log(n)}{n^{2}} + \dfrac{1}{n} + C\Delta_n^{2}\right),
    \end{aligned}
\end{equation*}
where $d_0 = d+3/2$, which concludes the proof.
\end{proof}

\subsection{Proof of Theorem~\ref{thm:upper-limit2-NAE}}

From Equations~\eqref{eq:risk} in the proof of Theorem~\ref{thm:upper-limit-NAE}, we have
\begin{equation*}
    \begin{aligned}
        \mathbb{E}\left[\left\|\w{\sigma}_m^2 - \sigma^2\right\|_n^2\right] \leq &~ 3\underset{f \in \mathcal{S}_{m,L}}{\inf}\left\|f - \sigma^2\right\|_{\pi}^2 + 24\mathbb{E}\left[\underset{f \in \mathcal{S}_m, ~ \|f\|_{\pi}^2 = 1}{\sup}{\nu_n^2(f)}\right] + \dfrac{24}{n}\sum_{k=0}^{n-1}\mathbb{E}\left[R_{k\Delta_n}^2\right]\\
        & + 2mL\mathbb{P}\left(\Omega_m^c\right) + 2\sqrt{\mathbb{E}\left[\sigma^4(X_0)\right]\mathbb{P}\left(\Omega_m^{c}\right)}.
    \end{aligned}
\end{equation*}
For all $f \in \mathcal{S}_m, ~ m \in \mathcal{M}_n$, $\nu_n(f) = n^{-1}\sum_{k = 0}^{n-1}\zeta_{k\Delta_n}f(X_{k\Delta_n})$ with $\zeta_{k\Delta_n} = \Delta_n^{-1}\sigma^2(X_{k\Delta_n})[(W_{(k+1)\Delta_n} - W_{k\Delta_n})^2 - \Delta_n]$. For all $k,k^{\prime} \in \{0, \ldots, n-1\}$ such that $k \neq k^{\prime}$ and for all $f \in \mathcal{S}_m, ~ m \in \mathcal{M}_n$, we have 
\begin{align*}
    & \mathbb{E}\left[\zeta_{k\Delta_n}f(X_{k\Delta_n})\zeta_{k\Delta_n}f(X_{k\Delta_n}) \biggm\vert \mathcal{F}_{k\Delta_n}^X\right]\\
    & = \zeta_{(k \land k^{\prime})\Delta_n}f(X_{(k \land k^{\prime})\Delta_n})f(X_{(k \vee k^{\prime})\Delta_n}) \mathbb{E}\left[\zeta_{(k \vee k^{\prime})\Delta_n} \biggm\vert \mathcal{F}_{(k \vee k^{\prime})}^X\right] = 0.
\end{align*}
In addition, for all $(k,i) \in \{0, \ldots, n-1\} \times \{0, \ldots, m-1\}$, we have
\begin{align*}
    \mathbb{E}\left[\zeta_{k\Delta_n}^2\phi_i^2(X_{k\Delta_n})\right] \leq &~ 2\Delta_n^{-2}\mathbb{E}\left[\sigma^4(X_{k\Delta_n})\phi_i^2(X_{k\Delta_n})\left\{\Delta_n^2 + \mathbb{E}\left((W_{(k+1)\Delta_n} - W_{k\Delta_n})^4\right)\right]\right\}\\
    = &~  \mathcal{O}\left(\|\phi_i\|_{\infty}^2\mathbb{E}\left[\sigma^4(X_0)\right]\right).
\end{align*}
We deduce from Equations~\eqref{eq:term00} and \eqref{eq:term0} that there exists a constant $C>0$ such that
$$\mathbb{E}\left[\underset{f \in \mathcal{S}_m, ~ \|f\|_{\pi}^2 = 1}{\sup}{\nu_n^2(f)}\right] \leq C \dfrac{m}{n}\left\|\Psi_m^{-1}\right\|_{\mathrm{op}}.$$
It follows that 
\begin{equation}\label{eq:risk-bound}
    \begin{aligned}
        \mathbb{E}\left[\left\|\w{\sigma}_m^2 - \sigma^2\right\|_n^2\right] \leq &~ 3\underset{f \in \mathcal{S}_{m,L}}{\inf}\left\|f - \sigma^2\right\|_{\pi}^2 + 24C \dfrac{m}{n}\left\|\Psi_m^{-1}\right\|_{\mathrm{op}} + \dfrac{24}{n}\sum_{k=0}^{n-1}\mathbb{E}\left[R_{k\Delta_n}^2\right]\\
        & + 2mL\mathbb{P}\left(\Omega_m^c\right) + 2\sqrt{\mathbb{E}\left[\sigma^4(X_0)\right]\mathbb{P}\left(\Omega_m^{c}\right)}.
    \end{aligned}
\end{equation}
Now, focusing on the third term on the right-hand side of the above equation, for all $k \in [\![0, n-1]\!]$, we have
$$ \mathbb{E}\left[\widetilde{R}_{k\Delta_n}^2\right] \leq 3\mathbb{E}\left[\left(\widetilde{R}_{k\Delta_n}^{(1)}\right)^2\right] + 3\mathbb{E}\left[\left(\widetilde{R}_{k\Delta_n}^{(2)}\right)^2\right] + 3\mathbb{E}\left[\left(\widetilde{R}_{k\Delta_n}^{(3)}\right)^2\right].$$
From Equation~\eqref{eq:Residual} and using the Cauchy Schwarz inequality and the Burkholder-Davis-Gundy inequality, there exists a constant $C>0$ such that for each $k \in \{0, \ldots, n-1\}$,
\begin{equation*}
    \begin{aligned}
        \mathbb{E}\left[\left(\widetilde{R}_{k\Delta_n}^{(1)}\right)^2\right] \leq &~  \dfrac{2}{\Delta_n^2}\mathbb{E}\left[\left(\int_{k\Delta_n}^{(k+1)\Delta_n}b(X_s)ds\right)^4\right] + \dfrac{C}{\Delta_n^2}\mathbb{E}\left[\left(\int_{k\Delta_n}^{(k+1)\Delta_n}\left(\sigma(X_s) - \sigma(X_{k\Delta_n})\right)^2ds\right)^2\right]\\
        \leq &~ 2C\Delta_n^2\left(1 + \mathbb{E}\left[\underset{s \in [k\Delta_n, (k+1)\Delta_n]}{\sup}|X_s|^4\right]\right) + \dfrac{C}{\Delta_n}\int_{k\Delta_n}^{(k+1)\Delta_n}\mathbb{E}\left[\left(\sigma(X_s) - \sigma(X_{k\Delta_n})\right)^4\right]ds.
    \end{aligned}
\end{equation*}
Under Assumption~\ref{ass:StrongAssDrift} and by Equation~\eqref{eq:prop1-7} and Proposition~\ref{prop:Exp-Holder} with $p = 2 < d/D$, where $d > 8D$ and $D \geq 1$, or under Assumption~\ref{ass:limit-case-elliptic} and by Equation~\eqref{eq:prop1-7} and Proposition~\ref{prop:Exp-Holder} with $p = 2 < d-1$, where $d > 4$, we obtain for each $k \in \{0, \ldots, n-1\}$, $\mathbb{E}[(\widetilde{R}_{k\Delta_n}^{(1)})^2] = \mathcal{O}\left(\Delta_n^{2\alpha}\right), ~~ \alpha \in [1/2, 1]$. For the second term on the right-hand side of Equation~\eqref{eq:Residual}, we have $\widetilde{R}_{k\Delta_n}^{(2)} = R_{k\Delta_n}^{(3)}$, then $\mathbb{E}[(\widetilde{R}_{k\Delta_n}^{(2)})^2] = \mathcal{O}\left(\Delta_n^2\right)$. 
For the third term on the right-hand side of Equation~\eqref{eq:Residual}, using once again Cauchy-Schwarz's inequality, H\"older's inequality and Burkholder-Davis-Gundy's inequality, for each $k \in \{0, \ldots, n-1\}$, the following holds:
\begin{equation*}
    \begin{aligned}
        &\mathbb{E}\left[\left(\widetilde{R}_{k\Delta_n}^{(3)}\right)^2\right]\\
        &\leq \dfrac{C}{\Delta_n^{3/2}}\left(\mathbb{E}\left[\sigma^4(X_{k\Delta_n})\left(W_{(k+1)\Delta_n} - W_{k\Delta_n}\right)^4\right]\right)^{1/2} \left(\int_{k\Delta_n}^{(k+1)\Delta_n}\mathbb{E}\left[(\sigma(X_s) - \sigma(X_{k\Delta_n}))^4\right]ds\right)^{1/2}\\
        &\leq \dfrac{C}{\Delta_n^{1/2}}\left(\mathbb{E}\left[\sigma^4(X_{k\Delta_n})\right]\right)^{1/2}\left(\int_{k\Delta_n}^{(k+1)\Delta_n}\mathbb{E}\left[(\sigma(X_s) - \sigma(X_{k\Delta_n}))^4\right]ds\right)^{1/2},
    \end{aligned}
\end{equation*}
where $C>0$ is a constant. Using a similar reasoning as in the previous cases, we deduce that under Assumption~\ref{ass:StrongAssDrift} or Assumption~\ref{ass:limit-case-elliptic}, we have $\mathbb{E}[(\widetilde{R}_{k\Delta_n}^{(3)})^2] \leq C\Delta_n^{\alpha}, ~~ \alpha \in [1/2, 1]$.
Finally, combining the three terms and their respective bounds, we obtain for all $k \in \{0, \ldots, n-1\}$, $\mathbb{E}\left[\left(\widetilde{R}_{k\Delta_n}\right)^2\right] \leq C\Delta_n^{\alpha}, ~~ \alpha \in [1/2, 1]$. Then, back to Equation~\eqref{eq:risk-bound}, there exists a constant $C>0$ such that
\begin{equation*}
    \begin{aligned}
        \mathbb{E}\left[\left\|\w{\sigma}_m^2 - \sigma^2\right\|_n^2\right] \leq &~ 3\underset{f \in \mathcal{S}_{m,L}}{\inf}\left\|f - \sigma^2\right\|_{\pi}^2 + C\left(\dfrac{m}{n}\left\|\Psi_m^{-1}\right\|_{\mathrm{op}} + \dfrac{1}{n}\sum_{k=0}^{n-1}\mathbb{E}\left[R_{k\Delta_n}^2\right] + mL\mathbb{P}\left(\Omega_m^c\right) + \sqrt{\mathbb{P}\left(\Omega_m^{c}\right)}\right).
    \end{aligned}
\end{equation*}
The above result together with Proposition~\ref{prop:operator-norm-identity} and its proof and Lemma~\ref{lm:omega-comp} lead to the expected risk bounds.

\subsection{Proof of Theorem~\ref{thm:adaptation}}

\begin{proof}
    Recall that $$\w{m} := \underset{m \in \mathcal{M}_n}{\inf}{\left\{\gamma_n(\w{\sigma}_m^2) + \mathrm{pen}(m)\right\}},$$ 
    where $\mathcal{M}_n = \left\{1, \ldots, m_{\max}\right\}$ with $m_{\max} = \left\lceil(n\Delta_n)^{2/(2s_0+1)}\right\rceil$ and $s_0 = D + d + 2$ under Assumption~\ref{ass:StrongAssDrift} and $s_0 = 2d + 3$ under Assumption~\ref{ass:limit-case-elliptic}. Then, for all $m \in \mathcal{M}_n$,
    $$\gamma_n(\w{\sigma}_{\w{m}}^2) + \mathrm{pen}(\w{m}) \leq \gamma_n(\w{\sigma}_m^2) + \mathrm{pen}(m) \leq \gamma_n(P_m(\sigma^2)) + \mathrm{pen}(m),$$
    where $P_m(\sigma^2)$ is the orthogonal projection of $\sigma_I^2$ onto the approximation space $\mathcal{S}_m$. It follows that
    $$\gamma_n(\w{\sigma}_{\w{m}}^2) - \gamma_n(\sigma_I^2) \leq \gamma_n(P_m(\sigma^2)) - \gamma_n(\sigma_I^2) + \mathrm{pen}(m) - \mathrm{pen}(\w{m})$$
    and from Equation~\eqref{eq:thm31-3} and for all $m \in \mathcal{M}_n$,
    \begin{equation*}
    \begin{aligned}
        \left\|\w{\sigma}_{\w{m}}^2 - \sigma_I^2\right\|_n^2 \leq \left\|P_m(\sigma^2) - \sigma_I^2\right\|_n^2 + 2\mu_n(\w{\sigma}_{\w{m}}^2 - P_m(\sigma^2)) + 2\nu_n(\w{\sigma}_{\w{m}}^2 - P_m(\sigma^2)) + \mathrm{pen}(m) - \mathrm{pen}(\w{m}).
    \end{aligned}
\end{equation*}
Recall that $\nu_n = \nu_{n,1} + \nu_{n,2} + \nu_{n,3}$, where, for all $f \in \mathcal{S}_m$, $m \in \mathcal{M}_n$, 
\begin{equation*}
    \begin{aligned}
        \nu_{n,1}(f) := &~ \dfrac{1}{n}\sum_{k=0}^{n-1}\xi_{k\Delta_n}^{(1)}f(X_{k\Delta_n}), ~~ \nu_{n,2}(f) := \dfrac{1}{n}\sum_{k=0}^{n-1}\xi_{k\Delta_n}^{(2)}f(X_{k\Delta_n}) ~~ \mathrm{and} ~~ \nu_{n,3}(f) := \dfrac{1}{n}\sum_{k=0}^{n-1}\xi_{k\Delta_n}^{(3)}f(X_{k\Delta_n}).
    \end{aligned}
\end{equation*}
Using the inequality $2xy \leq x^2/r + ry^2$ for all $(x,y,r) \in \mathbb{R}^2 \times \mathbb{R}_+$, we obtain for all $m \in \mathcal{M}_n$,
\begin{equation*}
    \begin{aligned}
        \left\|\w{\sigma}_{\w{m}}^2 - \sigma_I^2\right\|_n^2 \leq &~ \left\|P_m(\sigma^2) - \sigma_I^2\right\|_n^2 + r^{-1}\left\|\w{\sigma}_{\w{m}}^2 - P_m(\sigma^2)\right\|_{\pi}^2 + 3r\underset{f \in \mathcal{S}_{m} + \mathcal{S}_{\w{m}}, ~ \|f\|_{\pi}^2 = 1}{\sup}{\nu_{n,1}^2(f)}\\
        & + 3r\underset{f \in \mathcal{S}_{m_{\max}}, ~ \|f\|_{\pi}^2 = 1}{\sup}{\nu_{n,2}^2(f)} + 3r\underset{f \in \mathcal{S}_{m_{\max}}, ~ \|f\|_{\pi}^2 = 1}{\sup}{\nu_{n,3}^2(f)} + \dfrac{1}{r}\left\|\w{\sigma}_{\w{m}}^2 - P_m(\sigma^2)\right\|_n^2\\
        & + \dfrac{r}{n}\sum_{k=0}^{n-1}R_{k\Delta_n}^2 + \mathrm{pen}(m) - \mathrm{pen}(\w{m}).
    \end{aligned}
\end{equation*}
From the proof of Theorem~\ref{thm:upper-limit-NAE} (Equations~\eqref{eq:term00}, \eqref{eq:term0}, \eqref{eq:term1} and \eqref{eq:term2}), there exists a constant $C>0$ such that
$$\mathbb{E}\left[\underset{f \in \mathcal{S}_{m_{\max}}, ~ \|f\|_{\pi}^2 = 1}{\sup}{\nu_{n,j}^2(f)}\right] \leq C\underset{i \in \{0, \ldots, m-1\}}{\max}{\|\phi_i\|_{\infty}^2}\dfrac{m\left\|\Psi_m^{-1}\right\|_{\mathrm{op}}\Delta_n}{n} = \mathcal{O}\left(\dfrac{m^{k_0}\Delta_n}{n}\right), ~~~ j\in \{2, 3\}.$$
Then, in the event $\Omega_{m_{\max}}$ and since $\mathcal{S}_{m} + \mathcal{S}_{\w{m}} \subset \mathcal{S}_{m_{\max}}$ for all $m \in \mathcal{M}_n$, for $r=12$, and by Equations~\eqref{eq:risk} and \eqref{eq:R-bis}, for all $m \in \mathcal{M}_n$, there exists a constant $C>0$ such that
\begin{equation*}
    \begin{aligned}
        \mathbb{E}\left[\left\|\w{\sigma}_{\w{m}}^2 - \sigma^2\right\|_n^2\mathds{1}_{\Omega_{m_{\max}}}\right] \leq &~ 3\left\{\underset{f \in \mathcal{S}_{m,L}}{\inf}{\left\|f - \sigma^2\right\|_{\pi}^2} + \mathrm{pen}(m)\right\} + C\left(\dfrac{m_{\max}^{k_0}\Delta_n}{n} + \Delta_n^{2}\right)\\
        & + 12\mathbb{E}\left[\left(\underset{f \in \mathcal{S}_{m} + \mathcal{S}_{\w{m}}, ~ \|f\|_{\pi}^2 = 1}{\sup}{\nu_{n,1}^2(f)} - Q(m,\w{m})\right)_{+}\mathds{1}_{\Omega_{m_{\max}}}\right],
    \end{aligned}
\end{equation*}
where for all $m,m^{\prime} \in \mathcal{M}_n, ~~  12Q(m,m^{\prime}) \leq \mathrm{pen}(m) + 3\mathrm{pen}(m^{\prime})$. It follows that for all $m \in \mathcal{M}_n$, 
\begin{equation}\label{eq:adaptive-1}
    \begin{aligned}
        \mathbb{E}\left[\left\|\w{\sigma}_{\w{m}}^2 - \sigma^2\right\|_n^2\right] \leq &~ 3\left\{\underset{f \in \mathcal{S}_{m,L}}{\inf}{\left\|f - \sigma^2\right\|_{\pi}^2} + \mathrm{pen}(m)\right\} + 2m_{\max}L\mathbb{P}\left(\Omega_{m_{\max}}^c\right) + 2\sqrt{\mathbb{E}\left[X_0^4\right]\mathbb{P}\left(\Omega_{m_{\max}}^c\right)}\\
        & + 12\sum_{m^{\prime} \in \mathcal{M}_n}\mathbb{E}\left[\left(\underset{f \in \mathcal{S}_{m \vee m^{\prime}}, ~ \|f\|_{\pi}^2 = 1}{\sup}{\nu_{n,1}^2(f)} - Q(m, m^{\prime})\right)_{+}\right] + C\left(\dfrac{m_{\max}^{k_0}\Delta_n}{n} + \Delta_n^{2}\right).
    \end{aligned}
\end{equation}

\paragraph{\textbf{Case of an exponentially ergodic diffusion processes.}}

Consider the following lemma.
\begin{lemma}\label{lm:lemma-exponential}
    Under Assumptions~\ref{ass:drift}, \ref{ass:diffusion}, \ref{ass:NonDegeneracy}, \ref{ass:StrongAssDrift}, \ref{ass:stationary} and \ref{ass:Regular-Sigma}, and assuming that $d>8D$ and $n\Delta_n^2 = \varepsilon_0\log^4(n)$, there exists a constant $C>0$ such that
    $$\sum_{m^{\prime} \in \mathcal{M}_n}\mathbb{E}\left[\left(\underset{f \in \mathcal{S}_{m \vee m^{\prime}}, ~ \|f\|_{\pi}^2 = 1}{\sup}{\nu_{n,1}^2(f)} - Q(m, m^{\prime})\right)_{+}\right] \leq \dfrac{C}{n}, ~~~~ m \in \mathcal{M}_n,$$
    where, for all $m,m^{\prime} \in \mathcal{M}_n, ~~ Q(m,m^{\prime}) = \mathfrak{c}\dfrac{(m \vee m^{\prime})\left\|\Psi_{m \vee m^{\prime} - 1}^{-1}\right\|_{\mathrm{op}}}{n}$ with $\mathfrak{c} > 0$ a constant.
\end{lemma}
From the above lemma, the penalty function is derived from the function $(m,m^{\prime}) \mapsto Q(m,m^{\prime})$ and from Equations~\eqref{eq:Gram1} and \eqref{eq:Gram2} in the proof of Proposition~\ref{prop:operator-norm-identity}, we obtain $\mathrm{pen}(m) = \mathrm{pen}_1(m) = \tau_1 \frac{m^{k_0}}{n}$, where $\tau_1>0$ is a constant. Moreover, since $L = \log(n)$, we deduce from Proposition~\ref{prop:Exp-Holder} and Lemma~\ref{lm:omega-comp} that
$$m_{\max}L\mathbb{P}\left(\Omega_{\max}^c\right) \leq C(n\Delta_n)^{\frac{2}{2s_0+1}}\log(n)\exp\left(-\frac{\gamma}{4}\log^2(n)\right)$$
and $$\sqrt{\mathbb{E}\left[\sigma^4(X_0)\right]\mathbb{P}\left(\Omega_{m_{\max}}^c\right)} \leq C\exp\left(-\dfrac{\gamma}{8}\log^2(n)\right).$$
Then, from Lemma~\ref{lm:lemma-exponential} and Equation~\eqref{eq:adaptive-1} there exists a constant $C>0$ such that
$$\mathbb{E}\left[\left\|\w{\sigma}_{\w{m}}^2 - \sigma^2\right\|_n^2\right] \leq 3\underset{m \in \mathcal{M}_n}{\inf}\left\{\underset{f \in \mathcal{S}_{m,L}}{\inf}{\left\|f - \sigma^2\right\|_{\pi}^2} + \mathrm{pen}_1(m)\right\} + C\frac{\log^4(n)}{n}.$$
To conclude, from Equation~\eqref{eq:pi-norm-risk} in the proof of Corollary~\ref{cor:upper-limit-NAE}, we obtain
$$\mathbb{E}\left[\left\|\w{\sigma}_{\w{m}}^2 - \sigma^2\right\|_{\pi}^2\right] \leq 9\underset{m \in \mathcal{M}_n}{\inf}\left\{\underset{f \in \mathcal{S}_{m,L}}{\inf}{\left\|f - \sigma^2\right\|_{\pi}^2} + \mathrm{pen}_1(m)\right\} + C\frac{\log^4(n)}{n},$$
where $C>0$ is a new constant.

\paragraph{\textbf{Case of a polynomially ergodic diffusion processes.}}

We have the following lemma:
\begin{lemma}\label{lm:lemma-polynomial}
    Under Assumptions~\ref{ass:drift}, \ref{ass:diffusion}, \ref{ass:NonDegeneracy}, \ref{ass:stationary} \ref{ass:limit-case-elliptic} and \ref{ass:Regular-Sigma}, and assuming that $d > 8D$ and $n\Delta_n^2 = 1$, there exists a constant $C>0$ such that
    $$\sum_{m^{\prime} \in \mathcal{M}_n}\mathbb{E}\left[\left(\underset{f \in \mathcal{S}_{m \vee m^{\prime}}, ~ \|f\|_{\pi}^2 = 1}{\sup}{\nu_{n,1}^2(f)} - Q(m, m^{\prime})\right)_{+}\right] \leq \dfrac{C}{n}, ~~~~ m \in \mathcal{M}_n$$
    where, for all $m,m^{\prime} \in \mathcal{M}_n, ~ Q(m,m^{\prime}) = \mathfrak{c}\dfrac{(m \vee m^{\prime})\left\|\Psi_{m \vee m^{\prime} - 1}^{-1}\right\|_{\mathrm{op}}}{n}$ and $\mathfrak{c} > 0$ is a constant.
\end{lemma}
Under Assumption~\ref{ass:limit-case-elliptic} and from Equations~\eqref{eq:Gram1.2} and \eqref{eq:Gram2-2} in the proof of Proposition~\ref{prop:operator-norm-identity}, the penalty function is derived from $Q(m,m^{\prime})$ by $\mathrm{pen}(m) = \mathrm{pen}_2(m) = \tau_2\frac{m^{d_0}}{n}$ where $\tau_2 > 0$ is a constant. In addition, from Lemma~\ref{lm:omega-comp} and the assumptions therein, we have $\sqrt{\mathbb{E}[X_0^4]\mathbb{P}\left(\Omega_{m_{\max}}^c\right)} = \mathcal{O}(n^{-1})$ and since $s_0 > 2$,
$$ m_{\max} L \mathbb{P}\left(\Omega_{\max}^c\right) \leq C\dfrac{(n\Delta_n)^{2/(2s_0+1)}\log(n)}{n^{2}} = \dfrac{C\log(n)}{n^{(4s_0+1)/(2s_0+1)}} = o\left(\dfrac{\log(n)}{n^{3/2}}\right) ~~ \mathrm{as} ~~ n \rightarrow \infty.$$
Then, from Lemma~\ref{lm:lemma-polynomial} and Equation~\eqref{eq:adaptive-1} there exists a constant $C>0$ such that
$$\mathbb{E}\left[\left\|\w{\sigma}_{\w{m}}^2 - \sigma^2\right\|_n^2\right] \leq 3\underset{m \in \mathcal{M}_n}{\inf}\left\{\underset{f \in \mathcal{S}_{m,L}}{\inf}{\left\|f - \sigma^2\right\|_{\pi}^2} + \mathrm{pen}_2(m)\right\} + Cn^{-1}.$$
Finally, from Equation~\eqref{eq:pi-norm-risk} in the proof of Corollary~\ref{cor:upper-limit-NAE}, there exists a constant $C>0$ such that
$$\mathbb{E}\left[\left\|\w{\sigma}_{\w{m}}^2 - \sigma^2\right\|_{\pi}^2\right] \leq 9\underset{m \in \mathcal{M}_n}{\inf}\left\{\underset{f \in \mathcal{S}_{m,L}}{\inf}{\left\|f - \sigma^2\right\|_{\pi}^2} + \mathrm{pen}_2(m)\right\} + Cn^{-1}.$$
\end{proof}

\begin{proof}[\textbf{Proof of Lemma~\ref{lm:lemma-exponential}}]
Recall that for all $f \in \mathcal{S}_{m \vee m^{\prime}}$ with $m,m^{\prime} \in \mathcal{M}_n$, we have 
$$\nu_{n,1}(f) = \dfrac{1}{n}\sum_{k=0}^{n-1}\xi_{k\Delta_n}^{(1)}f(X_{k\Delta_n}),$$ where 
$$\xi_{k\Delta_n}^{(1)} = \dfrac{1}{\Delta_n}\left[\left(\int_{k\Delta_n}^{(k+1)\Delta_n}\sigma(X_s)dW_s\right)^2 - \int_{k\Delta_n}^{(k+1)\Delta_n}\sigma^2(X_s)ds\right].$$
Consider a sequence $(a_n)$ of real numbers to be chosen later. For each $k \in \{0, \ldots, n-1\}$, set
\begin{equation}\label{eq:Z-1-2}
    \begin{aligned}
        Z_{k}^{(1)} := &~ \Delta_n\xi_{k\Delta_n}^{(1)}\mathds{1}_{\left|\Delta_n\xi_{k\Delta_n}^{(1)}\right| \leq a_n} - \mathbb{E}\left[\Delta_n\xi_{k\Delta_n}^{(1)}\mathds{1}_{\left|\Delta_n\xi_{k\Delta_n}^{(1)}\right| \leq a_n} \biggm\vert \mathcal{F}_{k\Delta_n}^{X}\right]\\
        Z_{k}^{(2)} := &~ \Delta_n\xi_{k\Delta_n}^{(1)} - Z_k^{(1)} = \Delta_n\xi_{k\Delta_n}^{(1)}\mathds{1}_{\left|\Delta_n\xi_{k\Delta_n}^{(1)}\right| > a_n} - \mathbb{E}\left[\Delta_n\xi_{k\Delta_n}^{(1)}\mathds{1}_{\left|\Delta_n\xi_{k\Delta_n}^{(1)}\right| > a_n}  \biggm\vert \mathcal{F}_{k\Delta_n}^{X}\right].
    \end{aligned}
\end{equation}
Then, for all $f \in \mathcal{S}_m$, we obtain $\nu_{n,1}(f) = \Lambda_{n,1}(f) + \Lambda_{n,2}(f)$, where
\begin{equation}\label{eq:Lambda-1-2}
    \Lambda_{n,1}(f) := \dfrac{1}{n\Delta_n}\sum_{k=0}^{n-1}Z_{k}^{(1)}f(X_{k\Delta_n}) ~~ \mathrm{and} ~~ \Lambda_{n,2}(f) = \dfrac{1}{n\Delta_n}\sum_{k=0}^{n-1}Z_{k}^{(2)}f(X_{k\Delta_n}).
\end{equation}
we deduce that
\begin{equation}\label{eq:adaptive2-2}
    \begin{aligned}
        \mathbb{E}\left[\left(\underset{f \in \mathcal{S}_{m \vee m^{\prime}}, ~ \|f\|_{\pi}^2 = 1}{\sup}{\nu_{n,1}^2(f)} - Q(m, m^{\prime})\right)_{+}\right] \leq &~ 2\mathbb{E}\left[\left(\underset{f \in \mathcal{S}_{m \vee m^{\prime}}, ~ \|f\|_{\pi}^2 = 1}{\sup}{\Lambda_{n,1}^2(f)} - \dfrac{1}{2}Q(m, m^{\prime})\right)_{+}\right]\\
        & + 2\mathbb{E}\left[\underset{f \in \mathcal{S}_{m \vee m^{\prime}}, ~ \|f\|_{\pi}^2 = 1}{\sup}{\Lambda_{n,2}^2(f)}\right].
    \end{aligned}
\end{equation}
We start with the second term on the right-hand side of Equation~\eqref{eq:adaptive2-2}. Recall that for all $f = \sum_{i=0}^{m \vee m^{\prime} - 1}a_i\phi_i \in \mathcal{S}_{m \vee m^{\prime}}$ such that $\|f\|_{\pi}^2 = 1$, we have $\sum_{i=0}^{m \vee m^{\prime}-1}a_i^2 \leq \left\|\Psi_{m \vee m^{\prime}}^{-1}\right\|_{\mathrm{op}}$. We deduce that
\begin{equation*}
    \begin{aligned}
        \Lambda_{n,2}^2(f) = \left(\sum_{i=0}^{m\vee m^{\prime}-1}a_i\Lambda_{n,2}(\phi_i)\right)^2 \leq \sum_{i = 0}^{m \vee m^{\prime} - 1}a_i^2\sum_{i=0}^{m \vee m^{\prime} - 1}\Lambda_{n,2}^2(\phi_i) \leq \left\|\Psi_{m \vee m^{\prime}}^{-1}\right\|_{\mathrm{op}}\sum_{i=0}^{m \vee m^{\prime} - 1}\Lambda_{n,2}^2(\phi_i).
    \end{aligned}
\end{equation*}
It follows that
\begin{multline*}
     \mathbb{E}\left[\underset{f \in \mathcal{S}_{m \vee m^{\prime}}, ~ \|f\|_{\pi}^2 = 1}{\sup}{\Lambda_{n,2}^2(f)}\right] \leq \left\|\Psi_{m \vee m^{\prime}}^{-1}\right\|_{\mathrm{op}}\sum_{i=0}^{m \vee m^{\prime} - 1}\mathbb{E}\left[\Lambda_{n,2}^2(\phi_i)\right]\\
     = \dfrac{\left\|\Psi_{m \vee m^{\prime}}^{-1}\right\|_{\mathrm{op}}}{(n\Delta_n)^2}\sum_{i=0}^{m \vee m^{\prime} - 1}\sum_{k=0}^{n-1}\mathbb{E}\left[\left(Z_k^{(2)}\right)^2\phi_i^{2}(X_{k\Delta_n})\right]\\
     + \dfrac{2\left\|\Psi_{m \vee m^{\prime}}^{-1}\right\|_{\mathrm{op}}}{(n\Delta_n)^2}\sum_{i=0}^{m \vee m^{\prime} - 1}\sum_{k > k^{\prime}}\mathbb{E}\left[Z_k^{(2)}\phi_i(X_{k\Delta_n})Z_{k^{\prime}}^{(2)}\phi_i(X_{k^{\prime}\Delta_n})\right].
\end{multline*}
For all $k,k^{\prime} \in \{0, \ldots, n-1\}$ such that $k \neq k^{\prime}$, we have
\begin{equation*}
    \begin{aligned}
        \mathbb{E}\left[Z_k^{(2)}\phi_i(X_{k\Delta_n})Z_{k^{\prime}}^{(2)}\phi_i(X_{k^{\prime}\Delta_n}) \biggm\vert \mathcal{F}_{k\Delta_n}^{X}\right] = \phi_i(X_{(k \land k^{\prime})\Delta_n})Z_{(k \land k^{\prime})}^{(2)}\phi_i(X_{(k \vee k^{\prime})\Delta_n})\mathbb{E}\left[Z_{k \vee k^{\prime}}^{(2)}\biggm\vert \mathcal{F}_{(k \vee k^{\prime})\Delta_n}^{X}\right] = 0.
    \end{aligned}
\end{equation*}
We deduce that
\begin{equation*}
    \begin{aligned}
        \mathbb{E}\left[\underset{f \in \mathcal{S}_{m \vee m^{\prime}}, ~ \|f\|_{\pi}^2 = 1}{\sup}{\Lambda_{n,2}^2(f)}\right] \leq &~ \dfrac{\left\|\Psi_{m \vee m^{\prime}}^{-1}\right\|_{\mathrm{op}}}{(n\Delta_n)^2}\sum_{i=0}^{m \vee m^{\prime} - 1}\left\|\phi_i\right\|_{\infty}^2\sum_{k=0}^{n-1}\mathbb{E}\left[\left(Z_k^{(2)}\right)^2\right]\\
        = &~ \dfrac{\left\|\Psi_{m \vee m^{\prime}}^{-1}\right\|_{\mathrm{op}}}{(n\Delta_n)^2}\sum_{i=0}^{m \vee m^{\prime} - 1}\left\|\phi_i\right\|_{\infty}^2\sum_{k=0}^{n-1}\mathbb{E}\left[\mathbb{E}\left\{\left|\Delta_n\xi_{k\Delta_n}^{(1)}\right|^2\mathds{1}_{\left|\Delta_n\xi_{k\Delta_n}^{(1)}\right| > a_n} \biggm\vert \mathcal{F}_{k\Delta_n}^{X}\right\}\right]\\
        = &~ \dfrac{\left\|\Psi_{m \vee m^{\prime}}^{-1}\right\|_{\mathrm{op}}}{(n\Delta_n)^2}\sum_{i=0}^{m \vee m^{\prime} - 1}\left\|\phi_i\right\|_{\infty}^2\sum_{k=0}^{n-1}\mathbb{E}\left[\left|\Delta_n\xi_{k\Delta_n}^{(1)}\right|^2\mathds{1}_{\left|\Delta_n\xi_{k\Delta_n}^{(1)}\right| > a_n}\right].
    \end{aligned}
\end{equation*}
Using Cauchy-Schwarz's inequality, we obtain
\begin{equation*}
    \begin{aligned}
        \mathbb{E}\left[\underset{f \in \mathcal{S}_{m \vee m^{\prime}}, ~ \|f\|_{\pi}^2 = 1}{\sup}{\Lambda_{n,2}^2(f)}\right] \leq &~ \dfrac{\left\|\Psi_{m \vee m^{\prime}}^{-1}\right\|_{\mathrm{op}}}{(n\Delta_n)^2}\sum_{i=0}^{m \vee m^{\prime} - 1}\left\|\phi_i\right\|_{\infty}^2\sum_{k=0}^{n-1}\left(\mathbb{E}\left[\left|\Delta_n\xi_{k\Delta_n}^{(1)}\right|^4\right]\right)^{1/2}\sqrt{\mathbb{P}\left(\left|\Delta_n\xi_{k\Delta_n}^{(1)}\right| > a_n\right)}.
    \end{aligned}
\end{equation*}
On the one hand, using Burkholder-Davis-Gundy's inequality combined with H\"older's inequality, for all $q \geq 1$, there exists a constant $C_q>0$ such that
$$\mathbb{E}\left[\left|\Delta_n\xi_{k\Delta_n}^{(1)}\right|^q\right] \leq C_q\Delta_n^{q-1}\int_{k\Delta_n}^{(k+1)\Delta_n}\mathbb{E}\left[\sigma^{2q}(X_s)\right]ds \leq C_q\Delta_n^q\mathbb{E}\left[\sigma^{2q}(X_0)\right].$$
On the other hand, using Markov's inequality, for all $q \geq 1$, there exists a constant $C_p>0$ depending on $p$ such that 
$$ \mathbb{P}\left(\left|\Delta_n\xi_{k\Delta_n}^{(1)}\right| > a_n\right) \leq \dfrac{1}{a_n^q}\mathbb{E}\left[\left|\Delta_n\xi_{k\Delta_n}^{(1)}\right|^q\right] \leq C_p\dfrac{\Delta_n^q}{a_n^q}\mathbb{E}\left[\sigma^{2q}(X_0)\right].$$
Then, there exists a constant $C>0$ such that for all $m, m^{\prime} \in \mathcal{M}_n$,
$$\mathbb{E}\left[\underset{f \in \mathcal{S}_{m \vee m^{\prime}}, ~ \|f\|_{\pi}^2 = 1}{\sup}{\Lambda_{n,2}^2(f)}\right]  \leq \dfrac{C\Delta_n^q}{na_n^q}\left\|\Psi_{m \vee m^{\prime}}^{-1}\right\|_{\mathrm{op}}\sqrt{\mathbb{E}\left[\sigma^8(X_0)\right]\mathbb{E}\left[\sigma^{2q}(X_0)\right]}\sum_{i=0}^{m \vee m^{\prime} - 1}\left\|\phi_i\right\|_{\infty}^2.$$
For $q = 8$ and by Proposition~\ref{prop:Exp-Holder} with $p = 8 < d/D$, there exists a constant $C>0$such that
\begin{equation}\label{eq:adaptive2-3}
    \begin{aligned}
        \mathbb{E}\left[\underset{f \in \mathcal{S}_{m \vee m^{\prime}}, ~ \|f\|_{\pi}^2 = 1}{\sup}{\Lambda_{n,2}^2(f)}\right]  \leq \dfrac{C\Delta_n^8}{na_n^8}(m \vee m^{\prime})\left\|\Psi_{m \vee m^{\prime}}^{-1}\right\|_{\mathrm{op}}.
    \end{aligned}
\end{equation}
For the first term on the right-hand side of Equation~\eqref{eq:adaptive2-2}, we apply the Talagrand inequality (see \cite{comte2021drift}, \textit{Theorem A.2}). For this purpose, we assume for simplicity that $n = 2p_nq_n$, where $p_n \geq 1$ and $q_n \geq 1$ are integers such that $p_n, q_n \rightarrow \infty$ as $n \rightarrow \infty$. Then, using Berbee's coupling lemma and the method used in \cite{comte2021drift} and references therein, there exist random couples $\left(X_{k\Delta_n}^*, \xi_{k\Delta_n}^{(1)*}\right), ~ k = 0, \ldots, n-1$ such that:
\begin{itemize}
    \item [(i)] For each $\delta \in \{0,1\}$ and for $i = 0, \ldots, p_n - 1$, the random vectors $\vec{Y}_{i,\delta} = \left(\vec{X}_{i,\delta}, \vec{\xi}_{i,\delta}^{(1)}\right)$ and $\vec{Y}_{i,\delta}^* = \left(\vec{X}_{i,\delta}^*, \vec{\xi}_{i,\delta}^{(1)*}\right)$ have the same probability distribution, where 
    \begin{equation*}
        \begin{aligned}
            \vec{X}_{i,\delta} := &~ \left(X_{(2i + \delta)q_n\Delta_n}, \ldots, X_{[(2i + \delta)q_n +q_n-1]\Delta_n}\right)^{\prime}, ~~ \vec{X}_{i,\delta}^* := \left(X_{(2i + \delta)q_n\Delta_n}^*, \ldots, X_{[(2i + \delta)q_n + q_n - 1]\Delta_n}^*\right)^{\prime}\\
            \vec{\xi}_{i,\delta}^{(1)} = &~ \left(\xi_{(2i + \delta)q_n\Delta_n}^{(1)}, \ldots, \xi_{[(2i + \delta)q_n + q_n - 1]\Delta_n}^{(1)}\right)^{\prime}, ~~ \vec{\xi}_{i,\delta}^{(1)*} = \left(\xi_{(2i + \delta)q_n\Delta_n}^{(1)*}, \ldots, \xi_{[(2i + \delta)q_n + q_n - 1]\Delta_n}^{(1)*}\right)^{\prime}.
        \end{aligned}
    \end{equation*}
    \item [(ii)] For each $\delta \in \{0,1\}$ and for $i = 0, \ldots, p_n - 1, ~ \mathbb{P}\left(\vec{Y}_{i,\delta} \neq \vec{Y}_{i,\delta}^*\right) \leq \beta_X(q_n\Delta_n) \leq \exp(-\gamma q_n\Delta_n)$.
    \item [(iii)] For each $\delta \in \{0,1\}$, the random vectors $\vec{Y}_{0,\delta}^*, \ldots, \vec{Y}_{p_n-1,\delta}^*$ are independent.
\end{itemize}
Set $\Omega^* := \left\{\left(X_{k\Delta_n}, \xi_{k\Delta_n}^{(1)}\right) = \left(X_{k\Delta_n}^*, \xi_{k\Delta_n}^{(1)*}\right), ~~ k = 0, \ldots, n-1\right\}$. We obtain $\mathbb{P}\left((\Omega^{*})^c\right) \leq 2p_n\beta_X(q_n\Delta_n)$. Then, from observations $\left\{\vec{Y}_{i,\delta}^*, ~ i= 0, \ldots, p_n - 1, ~ \delta = 0, 1\right\}$, we define $\Lambda_{n,1}^*(f) := \sum_{\delta \in \{0,1\}}\Lambda_{n,1}^{*,\delta}(f)$, where for each $\delta \in \{0,1\}$ and for all $f \in \mathcal{F}$, 
$\Lambda_{n,1}^{*,\delta}(f) = (1/p_n)\sum_{i=0}^{p_n-1}\Gamma_{n,f}(\vec{Y}_{i,\delta}^*)$, where
$$\Gamma_{n,f}(\vec{Y}_{i,\delta}^*) = \dfrac{1}{2q_n\Delta_n}\sum_{j=0}^{q_n-1}Z_{(2i+\delta)q_n + j}^{(1)*}f(X_{[(2i+\delta)q_n + j]\Delta_n}^*),$$
with $Z_k^{(1)*} = \Delta_n\xi_{k\Delta}^{(1)*}\mathds{1}_{\left|\Delta_n\xi_{k\Delta_n}^{(1)*}\right| \leq a_n} - \mathbb{E}\left[\Delta_n\xi_{k\Delta}^{(1)*}\mathds{1}_{\left|\Delta_n\xi_{k\Delta_n}^{(1)*}\right| \leq a_n}\right]$ for each $k \in \{0, \ldots, n-1\}$. Moreover, since $\mathbb{E}[Z_{(2i+\delta)q_n + j}^{(1)*}] = 0$, we deduce that for any $i \in \{0, \ldots, p_n-1\}$, $\mathbb{E}[\Gamma_{n,f}(\vec{Y}_{i,\delta}^*)] = 0$.\\
We now verify the conditions for applying Talagrand's inequality. Fix $m,m^{\prime} \in \mathcal{M}_n$ and set $$\mathcal{F} = \left\{f \in \mathcal{S}_{m \vee m^{\prime}}, ~~ \|f\|_{\pi}^2 = 1\right\}.$$ 
First, there exists a constant $C>0$ such that 
$$\underset{f \in \mathcal{F}}{\sup}{\left\|\Gamma_{n,f}\right\|_{\infty}} \leq a_n\underset{f \in \mathcal{F}}{\sup}\left\|f\right\|_{\infty} \leq Ca_n\left((m \vee m^{\prime})\left\|\Psi_{m \vee m^{\prime}}^{-1}\right\|_{\mathrm{op}}\right)^{1/2} := M.$$
Second, following the proof of Theorem~\ref{thm:upper-limit-NAE}, for all $f \in \mathcal{F}$ and $\delta \in \{0,1\}$,
\begin{equation*}
    \begin{aligned}
        \left(\Lambda_{n,1}^{*,\delta}\right)^2(f) = &~ \dfrac{1}{4p_n^2q_n^2\Delta_n^2}\left(\sum_{\ell = 0}^{m \vee m^{\prime} - 1}a_{\ell}\sum_{i=0}^{p_n-1}\sum_{j=0}^{q_n-1}Z_{(2i+\delta)q_n+j}^{(1)*}\phi_{\ell}(X_{[(2i+\delta)q_n + j]\Delta_n})\right)^2\\
        \leq &~ \dfrac{\left\|\Psi_{m \vee m^{\prime}}^{-1}\right\|_{\mathrm{op}}}{4p_n^2q_n^2\Delta_n^2}\sum_{\ell = 0}^{m \vee m^{\prime} - 1}\left(\sum_{i=0}^{p_n-1}\sum_{j=0}^{q_n-1}Z_{(2i+\delta)q_n+j}^{(1)*}\phi_{\ell}(X_{[(2i+\delta)q_n + j]\Delta_n})\right)^2.
    \end{aligned}
\end{equation*}
It follows, by Proposition~\ref{prop:Exp-Holder} with $p=2<d/D$, that
\begin{equation}\label{eq:adaptive-5}
    \begin{aligned}
        \mathbb{E}\left[\underset{f \in \mathcal{F}}{\sup}{\left(\Lambda_{n,1}^{*,\delta}\right)^2(f)}\right] \leq &~ \dfrac{\left\|\Psi_{m \vee m^{\prime}}^{-1}\right\|_{\mathrm{op}}}{4p_n^2q_n^2\Delta_n^2}\sum_{\ell = 0}^{m \vee m^{\prime} - 1}\sum_{i=0}^{p_n-1}\sum_{j=0}^{q_n-1}\mathbb{E}\left[\left(Z_{(2i+\delta)q_n+j}^{(1)*}\right)^2\phi_{\ell}^2(X_{[(2i+\delta)q_n + j]\Delta_n})\right]\\
        \leq &~ \dfrac{\left\|\Psi_{m \vee m^{\prime}}^{-1}\right\|_{\mathrm{op}}}{4p_n^2q_n^2\Delta_n^2}\sum_{\ell = 0}^{m \vee m^{\prime} - 1}\left\|\phi_i\right\|_{\infty}^2\sum_{i=0}^{p_n-1}\sum_{j=0}^{q_n-1}\mathbb{E}\left[\left|\Delta_n\xi_{[(2i+\delta)q_n+j]\Delta_n}\right|^2\right]\\
         \leq &~ C\mathbb{E}\left[\sigma^4(X_{0})\right]\dfrac{(m \vee m^{\prime})\left\|\Psi_{m \vee m^{\prime} - 1}^{-1}\right\|_{\mathrm{op}}}{n} \leq \mathfrak{c}\dfrac{(m \vee m^{\prime})\left\|\Psi_{m \vee m^{\prime} - 1}^{-1}\right\|_{\mathrm{op}}}{n}:= H^2, 
    \end{aligned}
\end{equation}
where $\mathfrak{c}>0$ is a constant depending on $\sigma$. Third, there exists a constant $C>0$ such that
\begin{equation*}
    \begin{aligned}
        \underset{f \in \mathcal{F}}{\sup}\dfrac{1}{p_n}\sum_{i=0}^{p_n-1}\mathrm{Var}\left[\Gamma_{n,f}(\vec{Y}_{i,\delta}^*)\right] = &~ \underset{f \in \mathcal{F}}{\sup}\dfrac{1}{4p_nq_n^2\Delta_n^2}\sum_{i=0}^{p_n-1}\sum_{j=0}^{q_n-1}\mathbb{E}\left[\left(Z_{(2i+\delta)q_n+j}^{(1)*}\right)^2f^2(X_{[(2i+\delta)q_n+j]\Delta_n}^*)\right]\\
        \leq &~ C\left(\mathbb{E}\left[\sigma^8(X_0)\right]\right)^{1/2}\underset{f \in \mathcal{F}}{\sup}\left\{\left(\mathbb{E}\left[f^2(X_0)\right]\right)^{1/2}\dfrac{\|f\|_{\infty}}{2q_n}\right\}\\
        \leq &~ C\dfrac{\sqrt{(m \vee m^{\prime})\left\|\Psi_{m \vee m^{\prime}}^{-1}\right\|_{\mathrm{op}}}}{q_n} := v,
    \end{aligned}
\end{equation*}
where $C > 0$ is a new constant. Thus, applying the Talaprand inequality (see Theorem A.2 in \cite{comte2021drift}), for all $r > 0$ with $C(r) = \left(\sqrt{1+r} - 1\right) \wedge 1$ and $b = 1/6$,
\begin{equation}\label{eq:talagrand1}
    \begin{aligned}
        \mathbb{E}\left[\left(\left\{\underset{f \in \mathcal{F}}{\sup}{\left|\Lambda_{n,1}^{*,\delta}(f)\right|^2} - 2(1+2r)H^2\right\}\mathds{1}_{\Omega^*}\right)_{+}\right] \leq &~ \dfrac{4}{b}\left(\dfrac{v}{p_n}\mathrm{e}^{-br \frac{p_nH^2}{v}} + \dfrac{49M^2}{bC^{2}(r)p_n^2}\mathrm{e}^{-\frac{\sqrt{2}b\mathfrak{C}(r)\sqrt{r}}{7}\frac{p_nH}{M}}\right).
    \end{aligned}
\end{equation}
Since $(m \vee m^{\prime})\left\|\Psi_{m \vee m^{\prime}}^{-1}\right\|_{\mathrm{op}} \rightarrow \infty$ as $m,m^{\prime}\rightarrow \infty$, there exists $m_0 \geq 1$ such that for all $m \vee m^{\prime} \geq m_0$,
\begin{equation}\label{eq:talagrand2}
    \begin{aligned}
        \dfrac{v}{p_n}\mathrm{e}^{-br\frac{p_nH^2}{v}} = &~ \mathcal{O}\left(\dfrac{\sqrt{(m \vee m^{\prime})\left\|\Psi_{m \vee m^{\prime}}^{-1}\right\|_{\mathrm{op}}}}{n}\exp\left(-br\sqrt{(m \vee m^{\prime})\left\|\Psi_{m \vee m^{\prime}}^{-1}\right\|_{\mathrm{op}}}\right)\right)\\
        = &~ \mathcal{O}\left(\dfrac{1}{n}\exp\left(-\dfrac{br}{2}\sqrt{(m \vee m^{\prime})\left\|\Psi_{m \vee m^{\prime}}^{-1}\right\|_{\mathrm{op}}}\right)\right) = \mathcal{O}\left(\dfrac{1}{n}\exp\left(-\dfrac{br}{2}\sqrt{m^{\prime}}\right)\right).
    \end{aligned}
\end{equation}
On the other hand, for $p_n = n/(2q_n) = \lfloor \sqrt{n} \rfloor$ and $a_n = n^{1/4}\Delta_n = \varepsilon_0\log^2(n)n^{-1/4}$, we have the following:
\begin{equation}\label{eq:talagrand3}
    \begin{aligned}
        \dfrac{49M^2}{bC^2(r)p_n^2}\mathrm{e}^{-\frac{\sqrt{2}bC(r)\sqrt{r}}{7}\frac{p_nH}{M}} = &~ \mathcal{O}\left(\dfrac{(m \vee m^{\prime})\left\|\Psi_{m \vee m^{\prime}}^{-1}\right\|_{\mathrm{op}}\log^2(n)}{n^{3/2}}\exp\left(-c\dfrac{p_n}{\sqrt{n}a_n}\right)\right)\\
        = &~ \mathcal{O}\left(\dfrac{\log^4(n)}{n^{3/2}}\exp\left(-c\dfrac{n^{1/4}}{\log^2(n)}\right)\right),
    \end{aligned}
\end{equation}
where $c>0$ is a constant depending on $r$ and  $\varepsilon_0$. We obtain from Equations~\eqref{eq:talagrand3}, \eqref{eq:talagrand2} and \eqref{eq:talagrand1} that 
$$\mathbb{E}\left[\left(\left\{\underset{f \in \mathcal{F}}{\sup}{\left|\Lambda_{n,1}^{*,\delta}(f)\right|^2} - 2(1+2r)H^2\right\}\mathds{1}_{\Omega^*}\right)_{+}\right] \leq C\left[\dfrac{1}{n}\exp\left(-\dfrac{br}{2}\sqrt{m^{\prime}}\right) + \dfrac{\log^4(n)}{n^{3/2}}\exp\left(-c\dfrac{n^{1/4}}{\log^2(n)}\right)\right].$$
For $n\Delta_n^2 = \varepsilon_0\log^{4}(n)$, we obtain $\mathbb{P}(\Omega^{*c}) \leq 2p_n\beta_X(q_n\Delta_n) = \mathcal{O}\left(n^{1/2}\exp\left(-0.5\varepsilon_0\gamma\log^2(n)\right)\right)$, and we obtain
\begin{equation*}
    \begin{aligned}
       &\sum_{m^{\prime} \in \mathcal{M}_n}\mathbb{E}\left[\left(\underset{f \in \mathcal{S}_{m \vee m^{\prime}}, ~ \|f\|_{\pi}^2 = 1}{\sup}{\Lambda_{n,1}^2(f)} - 2(1+2r)H^2\right)_{+}\right]\\
       & = \mathcal{O}\left(\dfrac{1}{n}\sum_{m^{\prime} \in \mathcal{M}_n}\exp\left(-\dfrac{br}{2}\sqrt{m^{\prime}}\right)  + m_{\max}\left(M^2+H^2\right)\mathbb{P}\left(\Omega^{* c}\right) + \dfrac{m_{\max}\log^4(n)}{n^{3/2}}\exp\left(-c\dfrac{n^{1/4}}{\log^2(n)}\right)\right)\\
       & =  \mathcal{O}\left(n^{-1}\right).
    \end{aligned}
\end{equation*}
We set $Q(m,m^{\prime}) = 2(1+2r)H^2$, from the above equation together with Equations~\eqref{eq:adaptive2-3} and \eqref{eq:adaptive2-2}, there exists a constant $C>0$ such that for $m \in \mathcal{M}_n = \{1, \ldots, m_{\max}\}$ with $m_{\max} = \left\lceil (n\Delta_n)^{2/(2s_0+1)}\right\rceil, ~ s_0 = D + d + 2$, and for $a_n = n^{1/4}\Delta_n$,
\begin{equation*}
    \begin{aligned}
        \sum_{m^{\prime} \in \mathcal{M}_n}\mathbb{E}\left[\left(\underset{f \in \mathcal{S}_{m \vee m^{\prime}}, ~ \|f\|_{\pi}^2 = 1}{\sup}{\nu_{n,1}^2(f)} - Q(m, m^{\prime})\right)_{+}\right] \leq &~ C\left(\dfrac{1}{n} + \dfrac{\Delta_n^8}{na_n^8}(m \vee m^{\prime})\left\|\Psi_{m \vee m^{\prime}}^{-1}\right\|_{\mathrm{op}}\right) \leq Cn^{-1}.
    \end{aligned}
\end{equation*}
\end{proof}
\begin{proof}[\textbf{Proof of Lemma~\ref{lm:lemma-polynomial}}]
    The proof of this lemma is identical to the one given for Lemma~\ref{lm:lemma-exponential}. We refer to it for clarity, thereby avoiding repetition. Then, we have the following:
\begin{multline}\label{eq:adaptive2}
    \mathbb{E}\left[\left(\underset{f \in \mathcal{S}_{m \vee m^{\prime}}, ~ \|f\|_{\pi}^2 = 1}{\sup}{\nu_{n,1}^2(f)} - Q(m, m^{\prime})\right)_{+}\right]\\
    \leq 2\mathbb{E}\left[\left(\underset{f \in \mathcal{S}_{m \vee m^{\prime}}, ~ \|f\|_{\pi}^2 = 1}{\sup}{\Lambda_{n,1}^2(f)} - \dfrac{1}{2}Q(m, m^{\prime})\right)_{+}\right] + 2\mathbb{E}\left[\underset{f \in \mathcal{S}_{m \vee m^{\prime}}, ~ \|f\|_{\pi}^2 = 1}{\sup}{\Lambda_{n,2}^2(f)}\right],
\end{multline}
where for each $f \in \mathcal{S}_{m \vee m^{\prime}}$,  $\Lambda_{n,1}(f)$ and $\Lambda_{n,2}(f)$ are given in Equation~\eqref{eq:Lambda-1-2}. For $a_n = n^{1/4}\Delta_n$ and by Equation~\eqref{eq:adaptive2-3},
\begin{equation}\label{eq:adaptive3}
    \begin{aligned}
        \mathbb{E}\left[\underset{f \in \mathcal{S}_{m \vee m^{\prime}}, ~ \|f\|_{\pi}^2 = 1}{\sup}{\Lambda_{n,2}^2(f)}\right]  \leq \dfrac{C\Delta_n^8}{na_n^8}(m \vee m^{\prime})\left\|\Psi_{m \vee m^{\prime}}^{-1}\right\|_{\mathrm{op}} \leq C\dfrac{(m \vee m^{\prime})\left\|\Psi_{m \vee m^{\prime}}^{-1}\right\|_{\mathrm{op}}}{n^3}.
    \end{aligned}
\end{equation}
Moreover, we set $Q(m,m^{\prime}) = 4(1+2r)H^2, ~~ m,m^{\prime} \in \mathcal{M}_n$, there exists a constant $C>0$ such that
\begin{multline}\label{eq:adaptive4}
    \mathbb{E}\left[\left(\underset{f \in \mathcal{S}_{m \vee m^{\prime}}, ~ \|f\|_{\pi}^2 = 1}{\sup}{\Lambda_{n,1}^2(f)} - \dfrac{1}{2}Q(m, m^{\prime})\right)_{+}\right] \\
    \leq \sum_{\delta \in \{0,1\}}\mathbb{E}\left[\left(\left\{\underset{f \in \mathcal{S}_{m \vee m^{\prime}}, ~ \|f\|_{\pi}^2 = 1}{\sup}{\left|\Lambda_{n,1}^{* \delta}(f)\right|^2} - 2(1+2r)H^2\right\}\mathds{1}_{\Omega^*}\right)_{+}\right] + C(M^2+H^2)\mathbb{P}\left(\Omega^{* c}\right).
\end{multline}
Using Talagrand's inequality on the first term on the right-hand side of Equation~\eqref{eq:adaptive4}, for $m \vee m^{\prime}$ large enough and for $m \in \mathcal{M}_n$, $q_n = \left\lceil (n\Delta_n)^{5/(r_0-1)}\Delta_n^{-1}\right\rceil, ~ p_n = n/(2q_n) \geq (1/4)(n\Delta_n)^{1 - 5/(r_0-1)}$ with $r_0 \geq 16$, and $n\Delta_n^2 = 1$, there exist constants $C, c>0$ such that
\begin{multline}\label{eq:adaptive5}
    \sum_{\delta \in \{0,1\}}\mathbb{E}\left[\left(\left\{\underset{f \in \mathcal{S}_{m \vee m^{\prime}}, ~ \|f\|_{\pi}^2 = 1}{\sup}{\left|\Lambda_{n,1}^{* \delta}(f)\right|^2} - 2(1+2r)H^2\right\}\mathds{1}_{\Omega^*}\right)_{+}\right] \leq \dfrac{C}{n}\exp\left(-\dfrac{br}{2}\sqrt{m^{\prime}}\right)\\
    + C\dfrac{(m \vee m^{\prime})\left\|\Psi_{m \vee m^{\prime}}^{-1}\right\|_{\mathrm{op}}}{n^{\frac{3}{2} - \frac{5}{r_0-1}}}\exp\left(-cn^{\frac{1}{4} - \frac{5}{2(r_0-1)}}\right).
\end{multline}
For the second term on the right-hand side of Equation~\eqref{eq:adaptive4} and from the proof of Lemma~\ref{lm:omega-comp}, 
\begin{equation}\label{eq:adaptive6}
    (M^2+H^2)\mathbb{P}\left(\Omega^{*c}\right) \leq \dfrac{C}{n^2}\left(\dfrac{(m \vee m^{\prime})\left\|\Psi_{m \vee m^{\prime}}^{-1}\right\|_{\mathrm{op}}}{\sqrt{n}} + \dfrac{(m \vee m^{\prime})\left\|\Psi_{m \vee m^{\prime}}^{-1}\right\|_{\mathrm{op}}}{n} \right).
\end{equation}
We deduce from Equations~\eqref{eq:adaptive6}, \eqref{eq:adaptive5}, \eqref{eq:adaptive4}, \eqref{eq:adaptive3} and \eqref{eq:adaptive2}, there exists a constant $C>0$ such that
$$\sum_{m^{\prime} \in \mathcal{M}_n}\mathbb{E}\left[\left(\underset{f \in \mathcal{S}_{m \vee m^{\prime}}, ~ \|f\|_{\pi}^2 = 1}{\sup}{\nu_{n,1}^2(f)} - Q(m, m^{\prime})\right)_{+}\right] \leq Cn^{-1}.$$
\end{proof}

\subsection{Proof of Theorem~\ref{thm:adaptation2}}

\begin{proof}
    Following the proofs of Theorem~\ref{thm:upper-limit2-NAE} and \ref{thm:adaptation}, we obtain
    \begin{equation*}
    \begin{aligned}
        \mathbb{E}\left[\left\|\w{\sigma}_{\w{m}}^2 - \sigma^2\right\|_n^2\right] \leq &~ 3\left\{\underset{f \in \mathcal{S}_{m,L}}{\inf}{\left\|f - \sigma^2\right\|_{\pi}^2} + \mathrm{pen}(m)\right\} + 2m_{\max}L\mathbb{P}\left(\Omega_{m_{\max}}^c\right) + 2\sqrt{\mathbb{E}\left[X_0^4\right]\mathbb{P}\left(\Omega_{m_{\max}}^c\right)}\\
        & + 12\sum_{m^{\prime} \in \mathcal{M}_n}\mathbb{E}\left[\left(\underset{f \in \mathcal{S}_{m \vee m^{\prime}}, ~ \|f\|_{\pi}^2 = 1}{\sup}{\nu_{n}^2(f)} - Q(m, m^{\prime})\right)_{+}\right] + C \Delta_n^{\alpha},
    \end{aligned}
    \end{equation*}
    where for all $f \in \mathcal{S}_{m \vee m^{\prime}}$, $\nu_n(f) = n^{-1}\sum_{k = 0}^{n-1}\zeta_{k\Delta_n}f(X_{k\Delta_n})$ with $\zeta_{k\Delta_n} = \Delta_n^{-1}\sigma^2(X_{k\Delta_n})[(W_{(k+1)\Delta_n} - W_{k\Delta_n})^2 - \Delta_n]$. Remark that
    $$\zeta_{k\Delta_n} = \dfrac{1}{\Delta_n}\left[\left(\int_{k\Delta_n}^{(k+1)\Delta_n}\sigma(X_{\eta(s)})dW_s\right)^2 - \int_{k\Delta_n}^{(k+1)\Delta_n}\sigma^2(X_{\eta(s)})ds\right],$$
    where $\eta(s) = k\Delta_n$ for all $s \in [k\Delta_n, (k+1)\Delta_n)$. Consequently, Lemmas~\ref{lm:lemma-exponential}, \ref{lm:lemma-polynomial} and \ref{lm:omega-comp} and Proposition~\ref{prop:Exp-Holder} lead to the final results.
\end{proof}

\appendix
\section*{Appendix}

\section{Proof of Lemma~\ref{lm:omega-comp}}

\begin{proof}
Recall that for all $f = \sum_{j=0}^{m-1}a_j\phi_j = \left<\mathbf{a}, (\phi_0, \dots, \phi_{m-1})\right> \in \mathcal{S}_m$, with $\mathbf{a} = (a_0, \ldots, a_{m-1}) \in \mathbb{R}^m$ and $m \geq 1$,
\begin{equation*}
    \|f\|_n^2 = \mathbf{a}^{\prime}\widehat{\Psi}_m\mathbf{a}, ~~ \|f\|_{\pi}^2 = \mathbf{a}^{\prime}\Psi_m\mathbf{a} = \left\|\Psi_m^{1/2}\mathbf{a}\right\|_2^2 = 1 \iff \mathbf{a} = \Psi_m^{-1/2}\mathbf{u}, 
\end{equation*}
where $\mathbf{u} = (u_0, \ldots, u_{m-1}) \in \mathbb{R}^m$ and $\|\mathbf{u}\|_2 = 1$. We deduce that 
\begin{equation*}
    \underset{f \in \mathcal{S}_m, ~ \|f\|_{\pi}^2 = 1}{\sup}{\left|\|f\|_n^2 - 1\right|} = \underset{\mathbf{u} \in \mathbb{R}^m}{\sup}{\left|\mathbf{u}^{\prime}\left(\Psi_m^{-1/2}\widehat{\Psi}_m\Psi_m^{1/2} - \mathrm{Id}_m\right)\mathbf{u}\right|} = \left\|\Psi_m^{-1/2}\widehat{\Psi}_m\Psi_m^{1/2} - \mathrm{Id}_m\right\|_{\mathrm{op}}.
\end{equation*}
From \cite{comte2021drift}, \textit{Proof of Proposition 6.1}, we have
\begin{equation}\label{eq:omega-comp}
    \mathbb{P}\left(\Omega_m^c\right) \leq 4m\exp\left(-\dfrac{c(1/2)p_n}{\mathcal{L}(m)\left\|\Psi_m^{-1}\right\|_{\mathrm{op}}}\right) + 2p_n\beta_X(q_n\Delta_n),
\end{equation}
where, for simplicity, $p_n$ and $q_n$ are two integers such that $n = 2p_nq_n$, and $c(u) = u + (1-u)\log(1-u)$.\\
Under Assumptions~\ref{ass:StrongAssDrift}, the diffusion process $X$ is exponentially ergodic. Moreover, for $m = (n\Delta_n)^{2/(2s+k_0)}$ with $s \geq s_0 = D + d + 2$ and by Proposition~\ref{prop:operator-norm-identity}, there exists a constant $C>0$ such that $\mathcal{L}(m)\left\|\Psi_m^{-1}\right\|_{\mathrm{op}} \leq C(n\Delta_n)^{\beta(s)}$, where $\beta(s) < 1$ (see also Equation~\eqref{eq:operator-norm-identity}). Setting $p_n = \sqrt{n}, ~ q_n = \sqrt{n}/2$ and given that $n\Delta_n^2 = \varepsilon_0\log^{4}(n)$ with $\varepsilon_0 \in (0,10^{-1})$, we obtain
$$ \dfrac{p_n}{\mathcal{L}(m)\left\|\Psi_m^{-1}\right\|_{\mathrm{op}}} \geq \dfrac{\sqrt{n}}{Cn^{\beta(s)/2}(n\Delta_n^2)^{\beta(s)/2}} = \dfrac{n^{(1-\beta(s))/2}}{C\varepsilon_0^{\beta(s)/2}\log^{2\beta(s)}(n)} \rightarrow \infty ~~ \mathrm{as} ~~ n \rightarrow \infty.$$
On the other hand, $q_n\Delta_n = (n\Delta_n^2)^{1/2}/2 = \varepsilon_0^{1/2}\log^2(n)/2$. From Equation~\eqref{eq:omega-comp}, we obtain the following:
\begin{equation*}
    \begin{aligned}
        \mathbb{P}(\Omega_m^c) = &~  \mathcal{O}\left((n\Delta_n)^{\frac{2}{2s+k_0}}\exp\left(-\dfrac{c(1/2)n^{(1-\beta(s))/2}}{C\sqrt{\varepsilon_0^{\beta(s)}}\log^{2\beta(s)}(n)}\right) + \sqrt{n}\exp\left(-\dfrac{\gamma}{2}\sqrt{\varepsilon_0}\log^2(n)\right)\right)\\
        = &~ \mathcal{O}\left(\exp\left(-\dfrac{\gamma}{4}\sqrt{\varepsilon_0}\log^2(n)\right)\right),
    \end{aligned}
\end{equation*}
Under Assumptions~\ref{ass:limit-case-elliptic}, the process $X$ is polynomially ergodic and for all $t \geq 0, ~ \beta_X(t) \leq C(1+t)^{1 - r_0}$ for $k = r_0 - 2$, where $C>0$ is a constant and $r_0 \geq 16$ is defined in Assumption~\ref{ass:limit-case-elliptic}. We set $q_n = \left\lceil (n\Delta_n)^{5/(r_0 - 1)}\Delta_n^{-1} \right\rceil$, then 
$p_n = \frac{n}{2q_n} \leq \frac{n\Delta_n}{2(n\Delta)^{5/(r_0-1)}} = \frac{1}{2}(n\Delta_n)^{(r_0-6)/(r_0-1)}$. Since $n\Delta_n^2 = 1$, we have $n\Delta_n = n^{1/2}$. It follows that
\begin{equation}\label{eq:omega-comp2}
    2p_n\beta_X(q_n\Delta_n) = \dfrac{C(n\Delta_n)^{\frac{r_0-6}{r_0-1}}}{(1+q_n\Delta_n)^{r_0 - 1}} \leq \dfrac{C(n\Delta_n)^{\frac{r_0-6}{r_0-1}}}{(n\Delta_n)^{5}} = \dfrac{C}{(n\Delta_n)^{\frac{4r_0+1}{r_0-1}}} < \dfrac{C}{(n\Delta_n)^{4}} = \dfrac{C}{n^{2}}.
\end{equation}
In addition, for $m = (n\Delta_n)^{2/(2s+d_0)}$ with $s \geq s_0 = 2d + 3$ and $\beta(s) < 1/2$, we obtain
\begin{equation}\label{eq:omega-comp3}
    \dfrac{p_n}{\mathcal{L}(m)\left\|\Psi_m^{-1}\right\|_{\mathrm{op}}} \geq \dfrac{(n\Delta_n)^{\frac{r_0-6}{r_0-1}}}{2C(n\Delta_n)^{\beta(s)}} > \dfrac{(n\Delta_n)^{\frac{r_0-6}{r_0-1}}}{2C(n\Delta_n)^{\frac{1}{2}}} = \dfrac{1}{2C}(n\Delta_n)^{\frac{r_0-11}{2(r_0-1)}} \rightarrow \infty ~~ \mathrm{as} ~~ n \rightarrow \infty.
\end{equation}
We deduce from Equations~\eqref{eq:omega-comp3}, \eqref{eq:omega-comp2} and \eqref{eq:omega-comp} that there exists a constant $C>0$ such that $\mathbb{P}(\Omega_m^c) \leq Cn^{-2}$.    
\end{proof}

\subsection{Proof of Corollary~\ref{cor:upper-limit-NAE}}

\begin{proof}
    Grant Assumptions~\ref{ass:drift}, \ref{ass:diffusion}, \ref{ass:NonDegeneracy} and \ref{ass:stationary} for the whole proof. For $L=\log(n)$ and for all $m \in \mathcal{M}_n$, there exists a constant $C>0$ such that
\begin{equation}\label{eq:pi-norm-risk}
    \begin{aligned}
        \mathbb{E}\left[\left\|\w{\sigma}_m^2 - \sigma_I^2\right\|_{\pi}^2\right] \leq &~ \mathbb{E}\left[\left\|\w{\sigma}_m^2 - P_m(\sigma^2)\right\|_{\pi}^2\mathds{1}_{\Omega_m}\right] + \left\|P_m(\sigma^2) - \sigma_I^2\right\|_{\pi}^2 + 2mL\mathbb{P}\left(\Omega_m^c\right) + 2\sqrt{\mathbb{E}\left[\sigma^4(X_0)\right]\mathbb{P}\left(\Omega_m^{c}\right)}\\
        \leq &~ \mathbb{E}\left[\left\|\w{\sigma}_m^2 - P_m(\sigma^2)\right\|_{n}^2\right] + \left\|P_m(\sigma^2) - \sigma_I^2\right\|_{\pi}^2 + 2mL\mathbb{P}\left(\Omega_m^c\right) + 2\sqrt{\mathbb{E}\left[\sigma^4(X_0)\right]\mathbb{P}\left(\Omega_m^{c}\right)}\\
        \leq &~ 2\mathbb{E}\left[\left\|\w{\sigma}_m^2 - \sigma_I^2\right\|_{n}^2\right] + 3\underset{f \in \mathcal{S}_{m,L}}{\inf}\left\|f - \sigma_I^2\right\|_{\pi}^2 + 2mL\mathbb{P}\left(\Omega_m^c\right) + 2\sqrt{\mathbb{E}\left[\sigma^4(X_0)\right]\mathbb{P}\left(\Omega_m^{c}\right)},
    \end{aligned}
\end{equation}
where $P_m(\sigma^2)$ is the orthogonal projection of $\sigma_I^2$ onto $\mathcal{S}_m$. Then, under Assumption~\ref{ass:StrongAssDrift} and by Theorem~\ref{thm:upper-limit-NAE}, Proposition~\ref{prop:Exp-Holder} and Lemma~\ref{lm:omega-comp}, there exists a constant $C>0$ such that
\begin{equation*}
    \begin{aligned}
        \mathbb{E}\left[\left\|\w{\sigma}_m^2 - \sigma_I^2\right\|_{\pi}^2\right] \leq &~ 9\underset{f \in \mathcal{S}_{m,L}}{\inf}{\left\|f - \sigma_I^2\right\|_{\pi}^2} + C\left(\dfrac{m^{k_0}}{n} + m\log(n)\exp\left(-\frac{\gamma}{8}\log^2(n)\right) + C\Delta_n^{2}\right),
    \end{aligned}
\end{equation*}
where $k_0 = D+d+3/2$, and under Assumption~\ref{ass:limit-case-elliptic}, by Theorem~\ref{thm:upper-limit-NAE} and Lemma~\ref{lm:omega-comp}, there exists a constant $C>0$ such that
\begin{equation*}
    \begin{aligned}
        \mathbb{E}\left[\left\|\w{\sigma}_m^2 - \sigma_I^2\right\|_{\pi}^2\right] \leq &~ 9\underset{f \in \mathcal{S}_{m,L}}{\inf}{\left\|f - \sigma^2\right\|_{\pi}^2} + C\left(\dfrac{m^{d_0}}{n} + \dfrac{m\log(n)}{n^{2}} + C\Delta_n^{2}\right), ~~ d_0=d+3/2.
    \end{aligned}
\end{equation*}
\end{proof}

\section{Calibration of numerical constants of penalty terms}

Recall that the penalty terms in the risk bounds of adaptive estimators of $\sigma^2$ are $\mathrm{pen}_1(m) = \tau_1 m/n$ for polynomially ergodic diffusion processes and $\mathrm{pen}_2(m) = \tau_2 m^{k_0}/n$ for exponentially ergodic diffusion processes, where $m \in \{1, \ldots, (n\Delta_n)^{2/(2s_0+1)}\}$, $k_0 = D+d+3/2$, and $s_0 = D+d+2$ (exponential ergodicity) or $s_0 = 2d+3$ (polynomial ergodicity). The constants $\tau_1>0$ and $\tau_2>0$ depend on the unknown diffusion coefficient $\sigma$ and are therefore not computable in practice, which justifies the numerical calibration experiment. The dimension $m$ is chosen in the set $\mathcal{M} = \{1, \ldots, 10\}$ and $\tau_1, \tau_2 \in \mathfrak{C} = \{10^{-4}, 10^{-3}, 5 \times 10^{-3}, 10^{-2}, 5 \times 10^{-2}, 0.1, 0.5, 1, 2, 3, 4, 5, 10, 20, 30, 40, 50, 100\}$. Finally, we choose $n$ in the set $\mathcal{N} = \{10 000, 50 000, 100000, 200000\}$.\\
For each chosen model, for each $c \in \mathfrak{C}$ and each $n \in \mathcal{N}$, repeat $100$ times the following steps:
\begin{enumerate}
    \item [(i)] Simulate $X^{\mathrm{train}} = (X_0, \ldots, X_{n\Delta_n})$ and $X^{\mathrm{test}} = (X_0, \ldots, X_{q\Delta_{q}})$ with $q= 10000$.
    \item [(ii)] For each $m \in \mathcal{M}$, compute $\widehat{\sigma}_m^2$ from the diffusion path $X^{\mathrm{train}}$ using Equations~\eqref{eq:Coordinate-a} and \eqref{eq:Coordinate-b}.
    \item [(iii)] Select the optimal dimension $\widehat{m}$ in the set $\mathcal{M}$ using Equation~\eqref{eq:select-dimension}.
    \item [(iv)] From the path $X^{\mathrm{test}}$, compute the empirical estimation risk $\|\widehat{\sigma}_{\widehat{m}}^2 - \sigma_I^2\|_n^2$.
\end{enumerate}
The risks of estimation are then deduced from the average values of $\|\widehat{\sigma}_{\widehat{m}}^2 - \sigma_I^2\|_n^2$ on the $100$ repetitions. In the sequel, we apply the above algorithm for the chosen models.

\subsection{Case of exponentially ergodic processes}

We set $\varepsilon_0 = 0.0001$ and consider the following models:
\begin{itemize}
    \item \textbf{Model 1}: $b(x) = -9x, ~~ \sigma(x) = \sqrt{1+x^2}$, 
    \item \textbf{Model 2}: $b(x) = 5/x - 9x, ~~ \sigma(x) = \sqrt{1+x^2}$,
    \item \textbf{Model 3}: $b(x) = 2-9x, ~~ \sigma(x) = 1+|x|$
\end{itemize}
From the above models, we have $d = 9$ and $D=1$. From Figure~\ref{fig:risk_expo} we can see that the estimation risks generally reach their minimum values for $\widehat{\tau}_1 = 10^{-4}$, which is the optimal value of this numerical constant in the set $\mathfrak{C}$.  

\begin{figure}
\centering
\includegraphics[width=0.5\textwidth]{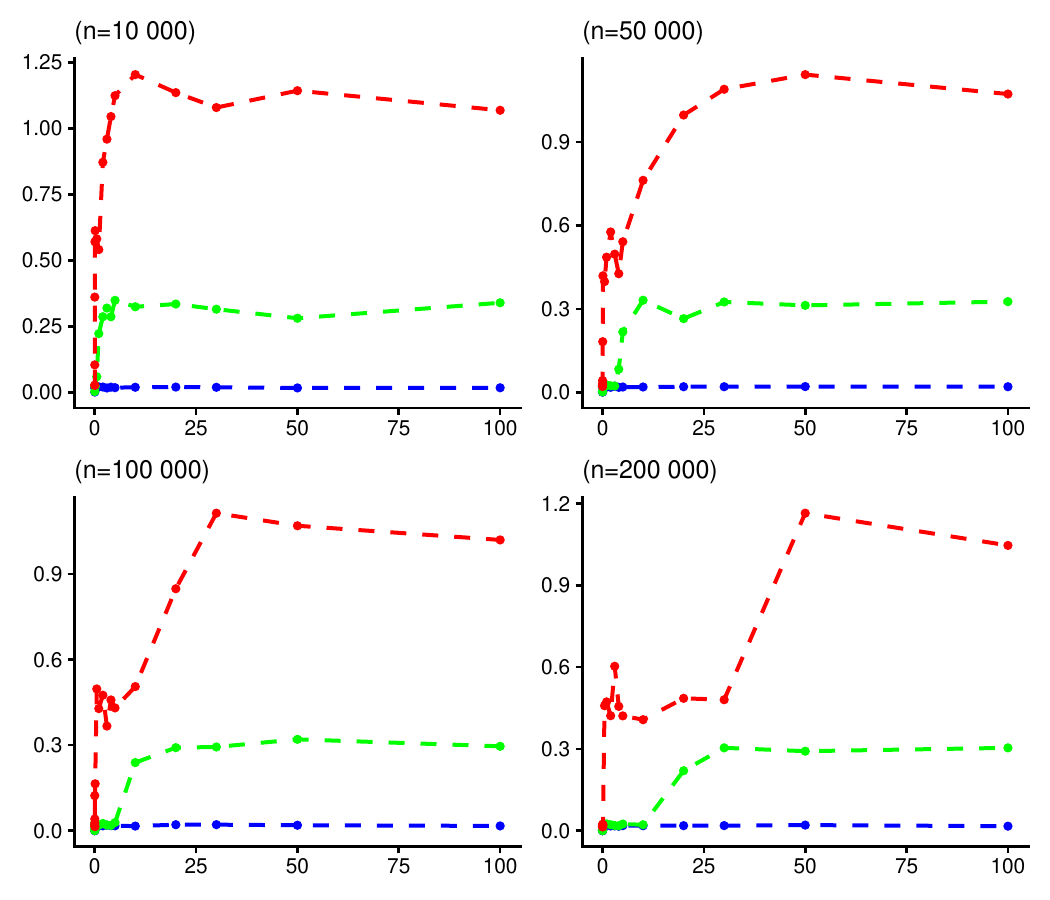}
\caption{The four graphs show the evolution of estimation risks of adaptive estimators $\widehat{\sigma}_{\widehat{m}}^2$ with respect to values of the numerical constant $\tau_1$ when $n \in \mathcal{\mathcal{N}}$.}
\label{fig:risk_expo}
\end{figure}

\subsection{Case of polynomially ergodic processes}

We consider the following models.
\begin{itemize}
    \item \textbf{Model 1}: $b(x) = -4.5x\left(3+\sqrt{1+x^2}\right)^2/9\left(1+x^2\right)^2, ~~ \sigma(x) = 1/3+1/\sqrt{1+x^2}$, 
    \item \textbf{Model 2}: $b(x) = -4.5x(2+\cos x)/(1+x^2), ~~ \sigma(x) = \sqrt{2+\cos x}$,
    \item \textbf{Model 3}: $b(x) = -5x/(3+x^2), ~~ \sigma(x) = 1$
\end{itemize}
From the above models, we have $d \in \{4.5, 5\}$. From the results obtained (see Figure~\ref{fig:risk_poly}), we have $\widehat{\tau}_2 = 100$.

\begin{figure}
\centering
\includegraphics[width=0.5\textwidth]{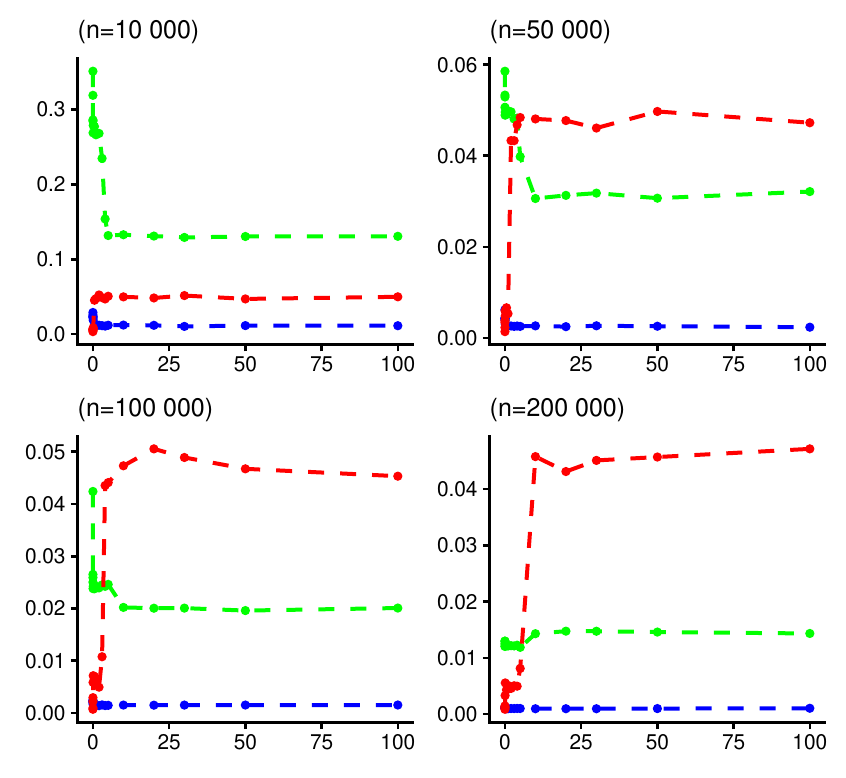}
\caption{The four graphs show the evolution of estimation risks of adaptive estimators $\widehat{\sigma}_{\widehat{m}}^2$ with respect to values of the numerical constant $\tau_2$ when $n \in \mathcal{\mathcal{N}}$.}
\label{fig:risk_poly}
\end{figure}


\section*{Declarations}

\paragraph{Conflict of interest} I have no conflict of interest to declare that is relevant to the content of this article. No funding was received to assist with the preparation of this document.

\bibliographystyle{ScandJStat}
\bibliography{mabiblio.bib}


\end{document}